\documentclass[reqno]{amsart}
\usepackage[a4paper]{geometry}
\usepackage[T2A]{fontenc}
\usepackage[utf8]{inputenc}
\usepackage{amssymb}
\usepackage[mathscr]{eucal}
\usepackage{mathtools}
\usepackage{enumerate}
\usepackage{accents}
\usepackage{comment}
\usepackage{dsfont}

\usepackage{color}
\usepackage[colorlinks=true]{hyperref}
\hypersetup{urlcolor=blue,citecolor=red,linkcolor=blue}
\usepackage[initials]{amsrefs}
\numberwithin{equation}{section}
\theoremstyle{plain}
\newtheorem{theorem}{Theorem}[section]
\newtheorem{proposition}[theorem]{Proposition}
\newtheorem{corollary}[theorem]{Corollary}
\newtheorem{lemma}[theorem]{Lemma}
\theoremstyle{definition}
\newtheorem*{rh-pb*}{Basic RH problem}
\newtheorem*{rh-I*}{RH problem}
\newtheorem*{sol-rh-pb*}{Soliton RH problem}
\newtheorem*{data*}{Data of this RH problem associated with $\BS{u_0(x)}$}
\theoremstyle{remark}
\newtheorem{remark}[theorem]{Remark}
\newtheorem{assumption}[theorem]{Assumption}
\newtheorem*{notations*}{Notations}
\newtheorem{definition}[theorem]{Definition}
\providecommand{\BS}[1]{\boldsymbol{#1}}  
\providecommand{\D}[1]{\mathbb{#1}}
\newcommand{\dd}{\mathrm{d}}
\newcommand{\eul}{\mathrm{e}}
\newcommand{\ii}{\mathrm{i}}
\newlength{\dhatheight}
\newcommand{\doublehat}[1]{%
    \settoheight{\dhatheight}{\ensuremath{\hat{#1}}}%
    \addtolength{\dhatheight}{-0.3ex}%
    \hat{\vphantom{\rule{1pt}{\dhatheight}}%
    \smash{\hat{#1}}}}
\providecommand{\abs}[1]{\lvert#1\rvert}

\renewcommand{\Im}{\operatorname{Im}}
\renewcommand{\Re}{\operatorname{Re}}
\newcommand{\ord}{\mathrm{O}}

\newif\ifshort
\shorttrue

\newif\ifnew
\newtrue
\begin{document}
\title[Riemann--Hilbert approach for the mCH equation]{The Modified Camassa--Holm Equation with Step-Like Boundary Conditions: A Characterization of Existence }
\author[I.~Karpenko]{Iryna Karpenko}
\address[I.~Karpenko]{
Universität Wien,
Oskar-Morgenstern-Platz 1
1090 Wien, Austria
\medskip
		\newline
B.~Verkin Institute for Low Temperature Physics and Engineering,
47 Nauky Avenue, 61103 Kharkiv, Ukraine
}
\email{inic.karpenko@gmail.com}
\author[D.~Shepelsky]{Dmitry Shepelsky}
\address[D.~Shepelsky]{B.~Verkin Institute for Low Temperature Physics and Engineering,
47 Nauky Avenue, 61103 Kharkiv, Ukraine
\medskip
		\newline
V.N.~Karazin Kharkiv National University, 4 Svobody Square, 61022 Kharkiv, Ukraine}
\email{shepelsky@yahoo.com}
\author[G.~Teschl]{Gerald Teschl}
\address[G.~Teschl]{Universität Wien,
Oskar-Morgenstern-Platz 1
1090 Wien, Austria}
\email{Gerald.Teschl@univie.ac.at}
\subjclass[2010]{Primary: 35Q53; Secondary: 37K15, 35Q15, 35B40, 35Q51, 37K40}
\keywords{Riemann--Hilbert problem, Camassa--Holm equation}
\date{\today}
\begin{abstract}
	We consider the Cauchy problem for 
		the modified Camassa--Holm equation
        \[
        u_t+\left((u^2-u_x^2)m\right)_x=0,\quad 
        m:= u-u_{xx},
        \quad t>0,\ \ -\infty<x<+\infty
        \]
        subject to the step-like initial data:
$u(x,0)\to A_1$ as $x\to-\infty$ and 
$u(x,0)\to A_2$ as $x\to+\infty$, 
where $0<A_1<A_2$.
       First, using the Riemann--Hilbert approach, we establish the global solvability of the corresponding problem in the reciprocal $(y,t)$-variables. Then 
       the reconstruction in the original $(x,t)$-variables is obtained under a positivity condition expressed in terms of the solution of the associated Riemann--Hilbert problem.
\end{abstract}
\maketitle

\section{Introduction}

In a recent paper \cite{KST22}, the Riemann--Hilbert (RH) formalism was developed
for studying initial value problems for the modified Camassa--Holm (mCH) equation \eqref{mCH-1}:
\begin{subequations}\label{mCH1-ic}
\begin{alignat}{4}           \label{mCH-1}
&m_t+\left((u^2-u_x^2)m\right)_x=0,&\quad&m= u-u_{xx},&\quad&t>0,&\;&-\infty<x<+\infty,\\
&u(x,0)=u_0(x),&&&&&&-\infty<x<+\infty\label{IC},
\end{alignat}
\end{subequations}
assuming that the initial data satisfy the nonzero step-like boundary conditions
\begin{equation}\label{mCH-bc}
u_0(x)\to\begin{cases}
A_1 \text{ as } x\to-\infty\\
A_2 \text{ as } x\to\infty\\
\end{cases},
\end{equation}
where $A_1$ and $A_2$ are two different constants satisfying $0<A_1<A_2$, 
and that the solution $u(x,t)$ preserves this behavior for all fixed $t>0$. 

An important additional assumption on the data for this problem adopted in \cite{KST22} is that
 $m_0(x):=(1-\partial_x^2)u_0(x)>0$ for all $x$; then it can be shown that   $m(x,t)>0$ for all $t$
 as long as the solution exists (see \cite{KST22}). 

The RH formalism developed in \cite{KST22} provides an inverse-scattering
representation for solutions of the problem
\eqref{mCH1-ic}--\eqref{mCH-bc}. More precisely, assuming \emph{a priori}
that a sufficiently regular classical solution $u(x,t)$ exists and
preserves the prescribed step-like behavior \eqref{mCH-bc} for all $t>0$, one can associate with it a
Riemann--Hilbert problem whose jump data are determined by the initial
datum $u_0$. The solution of this RH problem then yields a reconstruction
formula for $u$.

This leaves open an essential point: the RH representation by itself does
not establish the existence of a solution of the Cauchy problem. The main
purpose of the present paper is to close this gap. Starting directly from
the scattering data generated by the initial datum, and without assuming
the prior existence of a solution of the mCH equation, we prove the global
solvability of the associated RH problem, establish the required
regularity of its solution with respect to the space and time parameters,
and show that the reconstructed quantities satisfy the mCH system in
reciprocal variables. Then we discuss invertibility of the the reciprocal transformation.



We use the following weighted Sobolev-type spaces
\[
L^{2,j}(\Omega)
:=
\{
f(x)\in L^2(\Omega):~
x^{2j}f(x)\in L^2(\Omega)
\},
\]
\[
H^{k,j}(\Omega)=\{f(x)\in L^2(\Omega):~ x^{2j}f(x), f'(x), x^{2j}f'(x),...,f^{(k)}(x), x^{2j}f^{(k)}(x) \in L^2(\Omega)\},
\]

\begin{assumption}\label{assum_1}
    We assume that the initial function $u_0(x)$ satisfies $u_0(x)-A_1\in H^{6,3}(\mathbb{R}^-)$ and $u_0(x)-A_2\in H^{6,3}(\mathbb{R}^+)$, where $\mathbb{R}^-=\{x\in \mathbb{R}\, |\, x<0\}$
    and $\mathbb{R}^+=\{x\in \mathbb{R}\, |\,x>0\}$.
\end{assumption}

\begin{definition}\label{def:cl}
    A function $u(x,t)$ is a \emph{global solution} of the Cauchy problem   \eqref{mCH-1}--\eqref{mCH-bc} if it:
\begin{itemize}
    \item satisfies the mCH equation \eqref{mCH-1} in the classical sense for all $t > 0$ and $x \in \mathbb{R}$,
    \item satisfies the initial condition \eqref{IC},

    \item satisfies the boundary condition
    \[
    u(x,t)\to\begin{cases}
A_1 \text{ as } x\to-\infty\\
A_2 \text{ as } x\to\infty\\
\end{cases}
    \]
for all fixed $t>0$, 
    \item $m(\cdot,t)-A_1\in L^{1}(\mathbb{R}^-)$ and $m(\cdot,t)-A_2\in L^{1}(\mathbb{R}^+)$ for all $t \geq 0$.
\end{itemize}
\end{definition}

The regularity required in Assumption \ref{assum_1} is stronger than the
regularity retained in the definition of a global solution. This is a
technical drawback of the inverse-scattering/Riemann--Hilbert approach. A similar feature occurs in the
inverse-scattering treatment of the KdV equation with step-like initial data;
see \cite{E11}.

Global existence for nonlinear dispersive equations with step-like
boundary conditions is substantially more delicate than in the decaying
or equal-background settings. In the present problem, the two spatial
infinities are associated with different nonzero backgrounds, and the
standard functional-analytic theory for the mCH equation does not directly
provide a global well-posedness result in this setting. This makes the
inverse-scattering/Riemann--Hilbert approach particularly natural: rather
than constructing the solution by a direct evolution argument, one seeks
to reconstruct it from scattering data determined by the initial datum.

The use of inverse scattering to establish existence for integrable
equations with step-like data goes back to the work of Kappeler on the
Korteweg--de Vries equation \cite{K86}; further developments for KdV can
be found in \cites{E09,E11,GR19}; see also \cite{E22}. A complementary
approach is to formulate the inverse problem as a Riemann--Hilbert problem
and deduce existence from its solvability; see, for example, \cite{FLQ21}
for the defocusing nonlinear Schr\"odinger equation. 
With this approach,
it is necessary not only to prove the solvability of the PDE but also to establish certain properties of the solution as a function of spatial and temporal parameters in order to justify the reconstruction
of the classical solution to the mCH equation.
These issues are particularly delicate in the
present step-like setting because of the different spectral backgrounds
at $x=-\infty$ and $x=+\infty$ and the resulting branch-point structure of
the RH problem.

A characteristic feature of the RH approach to Camassa--Holm-type
equations is the use of a reciprocal spatial transformation
\[
(x,t)\longmapsto(y,t).
\]
It is the variables $(y,t)$, rather than $(x,t)$, that enter explicitly as
parameters in the RH problem. Consequently, the existence problem naturally
splits into two distinct parts.

The first part is entirely formulated in the $(y,t)$-variables: one has to
establish the global solvability of the RH problem and show that the quantities
reconstructed from its solution form a classical solution of the transformed
mCH system. The second part consists in recovering the physical spatial
variable $x$ and therefore requires global invertibility of the reciprocal
transformation. In particular, the loss of invertibility of this transformation may be responsible for wave breaking; see \cites{BSZ17,BKS20}.

Regarding the first part, we prove that the RH problem is uniquely solvable for every
\[
(y,t)\in\mathbb R\times[0,\infty),
\]
and that the corresponding reconstruction yields a global classical
solution in the $(y,t)$-scale. The passage back to the original
$(x,t)$-variables is, however, conditional. More precisely, the inverse
transformation satisfies
\[
x_y(y,t)=A_2\alpha(y,t),
\]
where $\alpha$ is obtained by evaluating the large-$\lambda$ behavior of the solution of the associated RH problem, see Proposition \ref{prop-y}.
Therefore, its global invertibility follows provided that
\begin{equation}\label{intro-positivity}
\alpha(y,t)>0,
\qquad
(y,t)\in\mathbb R\times[0,\infty).
\end{equation}

Accordingly, the main unconditional result of this paper is global
existence in the reciprocal $(y,t)$-variables. As for the global existence in the
physical $(x,t)$-variables, it follows under the additional positivity
condition \eqref{intro-positivity},
which can be expressed directly in terms of the solution of the RH problem parameterized by $y$ and $t$.

Throughout the paper, we impose the scattering hypotheses used in
Appendix~\ref{app:A}.  More precisely:

\begin{assumption}\label{ass:scattering}
     We restrict ourselves to the solitonless case, so that the associated RH problem contains no discrete spectral contribution, and we assume the genericity condition  $\gamma_2\neq0 $ at the inner branch points $\lambda=\pm A_2^{-1}$, see Corollary \ref{cor:c11-c21-branch-expansion}.
\end{assumption}

The paper is organized as follows. In Section~\ref{sec:2}, we formulate the basic
right and left Riemann--Hilbert problems associated with the step-like
scattering data.

In Section~\ref{sec:3}, we establish the existence and uniqueness of solutions of
these RH problems. The proof is based on a vanishing lemma together with
Fredholm theory for the associated singular integral operators.

In Section~\ref{sec:4}, we study the dependence of the RH solutions on the parameters
$y$ and $t$ and prove the differentiability properties required for the
subsequent reconstruction.

In Section~\ref{sec:5}, starting from sufficiently regular solutions of the RH
problems and without assuming the prior existence of a solution
$u(x,t)$ of the original Cauchy problem, we derive the corresponding
differential systems in the transformed spatial variables. In particular,
the functions reconstructed from  Right RH problem satisfy the mCH
system in the $(y,t)$-variables. We also consider the analogous
reconstruction from Left RH problem and the relation between the two
formulations.

Finally, in Section~\ref{sec:6}, we address the existence problem for the mCH
equation itself. We first identify the initial data of the functions
reconstructed from the RH problem and establish their behavior as
$y\to\pm\infty$. This yields the existence result for the transformed
problem in the $(y,t)$-scale. We then consider the return to the physical
variables $(x,t)$. This last step requires the reciprocal transformation
$y\mapsto x(y,t)$ to remain globally invertible; equivalently, one has to
assume
$\alpha(y,t)>0.$
Thus, while the RH construction and the resulting solution in the
$(y,t)$-variables are obtained globally, the reconstruction of a global
solution in the original $(x,t)$-variables is conditional on the
preservation of this positivity condition.

\begin{notations*}
In what follows, $\sigma_1\coloneqq\left(\begin{smallmatrix}0&1\\1&0\end{smallmatrix}\right)$, $\sigma_2\coloneqq\left(\begin{smallmatrix}0&-\ii\\\ii&0\end{smallmatrix}\right)$, and $\sigma_3\coloneqq\left(\begin{smallmatrix}1&0\\0&-1\end{smallmatrix}\right)$ denote the standard Pauli matrices, $
E_{12}
:=
\left(\begin{smallmatrix}0&1\\0&0\end{smallmatrix}\right)
$, $
E_{21}
:=
\left(\begin{smallmatrix}0&0\\1&0\end{smallmatrix}\right),
$
$e^{\hat\sigma_3}A:=e^{\sigma_3}A e^{-\sigma_3}$,
$\mathbb{C}^+:=\{\lambda\in\mathbb{C}|\Im(\lambda)> 0\}$, and 
$\mathbb{C}^-:=\{\lambda\in\mathbb{C}|\Im(\lambda)< 0\}$. We also let $f^*(k)\coloneqq\overline{f(\bar k)}$ denote the Schwarz 
reflection of a function $f(k)$, $k\in\D{C}$.
\end{notations*}

\section{Basic Riemann--Hilbert problems}\label{sec:2}

In Appendix~\ref{app:A}, under the assumption that a sufficiently
regular global solution of the Cauchy problem
\eqref{mCH-1}--\eqref{mCH-bc} exists, the associated
Riemann--Hilbert problems are derived from the corresponding Jost
solutions.

In the present section, we do not  assume the existence of such a
solution. Instead, starting from the scattering data generated by the initial
data in a suitable class and motivated by the direct scattering
analysis in Appendix~\ref{app:A}, we formulate the corresponding
right and left problems directly as $L^2$-Riemann--Hilbert problems.
The basic definitions and results concerning $L^2$-RH problems used
throughout the paper are collected in Appendix~\ref{app:B}.

Under Assumption \ref{ass:scattering}, 
 \textbf{Right RH problem} (parameterized by $y$ and $t$) is as follows:
find a $2\times2$ matrix-valued function
$\hat N(y,t,\lambda)$ such that
\begin{equation}
    \label{right_RH_norm}
    \hat N(y,t,\cdot)-I
\in
\dot E^2(\mathbb C\setminus\mathbb R),
\end{equation}
and whose non-tangential boundary values satisfy the following jump
condition:
\begin{subequations}\label{G_right}
\begin{equation}\label{jump-y}
\hat N_+(y,t,\lambda)=\hat N_-(y,t,\lambda) \hat G(y,t,\lambda)
\qquad
\text{for a.e. }\lambda\in\mathbb R 
\end{equation} with
\begin{equation}\label{G-}
 \hat G(y,t,\lambda)=\begin{cases}
\hat G_{\Sigma_1}(y,t,\lambda), \quad \lambda\in\dot\Sigma_1,\\
\hat G_{\Sigma_0}(y,t,\lambda), \quad \lambda\in\dot\Sigma_0,\\
-\ii\sigma_1, \quad \lambda\in\dot\Gamma_2,
\end{cases}  
\end{equation}
where
\begin{equation}\label{G_Sigma-l-}
\hat G_{\Sigma_i}(y,t,\lambda)=\eul^{-\hat f_2(y,t,\lambda)\sigma_3}
G_0(\lambda)
\eul^{\hat f_2(y,t,\lambda)\sigma_3}, \qquad i=0,1
\end{equation}
with the phase function
\begin{equation}\label{f-y}
\hat f_2(y,t,\lambda) := \frac{\ii A_2 l_2(\lambda)}{2}\left(y-\frac{2t}{\lambda^2}\right)
\end{equation}
and the constant matrix factor
\begin{equation}\label{G_0}
G_0(\lambda)=\begin{cases}
\begin{pmatrix}
1-|r(\lambda)|^2&-\overline{r(\lambda)}\\
r(\lambda)&1
\end{pmatrix}, \quad \lambda\in\dot\Sigma_1,\\
\begin{pmatrix}
0&-\frac{1}{r(\lambda)}\\
r(\lambda)&1
\end{pmatrix}, \quad \lambda\in\dot\Sigma_0.
\end{cases}
\end{equation}
\end{subequations}
Here 
$\dot E^2$ is the Smirnov class (see Appendix~\ref{app:B}) and
the ``reflection coefficient'' $r(\lambda)$, $\lambda \in\dot\Sigma_1\cup\dot\Sigma_0$, is defined in terms of $u_0(x)$ through the direct scattering transform associated with the Lax pair of  \eqref{mCH-1} via
\begin{equation}\label{rho-l__}
    r(\lambda)=\begin{cases}
    \frac{c_{21}(\lambda)}{c_{11}(\lambda)},\quad\lambda\in\dot\Sigma_1,\\
    \frac{c_{21}(\lambda_+)}{c_{11}(\lambda_+)},\quad\lambda\in\dot\Sigma_0.
    \end{cases}
\end{equation}

\begin{remark}\label{rem2_1}
In the framework of the direct problem 
(given $u(x,t)$, construct the Jost solutions of the Lax pair equations 
and the associated $\hat N(y,t,\lambda)$, see Appendix \ref{app:A}), 
$\hat N(\lambda)\equiv \hat N(y,t,\lambda)$ has the 
symmetry properties:
\begin{subequations}\label{sym-N_right}
\begin{alignat}{3}\label{sym-N_right_1}
&\hat N(\lambda)=\sigma_2 \hat N(-\lambda)\sigma_2,\qquad 
&&\hat N(\lambda)=\sigma_1\overline{\hat N(\overline{\lambda})}\sigma_1,\qquad &&\lambda\in\mathbb{C}^+,\\\label{sym-N_right_2}
&\hat N(\lambda_+)=-\sigma_2\hat N((-\lambda)_+)\sigma_3,\qquad 
&&\hat N(\lambda_+)=-\ii\sigma_1\overline{\hat N(\lambda_+)},\qquad 
&&\lambda\in\dot\Gamma_2,\\ \label{sym-N_right_3}
&\hat N(\lambda_-)=-\sigma_2 \hat N((-\lambda)_-)\sigma_3,\qquad 
&&\hat N(\lambda_-)=\ii\sigma_1\overline{\hat N(\lambda_-)},\qquad 
&&\lambda\in\dot\Gamma_2.
\end{alignat}
\end{subequations}
Consequently, 
$r(\lambda)$ constructed from $u(x,0)$ satisfies the following symmetries:
\begin{subequations}\label{sym_r}
    \begin{equation}\label{sym_r_1}
        r(\lambda)=-\overline{r(-\lambda)}, \quad \lambda\in\dot\Sigma_1;
    \end{equation}
        \begin{equation}\label{sym_r_0}
        r(\lambda)=-\frac{1}{r(-\lambda)}, \quad \lambda\in\dot\Sigma_0;
    \end{equation}
and
        \begin{equation}\label{requi1}
        |r(\lambda)|=1, \quad \lambda\in\dot\Sigma_0.
    \end{equation}
\end{subequations}

Conversely, in the framework of the inverse problem, suppose that the
prescribed reflection coefficient satisfies the symmetry relations
\eqref{sym_r}. Since
\[
\det\hat G(y,t,\lambda)\equiv1
\]
for almost every $\lambda\in\mathbb R$, Lemma~\ref{lem:L2-uniqueness}
implies that, whenever the right $L^2$-RH problem is solvable, its
solution is unique and has unit determinant:
\begin{equation}\label{det_N}
\det\hat N(y,t,\lambda)\equiv1.
\end{equation}

Moreover, each matrix obtained from $\hat N$ by one of the
transformations appearing in \eqref{sym-N_right} belongs to the same
$L^2$ class, has the same normalization at infinity, and, by
\eqref{sym_r}, satisfies the same jump condition as $\hat N$.
Therefore, uniqueness of the right $L^2$-RH problem implies that its
solution satisfies the symmetry relations \eqref{sym-N_right}.

Furthermore, in the framework of the direct problem, assuming the existence of a
sufficiently regular  $u(x,t)$, the matrix
$\hat N(y,t,\lambda)$ constructed from the corresponding Jost
solutions satisfies the sharper local estimate
\[
\hat N(y,t,\lambda)
=
\ord\left(
\abs{\lambda-\lambda_0}^{-\frac14}
\right),
\qquad
\lambda\to\lambda_0, \text{ for } \lambda_0\in
\left\{
\pm\frac{1}{A_1},
\pm\frac{1}{A_2}
\right\}.
\]
In the $L^2$-RH problem formulated above, this quarter-root estimate
is not imposed as part of the definition.
    
\end{remark}


The corresponding \textbf{Left RH problem} (parameterized by $\tilde y$ and $t$) is as follows:
find a $2\times2$ matrix-valued function $\hat {\tilde N}(\tilde y,t,\lambda)$ such
that
\begin{equation}
    \label{left_RH_norm}
    \hat {\tilde N}(\tilde y,t,\lambda)-I
\in
\dot E^2(\mathbb C\setminus\mathbb R),
\end{equation}
and whose non-tangential boundary values satisfy the following jump
condition:
\begin{subequations}
    \label{G_left}
\begin{equation}\label{jump-y_left}
\hat {\tilde N}_+(\tilde y,t,\lambda)=\hat {\tilde N}_-(\tilde y,t,\lambda) \hat {\tilde G}(\tilde y,t,\lambda),\qquad\lambda\in\mathbb{R}\setminus\left\{\pm\frac{1}{A_1},\pm\frac{1}{A_2}\right\},
\end{equation}
where
\begin{equation}\label{tilde_G_}
    \hat {\tilde G}(\tilde   y,t,\lambda)=\begin{cases}\eul^{-\hat f_1(\tilde y,t,\lambda)\sigma_3}\begin{pmatrix}
1&-\tilde r(\lambda)\\
\overline{\tilde r(\lambda)}&1-|\tilde r(\lambda)|^2
\end{pmatrix}\eul^{\hat f_1(\tilde y,t,\lambda)\sigma_3},\quad \lambda\in\dot\Sigma_1,\\
-\ii\eul^{-\hat f_1(\tilde y,t,\lambda_-)\sigma_3}\begin{pmatrix}
0&1\\
1& \tilde r(\lambda)
\end{pmatrix}\eul^{\hat f_1(\tilde y,t,\lambda_+)\sigma_3},\quad \lambda\in\dot\Sigma_0,\\
-\ii\sigma_1, \quad \lambda\in\dot\Gamma_2,
\end{cases}
\end{equation} with the phase function
\begin{equation}\label{f-y_left}
\hat f_1(\tilde y,t,\lambda) = \frac{\ii A_1 l_1(\lambda)}{2}\left(\tilde y-\frac{2t}{\lambda^2}\right).
\end{equation}
\end{subequations}
Here $\tilde r(\lambda)$, $\lambda \in\dot\Sigma_1\cup\dot\Sigma_0$, is defined in terms of $u_0(x)$ through the direct scattering transform associated with the Lax pair of  \eqref{mCH-1} via
\begin{equation}\label{rho-left__}
    \tilde r(\lambda)=\begin{cases}
        \frac{\overline{c_{21}(\lambda)}}{c_{11}(\lambda)},\quad\lambda\in\dot\Sigma_1,\\
        \frac{1}{c_{21}(\lambda_+)c_{11}(\lambda_+)}=\frac{\ii}{|c_{11}(\lambda_+)|^2}, \quad\lambda\in\dot\Sigma_0.
    \end{cases}
\end{equation}

\begin{remark}
   The symmetry \eqref{sym_l_i-e} implies that 
\[
 \hat {\tilde G}(\tilde y,t,\lambda)=-\ii\begin{pmatrix}
0&1\\
1& \tilde r(\lambda)\eul^{-2\hat f_1(\tilde y,t,\lambda_+)}
\end{pmatrix},\quad \lambda\in\dot\Sigma_0.
\] 
\end{remark}

\begin{remark} By the same argument as for  Right RH problem, we conclude that if the above RH problem is solvable, then its solution
$\hat {\tilde N}(\lambda)\equiv \hat {\tilde N}(\tilde y,t,\lambda)$
satisfies the following properties:
\begin{itemize}
    \item \emph{Symmetry relations:}
    \begin{subequations}\label{sym-til-N}
    \begin{alignat}{3}
    \hat {\tilde N}(\lambda)
    &=
    \sigma_2\hat {\tilde N}(-\lambda)\sigma_2,
    \qquad
    &&
    \hat {\tilde N}(\lambda)
    =
    \sigma_1\overline{\hat {\tilde N}(\overline{\lambda})}\sigma_1,
    \qquad
    &&
    \lambda\in\mathbb C^+,
    \\
    \hat {\tilde N}(\lambda_+)
    &=
    -\sigma_2\hat {\tilde N}((-\lambda)_+)\sigma_3,
    \qquad
    &&
    \hat {\tilde N}(\lambda_+)
    =
    -\ii\sigma_1\overline{\hat {\tilde N}(\lambda_+)},
    \qquad
    &&
    \lambda\in\dot\Gamma_2,
    \\
    \hat {\tilde N}(\lambda_-)
    &=
    -\sigma_2\hat {\tilde N}((-\lambda)_-)\sigma_3,
    \qquad
    &&
    \hat {\tilde N}(\lambda_-)
    =
    \ii\sigma_1\overline{\hat {\tilde N}(\lambda_-)},
    \qquad
    &&
    \lambda\in\dot\Gamma_2.
    \end{alignat}
    \end{subequations}

    \item \emph{Determinant condition:}
    \[
    \det\hat {\tilde N}(\lambda)\equiv 1.
    \]
\end{itemize}
\end{remark}

\begin{remark}
    \label{rem:connection}
    Assume that $r(\lambda)$ and $\tilde r(\lambda)$ are constructed from a given initial function $u_0(x)$. Introduce $\tilde y(y)$ by \eqref{ch_y_til_y}.
The analysis in Appendix~\ref{app:A.6} shows that
\[
\hat N(y,t,\lambda)
=
\hat{\tilde N}(\tilde y(y),t,\lambda)\,P(y,t,\lambda),
\]
where $P$ is given by \eqref{connection_P}.
It follows that for the scattering data arising from the same initial datum, the right and left RH formulations are equivalent. In particular, it is sufficient to establish the solvability and differentiability for  Right RH problem.
    
\end{remark}

\section{Existence of solutions of RH problems.} \label{sec:3}

In this section, we study the existence of solutions of  Right RH problem formulated in Section \ref{sec:2}. For this purpose, we will need some results related to the equivalence of $L^2$-RH problems (our jump may be discontinuous at $\pm\frac{1}{A_j}$) and singular integral equations (see, e.g., \cites{L18,ZH89,GK71}). These results are collected in Appendix \ref{app:B}. 

First, we prove the following vanishing lemma; see \cite{ZH89}.

\begin{proposition}\label{vanishing_lemma} (Vanishing lemma) Let $y\in\mathbb{R}$ and $t\geq0$. Let $M(y,t,\lambda)$ be a solution of 
the RH problem: find a $2\times2$ matrix-valued function
$M(y,t,\lambda)$ such that
\[
M(y,t,\lambda)
\in
\dot E^2(\mathbb C\setminus\mathbb R),
\]
and whose non-tangential boundary values satisfy the  jump
condition \eqref{jump-y} - \eqref{G_right}.

Then $M(y,t,\lambda)\equiv 0$.
\end{proposition}

\begin{proof} First, notice that $\hat{G}(y,t,\lambda)$ is piecewise continuous on $\mathbb{R}$, with its only discontinuities occurring at $\lambda=\pm \frac{1}{A_2}$.
Define $F(\lambda)=M(\lambda)\overline{M(\overline{\lambda})}^T$. Since $F(\lambda)$ is analytic in $\mathbb{C}\setminus\mathbb{R}$ and 
$
M\in\dot E^2(\mathbb C\setminus\mathbb R)$,
the Cauchy--Schwarz inequality for Smirnov classes gives
$
F\in\dot E^1(\mathbb C^\pm)$.
Consequently, the Cauchy theorem for Smirnov functions yields
\[
\int_\mathbb{R} F_\pm(\lambda) d\lambda=0.
\]

For almost every $\lambda\in\mathbb R$, the boundary values of $F$
are $F_+(\lambda)=M_-(\lambda)\hat {G}(\lambda)\overline{M_-(\lambda)}^T$ and $F_-(\lambda)=M_-(\lambda)\overline{\hat {G}(\lambda)}^T\overline{M_-(\lambda)}^T$. Adding the preceding integral identities, we obtain
\[
\int_\mathbb{R} M_-(\lambda)(\hat {G}(\lambda)+\overline{\hat {G}(\lambda)}^T)\overline{M_-(\lambda)}^T d\lambda=0.
\]

We now determine the Hermitian part of the jump matrix. For
$\lambda\in\dot\Sigma_1\cup\dot\Sigma_0$, the phase
$\hat f_2(y,t,\lambda)$ is purely imaginary. Hence
$
\eul^{-\hat f_2(y,t,\lambda)\sigma_3}
$
is unitary. Using the definition of $\hat G$, we find
\[
\hat G(\lambda)+\overline{\hat {G}(\lambda)}^T
=
\begin{cases}
\displaystyle
2\begin{pmatrix}
1-|r(\lambda)|^2&0\\
0&1
\end{pmatrix},
&\lambda\in\dot\Sigma_1,\\[4mm]
\displaystyle
2\begin{pmatrix}
0&0\\
0&1
\end{pmatrix},
&\lambda\in\dot\Sigma_0,\\[4mm]
0,
&\lambda\in\dot\Gamma_2.
\end{cases}
\]
Moreover, Corollary \ref{cor:C-sym} implies that $1-|r(\lambda)|^2=\frac{1}{|c_{11}(\lambda)|^2}>0$.
Thus
$\hat G(\lambda)+\overline{\hat {G}(\lambda)}^T$
is positive semi-definite for almost every $\lambda\in\mathbb R$ and
is positive definite on $\dot\Sigma_1$. Hence diagonal elements of $M_-(\lambda)(\hat {G}(\lambda)+\overline{\hat {G}(\lambda)}^T)\overline{M_-(\lambda)}^T$ should vanish, which implies that $M_-(y,t,\lambda)=M_+(y,t,\lambda)=0$ almost everywhere on $\Sigma_1$.

Since each entry of $M$ belongs to the Smirnov class
$E^2(\mathbb C\setminus\mathbb{R})$ and its non-tangential boundary values vanish
almost everywhere on the set $\dot\Sigma_1$, which has positive
Lebesgue measure, the Luzin-Privalov boundary uniqueness theorem for Smirnov
functions implies that every entry of $M$ vanishes identically in the
corresponding half-plane. Therefore,
\[
M(\lambda)\equiv0.
\]
\end{proof}

We proceed as follows:

\begin{enumerate}[(a)]
    \item We choose an admissible factorization of the jump matrix and apply
    Lemma~\ref{Lem 5.2} to identify  Right RH problem with the
    associated Beals--Coifman singular integral equation
    \eqref{mu_int_eq}.

    \item
    Since the jump matrix and its factors are only piecewise
    continuous, Lemma~\ref{Lem 5.3} cannot be applied directly.
    Instead,  the
    local-symbol criterion from
    Remark~\ref{rem:piecewise-continuous-jump} allows us to conclude that the
    singular integral operator
    \[
    \mathcal A_{\hat G}h=C_+h-(C_-h)\hat G
    \]
    is Fredholm of index zero. The vanishing lemma then implies that
    $\ker\mathcal A_{\hat G}=\{0\}$. Consequently,
    $\mathcal A_{\hat G}$ is bijective, which yields the existence of
    the solution of  Right RH problem.
\end{enumerate}

We first carry out step~\textup{(a)}. Since
$\hat f_2(y,t,\lambda)$ is purely imaginary for
$\lambda\in\dot\Sigma_2$, we have
\[
\overline{\eul^{2\hat f_2(y,t,\lambda)}}
=
\eul^{-2\hat f_2(y,t,\lambda)}.
\]
Moreover, $|r(\lambda)|=1$ on $\dot\Sigma_0$. It follows that the jump
matrix admits the factorization $\hat G=(\hat {G}^-)^{-1}\hat {G}^+$ for almost every $\lambda\in\mathbb R$, where
\begin{subequations}\label{jump_split}
\begin{equation}\label{jump_split_-}
    \hat {G}^-=
    \begin{cases}
    \begin{pmatrix}
    1&\overline{r(\lambda)\eul^{2\hat f_2(\lambda)}}\\
    0&1
    \end{pmatrix},\quad\lambda\in\dot\Sigma_2,\\
        \begin{pmatrix}
    1&0\\
    0&1
    \end{pmatrix},\quad\lambda\in\mathbb{R}\setminus\Sigma_2.
    \end{cases}
    \end{equation}
and
    \begin{equation}\label{jump_split_+}
    \hat {G}^+=
    \begin{cases}
    \begin{pmatrix}
    1&0\\
    r(\lambda)\eul^{2\hat f_2(\lambda)}&1
    \end{pmatrix},\quad\lambda\in\dot\Sigma_2,\\
    \begin{pmatrix}
    0&-\ii\\
    -\ii&0
    \end{pmatrix},\quad\lambda\in\mathbb{R}\setminus\Sigma_2.
    \end{cases}
    \end{equation}
\end{subequations}

By Remark~\ref{ratbranchpoints}, $r\in L^\infty(\Sigma_2)$ and $
r(\lambda)=\ord(\lambda^{-1})
$as $\lambda\to\infty$. Since
$\left|\eul^{2\hat f_2(y,t,\lambda)}\right|=1$ for $\lambda\in\dot\Sigma_2$,
it follows that
$
\hat G^\pm,\;(\hat G^\pm)^{-1}
\in
I+L^2(\mathbb R)\cap L^\infty(\mathbb R)$.
Here the assertion for the inverses follows 
from the fact that $\det \hat {G}^\pm=1$.
Thus the hypotheses of Lemma~\ref{Lem 5.2} are satisfied, and the right
RH problem is equivalent to the Beals--Coifman equation
\begin{equation}
\label{mu_int_eq_RH}
\mu-I=C_w\mu,
\qquad
\mu\in I+L^2(\Gamma),
\end{equation}
where $w^\pm$ are given by
\begin{subequations}\label{w_pm}
\begin{equation}\label{w_-}
    w^-=
    \begin{cases}
    \begin{pmatrix}
    0&-\overline{r(\lambda)\eul^{2\hat f_2(\lambda)}}\\
    0&0
    \end{pmatrix},\quad\lambda\in\dot\Sigma_2,\\
        \begin{pmatrix}
    0&0\\
    0&0
    \end{pmatrix},\quad\lambda\in\mathbb{R}\setminus\Sigma_2,
    \end{cases}
        \end{equation}
and
        \begin{equation}\label{w_+}
    w^+=
    \begin{cases}
    \begin{pmatrix}
    0&0\\
    r(\lambda)\eul^{2\hat f_2(\lambda)}&0
    \end{pmatrix},\quad\lambda\in\dot\Sigma_2,\\
    \begin{pmatrix}
    -1&-\ii\\
    -\ii&-1
    \end{pmatrix},\quad\lambda\in\mathbb{R}\setminus\Sigma_2.
    \end{cases}
    \end{equation}
\end{subequations}

We now turn to step~\textup{(b)}. Consider the operator
\[
\mathcal A_{\hat G}:L^2(\mathbb R)\longrightarrow L^2(\mathbb R),
\qquad
\mathcal A_{\hat G}h=C_+h-(C_-h)\hat G.
\]
Our goal is to prove that $\mathcal A_{\hat G}$ is Fredholm of index zero.

For $\tau\in[0,1]$, define
$\hat G_\tau=(\hat G_\tau^-)^{-1}\hat G_\tau^+$,
where
\[
\hat G_\tau^-(y,t,\lambda)
=
\begin{cases}
\begin{pmatrix}
1&
\tau\overline{r(\lambda)
\eul^{2\hat f_2(y,t,\lambda)}}\\
0&1
\end{pmatrix},
&\lambda\in\dot\Sigma_2,\\[4mm]
I,
&\lambda\in\mathbb R\setminus\Sigma_2,
\end{cases}
\]
and
\[
\hat G_\tau^+(y,t,\lambda)
=
\begin{cases}
\begin{pmatrix}
1&0\\
\tau r(\lambda)\eul^{2\hat f_2(y,t,\lambda)}&1
\end{pmatrix},
&\lambda\in\dot\Sigma_2,\\[4mm]
\eul^{-\frac{\pi\ii\tau}{2}\sigma_1}=
\begin{pmatrix}
\cos\left(\frac{\pi\tau}{2}\right)
&
-\ii\sin\left(\frac{\pi\tau}{2}\right)
\\[4pt]
-\ii\sin\left(\frac{\pi\tau}{2}\right)
&
\cos\left(\frac{\pi\tau}{2}\right)
\end{pmatrix},
&\lambda\in\mathbb R\setminus\Sigma_2.
\end{cases}
\]
Then
$
\hat G_0=I$ and 
$\hat G_1=\hat G$.
Moreover, the map
$\tau\longmapsto\hat G_\tau$
is continuous from $[0,1]$ into $L^\infty(\mathbb R)$. Therefore,
$\tau\longmapsto\mathcal A_{\hat G_\tau}$
is continuous in the operator norm, since
\[
\left\|
\mathcal A_{\hat G_\tau}
-
\mathcal A_{\hat G_s}
\right\|_{\mathcal B(L^2(\mathbb R))}
\leq
\|C_-\|_{\mathcal B(L^2(\mathbb R))}
\|\hat G_\tau-\hat G_s\|_{L^\infty(\mathbb R)}.
\]

The jump matrix $\hat G_\tau$ is continuous except possibly at
$\lambda=\pm A_2^{-1}$. Hence, by
Remark~\ref{rem:piecewise-continuous-jump}, it remains to verify the
local Fredholm condition at $\lambda=\pm A_2^{-1}$ and at infinity.

Since $r(\lambda)=\ord(\lambda^{-1})$ as $\lambda\to\infty$, we have
\[
\hat G_\tau(y,t,\lambda)\longrightarrow I
\]
as $\lambda\to\pm\infty$, uniformly for $\tau\in[0,1]$. Thus the
condition at infinity is satisfied.

Furthermore,
\[
\hat f_2(y,t,\pm A_2^{-1})=0,
\]
and Remark~\ref{ratbranchpoints} gives
\[
r(A_2^{-1}+0)=-1,
\qquad
r(-A_2^{-1}-0)=1.
\]
Therefore, the one-sided limits of $\hat G_\tau$ at
$\lambda=A_2^{-1}$ are
\[
\hat G_{\tau,\ell}(A_2^{-1})
=
\eul^{-\frac{\pi\ii\tau}{2}\sigma_1},
\qquad
\hat G_{\tau,r}(A_2^{-1})
=
\begin{pmatrix}
1-\tau^2&\tau\\
-\tau&1
\end{pmatrix}.
\]
At $\lambda=-A_2^{-1}$, they are
\[
\hat G_{\tau,\ell}(-A_2^{-1})
=
\begin{pmatrix}
1-\tau^2&-\tau\\
\tau&1
\end{pmatrix},
\qquad
\hat G_{\tau,r}(-A_2^{-1})
=
\eul^{-\frac{\pi\ii\tau}{2}\sigma_1}.
\]

At both transition points, a direct calculation gives
\[
\det\left(
\hat G_{\tau,\ell}^{-1}\hat G_{\tau,r}
\right)=1
\]
and
\[
\operatorname{tr}\left(
\hat G_{\tau,\ell}^{-1}\hat G_{\tau,r}
\right)
=
(2-\tau^2)
\cos\left(\frac{\pi\tau}{2}\right)
\geq0.
\]
Consequently,
\[
\operatorname{spec}\left(
\hat G_{\tau,\ell}^{-1}\hat G_{\tau,r}
\right)
\cap(-\infty,0]
=
\varnothing.
\]
Indeed, if one eigenvalue belonged to $(-\infty,0)$, then, since the
determinant is equal to one, both eigenvalues would be negative and
their sum would be at most $-2$, contradicting the nonnegativity of
the trace.

It follows from
Remark~\ref{rem:piecewise-continuous-jump} that
$\mathcal A_{\hat G_\tau}$ is Fredholm for every
$\tau\in[0,1]$. Since the Fredholm index is constant along continuous
paths of Fredholm operators and
$\mathcal A_{\hat G_0}
=
C_+-C_-
=
\mathds{1}$,
we conclude that
$
\operatorname{Ind}\mathcal A_{\hat G}
=
\operatorname{Ind}\mathcal A_{\hat G_0}
=
0$.

We next prove that $\mathcal A_{\hat G}$ is injective. If
$h\in\ker\mathcal A_{\hat G}$, then
\[
M(\lambda)=(Ch)(\lambda)
\]
satisfies
\[
M_+=C_+h=(C_-h)\hat G=M_-\hat G.
\]
Thus $M$ solves the homogeneous right RH problem, and
Proposition~\ref{vanishing_lemma} implies that $M\equiv0$. By the
Sokhotski--Plemelj identity,
\[
h=C_+h-C_-h=M_+-M_-=0.
\]
Hence $\ker\mathcal A_{\hat G}=\{0\}$. Since
$\operatorname{Ind}\mathcal A_{\hat G}=0$, the operator
$\mathcal A_{\hat G}$ is bijective.

This also implies the invertibility of the Beals--Coifman operator. Indeed,
letting $P=C_+$, $Q=-C_-$, and denoting by $\mathcal R_a$ right
multiplication by $a$, the factorization
$\widehat G=(\widehat G^-)^{-1}\widehat G^+$ gives
$(1-C_w)\mathcal R_{(\widehat G^-)^{-1}}
=I+Q(\mathcal R_{\widehat G}-I)$, whereas
$\mathcal A_{\widehat G}
=I+(\mathcal R_{\widehat G}-I)Q$. Since $I+XY$ is invertible if and
only if $I+YX$ is invertible, the bijectivity of
$\mathcal A_{\widehat G}$ implies that $1-C_w$ is invertible on
$L^2(\mathbb R)$.

Moreover, $\hat G-I\in L^2(\mathbb R)$. Let
\[
h=\mathcal A_{\hat G}^{-1}(\hat G-I)
\]
and define
\[
\hat N(y,t,\lambda)=I+(Ch)(\lambda).
\]
Then $\hat N-I\in\dot E^2(\mathbb C\setminus\mathbb R)$, and
\[
\hat N_+-\hat N_-\hat G
=
I+C_+h-(I+C_-h)\hat G
=
I-\hat G+\mathcal A_{\hat G}h
=
0.
\]
Therefore, $\hat N$ is a solution of  Right RH problem.

Finally, since $\det\hat G=1$ almost everywhere on $\mathbb R$,
Lemma~\ref{lem:L2-uniqueness} implies that the solution of the right
RH problem is unique. We have therefore proved the following result.

\begin{proposition}
\label{prop:right-RH-existence}
For every $y\in\mathbb R$ and $t\geq0$,  Right RH problem has a unique solution.
\end{proposition}

By Remark~\ref{rem:connection}, the connection transformation gives a
one-to-one correspondence between solutions of Right and Left RH
problems associated with the same admissible initial datum. Therefore,
Proposition~\ref{prop:right-RH-existence} immediately yields the
following result.

\begin{proposition}
\label{prop:left-RH-existence}
For every $y\in\mathbb R$ and $t\geq 0$,  Left RH problem has a unique solution.
\end{proposition}

\begin{remark}
\label{rem:right-RH-branch-growth}
From the existence proof, the solution of  Right RH problem admits
the representation
\[
\hat N(y,t,\lambda)
=
I+(Ch)(\lambda),
\qquad
h=
\mathcal A_{\hat G}^{-1}(\hat G-I)
\in L^2(\mathbb R).
\]
Therefore, Lemma~\ref{lem:local-Cauchy-growth} from
Appendix~\ref{app:B} immediately gives, for every
\[
\lambda_0\in
\left\{
\pm A_1^{-1},\pm A_2^{-1}
\right\},
\]
the non-tangential estimate
\[
\hat N(y,t,\lambda)
=
o\!\left(
|\lambda-\lambda_0|^{-1/2}
\right),
\qquad
\lambda\to\lambda_0.
\]
\end{remark}

\section{Differentiability of solutions of RH problems.} \label{sec:4}

\begin{proposition}
\label{Prop 5.6}
Let $r(\lambda)=\ord\left(\frac{1}{\lambda^2}\right)$ as $\lambda\to\infty$.
Then the map
\[
(y,t)\longmapsto \mu(y,t,\cdot),
\]
where $\mu$ is the solution of Equation \eqref{mu_int_eq_RH},
is $C^1$ from $\mathbb R\times[0,\infty)$ into the affine space
$I+L^2(\Gamma)$. Equivalently,
\[
(y,t)\longmapsto \mu(y,t,\cdot)-I
\]
is $C^1$ as an $L^2(\Gamma)$-valued map. Moreover, for
$p\in\{y,t\}$,
\begin{equation}
\label{mu-derivative}
\partial_p\mu(y,t,\cdot)
=
\bigl(\mathds{1}-C_{w(y,t)}\bigr)^{-1}
C_{\partial_pw(y,t)}
\mu(y,t,\cdot)
\in L^2(\Gamma),
\end{equation}
where
\[
\partial_pw
=
\bigl(\partial_pw^+,\partial_pw^-\bigr).
\]
At $t=0$, the derivative with respect to $t$ is understood as a
right derivative.
\end{proposition}

\begin{proof}
Recall that
$\mu(y,t,\lambda)
=
I+\bigl(\mathds{1}-C_{w(y,t)}\bigr)^{-1}
C_{w(y,t)}I$.
Thus $\mu(y,t,\cdot)-I$ is obtained from the following two maps:
\begin{align*}
&(y,t)\longmapsto C_{w(y,t)}I:
\mathbb R\times[0,\infty)\longrightarrow L^2(\Gamma),\\
&(y,t)\longmapsto
\bigl(\mathds{1}-C_{w(y,t)}\bigr)^{-1}:
\mathbb R\times[0,\infty)
\longrightarrow\mathcal B(L^2(\Gamma)),
\end{align*}
together with the continuous bilinear map
\[
\mathcal B(L^2(\Gamma))\times L^2(\Gamma)
\longrightarrow L^2(\Gamma),
\qquad
(T,h)\longmapsto Th.
\]

In turn, the map
\[
(y,t)\longmapsto C_{w(y,t)}I:
\mathbb R\times[0,\infty)\longrightarrow L^2(\Gamma)
\]
is the composition of
\begin{align*}
&(y,t)\longmapsto
\bigl(w^+(y,t),w^-(y,t)\bigr):
\mathbb R\times[0,\infty)
\longrightarrow L^2(\Gamma)\times L^2(\Gamma),\\
&Q:(w^+,w^-)\longmapsto C_wI:
L^2(\Gamma)\times L^2(\Gamma)
\longrightarrow L^2(\Gamma).
\end{align*}
Similarly, the map
\[
(y,t)\longmapsto
\bigl(\mathds{1}-C_{w(y,t)}\bigr)^{-1}:
\mathbb R\times[0,\infty)
\longrightarrow\mathcal B(L^2(\Gamma))
\]
is the composition of
\begin{align*}
&(y,t)\longmapsto
\bigl(w^+(y,t),w^-(y,t)\bigr):
\mathbb R\times[0,\infty)
\longrightarrow
L^\infty(\Gamma)\times L^\infty(\Gamma),\\
&F:(w^+,w^-)\longmapsto
\mathds{1}-C_w:
L^\infty(\Gamma)\times L^\infty(\Gamma)
\longrightarrow\mathcal B(L^2(\Gamma)),\\
&\operatorname{Inv}:T\longmapsto T^{-1},
\end{align*}
where $\operatorname{Inv}$ is defined on the open subset of
invertible operators in $\mathcal B(L^2(\Gamma))$.

We first examine the dependence of $w^\pm$ on $(y,t)$. The
parameter-dependent entries of $w^\pm$ contain
\[
r(\lambda)\eul^{2\hat f_2(y,t,\lambda)}
\]
or its complex conjugate. Since
\[
\hat f_2(y,t,\lambda)
=
\frac{\ii A_2l_2(\lambda)}{2}
\left(y-\frac{2t}{\lambda^2}\right),
\]
we have
\begin{align}
\label{q-y-derivative}
\frac{\partial}{\partial y}
\left(
r(\lambda)\eul^{2\hat f_2(y,t,\lambda)}
\right)
&=
\ii A_2l_2(\lambda)
r(\lambda)\eul^{2\hat f_2(y,t,\lambda)},\\
\label{q-t-derivative}
\frac{\partial}{\partial t}
\left(
r(\lambda)\eul^{2\hat f_2(y,t,\lambda)}
\right)
&=
-\frac{2\ii A_2l_2(\lambda)}{\lambda^2}
r(\lambda)\eul^{2\hat f_2(y,t,\lambda)}.
\end{align}
On the part of the contour where these functions occur,
$\hat f_2(y,t,\lambda)$ is purely imaginary, and hence
\[
\left|\eul^{2\hat f_2(y,t,\lambda)}\right|=1.
\]
Moreover,
\[
l_2(\lambda)=\ord(\lambda),
\qquad \lambda\to\infty.
\]
Therefore, the assumption
\[
r(\lambda)=\ord\left(\frac{1}{\lambda^2}\right)
\]
implies
\[
l_2(\lambda)r(\lambda)\in L^2(\Gamma)\cap L^\infty(\Gamma),
\qquad
l_2^2(\lambda)r(\lambda)\in L^\infty(\Gamma),
\]
as well as
\[
\frac{l_2(\lambda)}{\lambda^2}r(\lambda)
\in L^2(\Gamma)\cap L^\infty(\Gamma).
\]
The functions involved are also bounded near the finite endpoints of
the contour.

Using the estimates
\[
\left|\frac{\eul^{\ii hs}-1}{h}-\ii s\right|
\leq 2|s|
\]
and
\[
\left|\frac{\eul^{\ii hs}-1}{h}-\ii s\right|
\leq \frac{|h|s^2}{2},
\]
together with the dominated convergence theorem, we conclude that
\[
(y,t)\longmapsto
\bigl(w^+(y,t),w^-(y,t)\bigr)
\]
is $C^1$ both as a map into
$L^2(\Gamma)\times L^2(\Gamma)$ and as a map into
$L^\infty(\Gamma)\times L^\infty(\Gamma)$.

By the estimate \eqref{C_w_ess_2}, the map $Q$
is bounded and linear from
$L^2(\Gamma)\times L^2(\Gamma)$ into $L^2(\Gamma)$. Its differential
is
\[
dQ_{(w^+,w^-)}(u^+,u^-)=C_uI,
\qquad
u=(u^+,u^-).
\]
It follows from the chain rule that
\[
(y,t)\longmapsto C_{w(y,t)}I
\]
is $C^1$ as an $L^2(\Gamma)$-valued map and
\[
\partial_p\bigl(C_wI\bigr)=C_{\partial_pw}I,
\qquad p\in\{y,t\}.
\]

By the estimate \eqref{C_w_ess_1}, the map
\[
(w^+,w^-)\longmapsto C_w
\]
is bounded and linear from
$L^\infty(\Gamma)\times L^\infty(\Gamma)$ into
$\mathcal B(L^2(\Gamma))$. Hence
\[
F:(w^+,w^-)\longmapsto\mathds{1}-C_w
\]
is $C^1$, with the differential
\[
dF_{(w^+,w^-)}(u^+,u^-)=-C_u.
\]

The set of invertible operators in $\mathcal B(L^2(\Gamma))$ is open,
and the inversion map is $C^1$, with differential
\[
d\operatorname{Inv}_T(S)
=
-T^{-1}ST^{-1}.
\]
Since $\mathds{1}-C_{w(y,t)}$ is invertible for every $(y,t)$, the
chain rule shows that
\[
(y,t)\longmapsto
\bigl(\mathds{1}-C_{w(y,t)}\bigr)^{-1}
\]
is $C^1$. Furthermore, for $p\in\{y,t\}$,
\begin{align}
\label{inverse-Cw-derivative}
\partial_p
\bigl(\mathds{1}-C_w\bigr)^{-1}
&=
-\bigl(\mathds{1}-C_w\bigr)^{-1}
\partial_p\bigl(\mathds{1}-C_w\bigr)
\bigl(\mathds{1}-C_w\bigr)^{-1}=
\bigl(\mathds{1}-C_w\bigr)^{-1}
C_{\partial_pw}
\bigl(\mathds{1}-C_w\bigr)^{-1}.
\end{align}

It now follows from the product rule that
\[
\mu-I
=
\bigl(\mathds{1}-C_w\bigr)^{-1}C_wI
\]
is $C^1$ as an $L^2(\Gamma)$-valued map. Using
\eqref{inverse-Cw-derivative}, we obtain
\begin{align*}
\partial_p\mu
&=
\bigl(\mathds{1}-C_w\bigr)^{-1}
C_{\partial_pw}
\bigl(\mathds{1}-C_w\bigr)^{-1}C_wI
+
\bigl(\mathds{1}-C_w\bigr)^{-1}
C_{\partial_pw}I=
\bigl(\mathds{1}-C_w\bigr)^{-1}
C_{\partial_pw}
\left(
I+
\bigl(\mathds{1}-C_w\bigr)^{-1}C_wI
\right).
\end{align*}
Recalling the definition of $\mu$, we arrive at
\[
\partial_p\mu
=
\bigl(\mathds{1}-C_w\bigr)^{-1}
C_{\partial_pw}\mu,
\qquad p\in\{y,t\},
\]
which is precisely \eqref{mu-derivative}.
\end{proof}

A direct corollary of this proposition is the following theorem.

\begin{theorem}\label{thm:diff_RH}
Let $r(\lambda)=\ord\left(\frac{1}{\lambda^2}\right)$
as $\lambda\to\infty$.
Then, for every $\lambda\in\mathbb C\setminus\Gamma$, the solution
$\hat N(y,t,\lambda)$ of Right RH problem \eqref{right_RH_norm}--\eqref{G_right} is $C^1$ in $y$ and $C^1$
in $t$, locally uniformly with respect to
$\lambda\in\mathbb C\setminus\Gamma$.

Moreover, its non-tangential boundary values satisfy
\[
(y,t)\longmapsto \hat N_\pm(y,t,\cdot)-I
\]
is $C^1$ from $\mathbb R\times[0,\infty)$ into $L^2(\Gamma)$.
At $t=0$, differentiation with respect to $t$ is understood as right
differentiation.
\end{theorem}

\begin{proof}
Recall that
\[
\hat N(y,t,\lambda)
=
I+\frac{1}{2\pi\ii}
\int_\Gamma
\frac{\mu(y,t,z)
\bigl(w^+(y,t,z)+w^-(y,t,z)\bigr)}
{z-\lambda}\,dz.
\]
Set
\[
W(y,t,\cdot):=w^+(y,t,\cdot)+w^-(y,t,\cdot)
\]
and
\[
h(y,t,\cdot):=\mu(y,t,\cdot)W(y,t,\cdot).
\]
Since $\mu-I\in L^2(\Gamma)$, this product is understood as
\[
h
=
W+(\mu-I)W.
\]

By Proposition~\ref{Prop 5.6}, the map
\[
(y,t)\longmapsto\mu(y,t,\cdot)-I
\]
is $C^1$ with values in $L^2(\Gamma)$. Moreover, the proof of
Proposition~\ref{Prop 5.6} shows that
\[
(y,t)\longmapsto w^\pm(y,t,\cdot)
\]
is $C^1$ both as an $L^2(\Gamma)$-valued map and as an
$L^\infty(\Gamma)$-valued map. Consequently,
\[
(y,t)\longmapsto W(y,t,\cdot)
\]
is $C^1$ with values in
$L^2(\Gamma)\cap L^\infty(\Gamma)$.

Since multiplication defines a continuous bilinear map
\[
L^2(\Gamma)\times L^\infty(\Gamma)
\longrightarrow L^2(\Gamma),
\]
the product rule implies that
\[
(y,t)\longmapsto h(y,t,\cdot)
\]
is $C^1$ as an $L^2(\Gamma)$-valued map. More precisely, for
$p\in\{y,t\}$,
\begin{align*}
\label{h-derivative}
\partial_ph
=
(\partial_p\mu)W+\mu\,\partial_pW =
(\partial_p\mu)W
+\partial_pW
+(\mu-I)\partial_pW.
\end{align*}
Indeed, the first and third terms on the right-hand side belong to
$L^2(\Gamma)$ because
\[
\partial_p\mu,\mu-I\in L^2(\Gamma),
\qquad
W,\partial_pW\in L^\infty(\Gamma),
\]
whereas $\partial_pW\in L^2(\Gamma)$.

Using the derivative formula from Proposition~\ref{Prop 5.6}, we may
also write
\[
\partial_p\mu
=
\bigl(\mathds{1}-C_w\bigr)^{-1}
C_{\partial_pw}\mu,
\qquad p\in\{y,t\}.
\]
Thus,
\[
\partial_ph
=
\bigl(\mathds{1}-C_w\bigr)^{-1}
C_{\partial_pw}\mu\,(w^++w^-)
+
\mu\bigl(\partial_pw^++\partial_pw^-\bigr).
\]

Let
\[
(Ch)(\lambda)
=
\frac{1}{2\pi\ii}
\int_\Gamma\frac{h(z)}{z-\lambda}\,dz,
\qquad
\lambda\in\mathbb C\setminus\Gamma.
\]
For each fixed $\lambda\in\mathbb C\setminus\Gamma$, the map
\[
h\longmapsto(Ch)(\lambda)
\]
is a bounded linear map from $L^2(\Gamma)$ into
$\mathbb C^{2\times2}$. Indeed, by the Cauchy--Schwarz inequality,
\[
\left\|(Ch)(\lambda)\right\|
\leq
\frac{1}{2\pi}
\|h\|_{L^2(\Gamma)}
\left\|\frac{1}{\,\cdot-\lambda\,}\right\|_{L^2(\Gamma)}.
\]
The latter bound is uniform when $\lambda$ ranges over a compact subset
of $\mathbb C\setminus\Gamma$. Since
\[
\hat N-I=Ch,
\]
it follows that, for $p\in\{y,t\}$,
\begin{equation*}
\label{N-derivative}
\partial_p\hat N(y,t,\lambda)
=
\frac{1}{2\pi\ii}
\int_\Gamma
\frac{\partial_ph(y,t,z)}
{z-\lambda}\,dz,
\qquad
\lambda\in\mathbb C\setminus\Gamma.
\end{equation*}
Hence $\hat N(y,t,\lambda)$ is $C^1$ in $y$ and $t$, locally
uniformly with respect to
$\lambda\in\mathbb C\setminus\Gamma$.

It remains to consider the boundary values on $\Gamma$. By the
Sokhotski--Plemelj formulas,
\[
\hat N_\pm(y,t,\cdot)
=
I+C_\pm h(y,t,\cdot)
\]
almost everywhere on $\Gamma$. Since
\[
C_\pm\in\mathcal B(L^2(\Gamma))
\]
and $(y,t)\mapsto h(y,t,\cdot)$ is $C^1$ with values in
$L^2(\Gamma)$, it follows that
\[
(y,t)\longmapsto\hat N_\pm(y,t,\cdot)-I
\]
is $C^1$ as an $L^2(\Gamma)$-valued map. Furthermore,
\[
\partial_p\hat N_\pm
=
C_\pm(\partial_ph),
\qquad p\in\{y,t\},
\]
in $L^2(\Gamma)$.

Thus, the differentiability assertion holds pointwise and locally
uniformly away from $\Gamma$, and in the $L^2$ sense for the
non-tangential boundary values on $\Gamma$.
\end{proof}

\begin{remark}
\label{rem:N-derivatives-branch-growth}
For $p\in\{y,t\}$, the proof above gives
\[
\partial_p\hat N(y,t,\lambda)
=
C(\partial_ph)(y,t,\lambda),
\qquad
\partial_ph(y,t,\cdot)\in L^2(\Gamma).
\]
Therefore, Lemma~\ref{lem:local-Cauchy-growth} implies that, for every
\[
\lambda_0\in
\left\{
\pm A_1^{-1},\pm A_2^{-1}
\right\},
\]
we have the non-tangential estimate
\[
\partial_p\hat N(y,t,\lambda)
=
o\!\left(
|\lambda-\lambda_0|^{-1/2}
\right),
\qquad
\lambda\to\lambda_0.
\]
\end{remark}

\begin{theorem}\label{thm:N-differentiability} Assume that $r(\lambda)=\ord(\frac{1}{\lambda^2})$. Then 
 the solution $\hat {N}(y,t,\lambda)$ of Right RH problem \eqref{right_RH_norm}--\eqref{G_right}
 is $C^1$ in $y$ and $C^1$ in $t$.
   
\end{theorem} 

\begin{proof} Recall that $\hat {N}(y,t,\lambda)=I+\frac{1}{2\pi\ii}\int_\Gamma \frac{\mu(y,t,z)(w^+ + w^-)(y,t,z)}{z-\lambda}dz$. The product rule together with Proposition \ref{Prop 5.6} implies that map $(y,t)\mapsto \mu(y,t,z)(w^+ + w^-)(y,t,z):(\mathbb{R},[0,\infty))\to L^2(\Gamma)$ is $C^1$ in $y$ an $C^1$ in $t$. Hence the claim follows.
    
\end{proof}

Now let us discuss the sufficient conditions in terms of initial data for $r(\lambda)=\ord(\frac{1}{\lambda^2})$ to hold. We will also need the sufficient conditions for $r(\lambda)=\ord(\frac{1}{\lambda^3})$ when discussing the existence of the solution in the $(x,t)$ scale (item \textit{2}).

\begin{proposition}\label{prop_r}
\begin{enumerate}
    \item Let $u_0(x)-A_1\in H^{4,2}(\mathbb{R}^-)$, $u_0(x)-A_2\in H^{4,2}(\mathbb{R}^+)$, and $m_0(x):= u_0(x)-u_{0xx}(x)>0$. Then $r(\lambda)=\ord(\frac{1}{\lambda^2})$.
    \item Let $u_0(x)-A_1\in H^{6,3}(\mathbb{R}^-)$, $u_0(x)-A_2\in H^{6,3}(\mathbb{R}^+)$, and $m_0(x):= u_0(x)-u_{0xx}(x)>0$. Then $r(\lambda)=\ord(\frac{1}{\lambda^3})$.
\end{enumerate}

\end{proposition}

\begin{proof} Notice that the proof of the first statement is a part of the proof of the second statement. Hence we will show here only the proof of the second statement.

Recall that $r(\lambda):=\frac{c_{21}(\lambda)}{c_{11}(\lambda)},~\lambda\in\dot\Sigma_1\cup\dot\Sigma_0$.

We start with a formal derivation of the coefficients of the expansion of $\tilde \Psi_j (x,0,\lambda)$ as $\lambda\to\infty$, $\lambda\in\Sigma_1$. For this purpose, we put an Ansatz 
\begin{subequations}\label{til_phi_coeff}
\begin{equation*}\label{Ansatz}
\tilde \Psi_j (x,0,\lambda)=I+\frac{\psi_{j1}(x)}{\lambda}+\frac{\psi_{j2}(x)}{\lambda^2}+\ord(\frac{1}{\lambda^3}), \quad \psi_{jk}(x)=\begin{pmatrix}\psi_{jk}^{11}(x)&\psi_{jk}^{12}(x)\\
\psi_{jk}^{21}(x)&\psi_{jk}^{22}(x)\\
\end{pmatrix}
\end{equation*}
into the $x$-equation of \eqref{comsys}. Then comparing the coefficients near powers of $\lambda$, we get the following:
\begin{equation}
     \psi_{j1}(x)=\begin{pmatrix}\psi_{j1}^{11}(x)&\psi_{j1}^{12}(x)\\
\psi_{j1}^{12}(x)&-\psi_{j1}^{11}(x)\\
\end{pmatrix},\quad \psi_{j2}(x)=\begin{pmatrix}\psi_{j2}^{11}(x)&\psi_{j2}^{12}(x)\\
-\psi_{j2}^{12}(x)&\psi_{j2}^{11}(x)\\
\end{pmatrix}
\end{equation}
where 
\begin{alignat}{4}\label{coef11}
\psi_{j1}^{11}(x)&=-\frac{\ii}{4 A_j^2}\int_{(-1)^j\infty}^x\frac{(m_0(\xi)-A_j)(m_0(\xi)+A_j)}{m_0(\xi)}d\xi,\\\label{coef12}
\psi_{j1}^{12}(x)&=-\frac{m_0(x)-A_j}{2A_jm_0(x)},\\\label{coef21}
\psi_{j2}^{11}(x)&=\frac{\ii}{2 A_j}\int_{(-1)^j\infty}^x\left(m_0(\xi)-A_j\right)\left(-\ii\frac{m_{0x}(\xi)}{2m_0^3(\xi)}+\frac{\ii(m_0(\xi)+A_j)}{8A_j^3m_0(\xi)}\int_{(-1)^j\infty}^\xi \frac{(m_0^2(\tau)-A^2_j)}{m_0(\tau)}d\tau\right)d\xi,\\\label{coef22} \psi_{j2}^{12}(x)&=\frac{1}{m_0(x)}\left(-\ii\frac{m_{0x}(x)}{2m_0^2(x)}-\frac{\ii(m_0(x)-A_j)}{8A_j^3}\int_{(-1)^j\infty}^x\frac{(m^2_0(\xi)-A^2_j)}{m_0(\xi)}d\xi\right).
\end{alignat}
\end{subequations}
Observe that for the coefficients above to be well defined we need $m_0(x)>0$, $u_0(x)-A_1\in H^{3,2}(\mathbb{R}^-)$, $u_0(x)-A_2\in H^{3,2}(\mathbb{R}^+)$, which is provided by our assumptions.

For the rigorous proof that $\tilde\Psi_j(x,0,\lambda)=1+\frac{\psi_{j1}(x)}{\lambda}+\frac{\psi_{j2}(x)}{\lambda^2}+\ord(\frac{1}{\lambda^3})$ as $\lambda\to\infty$, $\lambda\in\Sigma_1$, we consider the Volterra integral equations \eqref{eq_l} columnwise. Applying the successive approximations to these Volterra equations, we get that each of them has a unique solution $\tilde \Psi_j^{(i)}$ ($j=1,2$ and $i=1,2$). In order to get \eqref{til_phi_coeff}--\eqref{coef22}, we use integration by parts (which is possible due to our assumptions) and successive approximations again to estimate the remainder.

Moreover, expanding $H_j^{-1}(\lambda)$ near $\infty$ in the upper half-plane, we get
\begin{equation*}
  H_j^{-1}(\lambda)=\sqrt{\frac{1}{2}}\left(
\begin{pmatrix}
-1&\ii\\
\ii&-1
\end{pmatrix}+\frac{1}{2A_j\lambda}\begin{pmatrix}
\ii&-1\\
-1&\ii
\end{pmatrix} +\frac{1}{8A^2_j\lambda^2}\begin{pmatrix}
-1&\ii\\
\ii&-1
\end{pmatrix}+\frac{3}{16A^3_j\lambda^3}\begin{pmatrix}
\ii&-1\\
-1&\ii
\end{pmatrix}\right)+\ord(\frac{1}{\lambda^4}).  
\end{equation*}
Combining this with \eqref{til_phi_coeff}, we get 
\begin{equation*}
    H_j^{-1} \tilde \Psi_j (\lambda)=\sqrt{\frac{1}{2}}\left(
\begin{pmatrix}
-1&\ii\\
\ii&-1
\end{pmatrix}+\frac{1}{\lambda}\begin{pmatrix}
-\ii\Pi_{j1}&\Pi_{j1}\\
\ii\Xi_{j1}&\Xi_{j1}
\end{pmatrix} +\frac{1}{\lambda^2}\begin{pmatrix}
\ii\Pi_{j2}&\Pi_{j2}\\
-\ii\Xi_{j2}&\Xi_{j2}
\end{pmatrix}\right)+\ord(\frac{1}{\lambda^3}),
\end{equation*}
where
\begin{align*}
    \Pi_{j1}&=-\ii\psi_{j1}^{11}-\psi_{j1}^{12}-\frac{1}{2A_j},\\
    \Xi_{j1}&=\psi_{j1}^{11}+\ii\psi_{j1}^{12}+\frac{\ii}{2A_j},\\
    \Pi_{j2}&=\ii\psi_{j2}^{11}-\psi_{j2}^{12}+\frac{1}{2A_j}(\psi_{j1}^{11}+\ii\psi_{j1}^{12})+\frac{\ii}{8A_j^2},\\
    \Xi_{j2}&=-\psi_{j2}^{11}+\ii\psi_{j2}^{12}-\frac{1}{2A_j}(-\psi_{j1}^{12}-\ii\psi_{j1}^{11})-\frac{1}{8A_j^2}.
\end{align*}

A direct substitution of the expressions for $\Pi_{jk}$ and
$\Xi_{jk}$ into the above expansion shows that the coefficients of
$\lambda^{-1}$ and $\lambda^{-2}$ in the determinant cancel. Hence $\det ( (H_2^{(-1)} \tilde \Psi_2)^{(1)},(H_1^{(-1)} \tilde \Psi_1)^{(1)})=\ord(\frac{1}{\lambda^3})$. By \eqref{Psi_tildePsi}, the determinant representation
\eqref{scatcoeff-l-c}, and  $|e^{-(f_2(x,0,\lambda)+f_1(x,0,\lambda))}|=1$ as $\lambda\in\Sigma_1$, it follows that $c_{21}=\ord(\frac{1}{\lambda^3})$. Finally, $c_{11}=1+\ord(\frac{1}{\lambda})$and therefore $\frac{1}{c_{11}}=1+\ord(\frac{1}{\lambda})$. 

Consequently,
\[
r(\lambda)
=
\frac{c_{21}(\lambda)}{c_{11}(\lambda)}
=
\ord\left(\frac{1}{\lambda^3}\right),
\]
which proves the claim.

\end{proof}


\section{From a solution of the RH problem to a solution of the mCH equation}\label{sec:5}

In this section, we consider Right and Left RH problems formulated
in Section~\ref{sec:2}, with reflection coefficients $r(\lambda)$ and
$\tilde r(\lambda)$ satisfying the corresponding symmetry conditions
but not necessarily arising from initial data $u_0(x)$.

Let
\[
r(\lambda)=\ord(\lambda^{-2}),
\qquad
\tilde r(\lambda)=\ord(\lambda^{-2}),
\qquad
\lambda\to\infty.
\]
By the results established above, the  solutions
$\hat N(y,t,\lambda)$ and
$\hat{\tilde N}(\tilde y,t,\lambda)$ 
of the corresponding RH problems
are $C^1$ with respect to their
spatial variables and to $t$. Moreover, for every branch point
\[
\lambda_0\in
\left\{
\pm A_1^{-1},\pm A_2^{-1}
\right\},
\]
we have, non-tangentially,
\[
\hat N,\ \hat N_y,\ \hat N_t
=
o\left(
|\lambda-\lambda_0|^{-1/2}
\right),
\qquad
\lambda\to\lambda_0,
\]
and
\[
\hat{\tilde N},\
\hat{\tilde N}_{\tilde y},\
\hat{\tilde N}_t
=
o\left(
|\lambda-\lambda_0|^{-1/2}
\right),
\qquad
\lambda\to\lambda_0.
\]
We will show below that these RH solutions determine, at least locally, a solution
of the mCH equation.

\subsection{From a solution of  Right RH problem to a solution of the mCH equation}\label{sec:5.1}

First, let us reformulate the original Lax pair equations 
\begin{subequations}\label{Lax}
\begin{alignat}{4} \label{Lax-x}
    \Psi_x(x,t,\lambda)&=U(x,t,\lambda)\Psi(x,t,\lambda), \\
    \Psi_t(x,t,\lambda)&=V(x,t,\lambda)\Psi(x,t,\lambda), \label{Lax-t}
\end{alignat}
where
\begin{alignat}{4} \label{U}
U&=\frac{1}{2}\begin{pmatrix} -1 & \lambda  m \\
-\lambda m & 1\end{pmatrix},\\ \label{V}
V&=\begin{pmatrix}\lambda^{-2}+\frac{u^2-u_x^2}{2} &
-\lambda^{-1}(u-u_x)-\frac{\lambda(u^2-u_x^2)m}{2}\\
\lambda^{-1}(u+u_x)+\frac{\lambda(u^2- u_x^2)m}{2} &
		 -\lambda^{-2}-\frac{u^2-u_x^2}{2}\end{pmatrix},
\end{alignat}
\end{subequations}
in the $(y,t)$ variables, where 
\begin{equation}
\label{shkala}
    y(x,t)=x-\frac{1}{A_2}\int_x^{+\infty}(m(\xi,t)-A_2)\dd\xi-A_2^2 t
\end{equation}
(recall that it is the introduction of the $(y,t)$ variables
\eqref{shkala} that 
makes it possible to develop, in the framework of the direct problem, the RH problem formalism, where the data for the RH problem depend explicitly on $y$ and $t$ as parameters \cite{KST22}).

\begin{proposition}
The Lax pair \eqref{Lax} in the variables $(y,t)$ takes the form   
\begin{subequations}\label{hatphieq}
\begin{alignat}{4}\label{hatphieq-x}
\hat\Psi_y&=\left( \frac{\lambda A_2}{2} \begin{pmatrix}
0&1\\-1&0
\end{pmatrix} - \frac{A_2}{2\hat m}
\begin{pmatrix}
1 & 0 \\ 0 & -1
\end{pmatrix}\right)\hat\Psi,\\
\hat\Psi_t&=\left(\frac{1}{\lambda^2}\begin{pmatrix}
1&0\\0&-1
\end{pmatrix} +\frac{1}{\lambda}\begin{pmatrix}0&\hat v-\hat u\\\hat v+\hat u& 0\end{pmatrix}\right)\hat\Psi,
\end{alignat}
\end{subequations}
where $\hat u(y,t)= u(x(y,t),t)$, $\hat v(y,t)= u_x(x(y,t),t)$, and
$\hat m(y,t)= m(x(y,t),t)$.
\end{proposition}
\begin{proof}
Differentiating  the identity $x(y(x,t),t)=x$ with respect to~$t$ gives
\begin{equation}\label{dt-0}
0=\frac{\dd}{\dd t}\left(x(y(x,t),t)\right)=x_y(y(x,t),t)y_t(x,t)+x_t(y(x,t),t).
\end{equation}
From \eqref{shkala} it follows that
\begin{equation}\label{x_y}
x_y(y,t)=\frac{A_2}{\hat m(y,t)},
\end{equation}
and 
\[
y_t(x,t)=-\frac{1}{A_2}( u^2- u_x^2)m.
\]
Substituting this and \eqref{x_y} into \eqref{dt-0}, we obtain
\begin{equation}\label{x_t}
x_t(y,t)=\hat u^2(y,t)-\hat v^2(y,t).
\end{equation}    
Introducing $\hat\Psi(y,t) = \Psi (x(y,t),t)$ and taking into account \eqref{x_t} and \eqref{x_y}, the Lax pair \eqref{Lax} in the variables $(y,t)$ takes the form \eqref{hatphieq}.
\end{proof}

\begin{proposition} (mCH equation in the $(y,t)$ variables) \label{mCH_(y,t)}
Let $u(x,t)$ and $m(x,t)>0$ satisfy \eqref{mCH-1} and let $y(x,t)$ be defined by \eqref{shkala}.
Then the mCH equation \eqref{mCH-1} in the $(y,t)$ variables reads as the following system of equations
with respect to $\hat m(y,t):=m(x(y,t),t)$, $\hat u(y,t):=u(x(y,t),t)$, and 
$\hat v(y,t):=u_x(x(y,t),t)$:
\begin{subequations}\label{ch_y}
\begin{align}\label{eq_y}
(\hat m^{-1})_t(y,t)&= 2\hat v(y,t),\\
\hat v(y,t)&=\hat u_y(y,t)\frac{\hat m(y,t)}{A_2},\label{u-1}
\\
\label{hatmy}
\hat u(y,t)&= \hat m(y,t)\left(1+ \frac{\hat v_y(y,t)}{A_2}\right).
\end{align}
\end{subequations}
\end{proposition}

\begin{proof}
Equations \eqref{u-1} and \eqref{hatmy} follow directly by construction.
Substituting $m_t=-((u^2-u_x^2) m)_x$ from \eqref{mCH-1} and $x_t=\hat u^2-\hat v^2$ from \eqref{x_t} into the equality
\[
\hat m_t(y,t)= m_x(x(y,t),t)x_t(y,t)+ m_t(x(y,t),t)
\]
and, using that $(u^2-u_x^2)_x=2mu_x$, we obtain
\[
\hat m_t(y,t)= m_x(x(y,t),t)(\hat u^2-\hat v^2)-m_x(x(y,t),t)(\hat u^2-\hat v^2)-2 m^2(x(y,t),t)\hat v(y,t)=-2\hat v\hat m^2(y,t)
\]
and thus \eqref{eq_y} follows. 
\end{proof}

\begin{remark}
    Introducing $\hat { \omega}(y,t):=\hat { u}^2(y,t)-\hat {v}^2(y,t)$, equations \eqref{ch_y} imply

\begin{equation}\label{mCH_in_y-4}
(\hat { m}^{-1})_t(y,t)= \frac{\hat {\omega}_y(y,t)}{A_2}.
\end{equation}

\end{remark}

Similarly, the reverse  change of variable $(y,t)\mapsto(x,t)$ 
reduces \eqref{ch_y} to \eqref{mCH-1} with 
$ u(x,t):=\hat{ u}(y(x,t),t)$ and $ m(x,t):=\hat{{ m}}(y(x,t),t)$. 
\begin{proposition}\label{prop:hatu_u}
    Let $\hat{ u}(y,t)$, $\hat{ m}(y,t)$,  and $\hat{v}(y,t)$  satisfy 
    \eqref{ch_y} and let $x(y,t)$ be 
such that \eqref{x_y} and \eqref{x_t} hold. Define $ u(x,t):=\hat{ u}(y(x,t),t)$, $ m(x,t):=\hat{{ m}}(y(x,t),t)$ and $ u_x(x,t):=\hat{v} (y(x,t),t)$.
Then \eqref{mCH-1} holds for $ u(x,t)$ and $ m(x,t)$.
\end{proposition}

We proceed  as follows:
\begin{enumerate}[(a)]
\item 
Starting from $\hat N(y,t,\lambda)$, we define $2\times 2$ matrix-valued function \[\hat\Phi(y,t,\lambda) =\sqrt{\frac{1}{2}}\begin{pmatrix}
-1&\ii \\\ii &-1
\end{pmatrix} \hat N(y,t,\lambda) \eul^{-\hat f_2(y,t,\lambda) \sigma_3}\] and show that $\hat\Phi(y,t,\lambda)$ satisfies the system of differential equations:
\begin{equation}\label{Lax-hat-hat}
\begin{split}
\hat\Phi_y&=\doublehat{U}\hat\Phi, \\
\hat \Phi_t&=\doublehat{V}\hat\Phi,
\end{split}
\end{equation}
where $\doublehat{U}$ and $\doublehat{V}$ have the same (rational) dependence on $\lambda$ as in \eqref{hatphieq}, with coefficients given in terms of $\hat N(y,t,\lambda)$ evaluated at appropriate values of $\lambda$.
\item
We show that the compatibility condition for \eqref{Lax-hat-hat}, which is the equality $\doublehat{U}_t - \doublehat{V}_y + [\doublehat{U},\doublehat{V}]=0$, reduces to 
\eqref{ch_y}.
\end{enumerate}

\begin{proposition}\label{prop-y}
Let $\hat N(y,t,\lambda)$ be the solution of  Right RH problem and $\hat n(y,t,\lambda)=\sqrt{\frac{1}{2}}\begin{pmatrix}
-1&\ii \\\ii &-1
\end{pmatrix}\hat N(y,t,\lambda)$. Define
\begin{equation}\label{hatpsiy}
\hat\Phi(y,t,\lambda):=\hat n(y,t,\lambda)\eul^{-\hat f_2(y,t,\lambda)\sigma_3},
\end{equation}
where $\hat f_2(y,t,\lambda)=\frac{\ii A_2l_2(\lambda)}{2}(y-\frac{2t}{\lambda^2})$. Then $\hat\Phi(y,t,\lambda)$ satisfies the differential equation
\[
\hat\Phi_y=\doublehat{U}\hat\Phi
\]
with 
\[\doublehat{U}=\frac{\lambda A_2}{2} \begin{pmatrix}
0&1\\-1&0
\end{pmatrix} - \frac{A_2 \alpha(y,t)}{2}
\begin{pmatrix}
1 & 0 \\ 0 & -1
\end{pmatrix},\]
where $ \alpha(y,t)\in\mathbb{R}$ can be obtained from the large $\lambda$ expansion of $\hat n(y,t,\lambda)$: 
\begin{equation}\label{alpha}
\alpha(y,t)=2\hat N_{12}(y,t).
\end{equation}
Here $\hat N_{jk}$ are  determined by the coefficient in the term of order  $\lambda^{-1}$ in the large-$\lambda$ expansion
\begin{equation}\label{alpha-M}
\hat N(y,t,\lambda)= I+\frac{1}{\lambda}
\begin{pmatrix}\hat N_{11}(y,t) & \hat N_{12}(y,t)\\\hat N_{12}(y,t)&-\hat N_{11}(y,t)\end{pmatrix}+\ord\left(\frac{1}{\lambda^2}\right), \qquad \lambda\in\mathbb{C}^+.
\end{equation}
\end{proposition}

\begin{proof}
First, notice that $\hat\Phi(y,t,\lambda)$ satisfies the jump condition
\[
\hat\Phi^+(y,t,\lambda)=\hat\Phi^-(y,t,\lambda)J_0(\lambda)
\]
with the jump matrix \[J_0=\begin{cases}
G_0,\quad\lambda\in\dot\Sigma_2\setminus\{\pm\frac{1}{A_1}\},\\
-\ii\sigma_1,\quad\lambda\in\dot\Gamma_2
\end{cases}\]
independent of $y$. Hence, $\hat\Phi_y(y,t,\lambda)$ satisfies the same jump condition. Consequently,
\[
\hat\Phi_y \hat\Phi^{-1}=\hat n_y\hat n^{-1}-\hat f_{2y}\hat n\sigma_3\hat n^{-1}
\]
has no jump, and thus it is a meromorphic function, with possible singularities at $\lambda=\infty$ and $\lambda=\pm\frac{1}{A_j} $. Let us evaluate $\hat\Phi_y\hat\Phi^{-1}$ near these points.

\begin{enumerate}
   [(i)]

\item
As $\lambda\to\infty$, we have \[\hat f_{2y}=
\frac{\ii A_2\lambda}{2}+\ord(\lambda^{-1}).\] 
Moreover, the symmetry \eqref{sym-N_right_1} implies
\[
\hat N(y,t,\lambda)= I+\frac{1}{\lambda}
\begin{pmatrix}\hat N_{11}(y,t) & \hat N_{12}(y,t)\\\hat N_{21}(y,t)&\hat N_{22}(y,t)\end{pmatrix}+\ord\left(\frac{1}{\lambda^2}\right), \qquad \lambda\in\mathbb{C}^+,
\]
\[
\hat N(y,t,\lambda)= I+\frac{1}{\lambda}
\begin{pmatrix}-\hat N_{22}(y,t) & \hat N_{21}(y,t)\\\hat N_{12}(y,t)&-\hat N_{11}(y,t)\end{pmatrix}+\ord\left(\frac{1}{\lambda^2}\right), \qquad \lambda\in\mathbb{C}^-
\]
with $\hat N_{22}$, $\hat N_{11}\in \ii \mathbb{R}$, and $\hat N_{12}$, $\hat N_{21}\in \mathbb{R}$. Therefore,
\begin{equation}\label{hat_M_infty}
  \hat n(y,t,\lambda)=
\sqrt{\frac{1}{2}}
\begin{pmatrix}
-1&\ii\\
\ii&-1
\end{pmatrix}+\frac{1}{\lambda}
\begin{pmatrix}\hat n_{11}(y,t) & \hat n_{12}(y,t)\\\hat n_{21}(y,t)&\hat n_{22}(y,t)\end{pmatrix}+\ord\left(\frac{1}{\lambda^2}\right),\qquad \lambda\in\mathbb{C}^+  
\end{equation}
with 
\[
\hat n_{11}=\frac{1}{\sqrt{2}}(-\hat N_{11}+\ii \hat N_{21} ), \quad \hat n_{12}=\frac{1}{\sqrt{2}}(-\hat N_{12}+\ii \hat N_{22} ), \quad \hat n_{21}=\frac{1}{\sqrt{2}}(-\hat N_{21}+\ii \hat N_{11} ),\quad \hat n_{22}=\frac{1}{\sqrt{2}}(-\hat N_{22}+\ii \hat N_{12} ).
\]

Since $\hat \Phi_y \hat \Phi^{-1}$ is meromorphic in $\mathbb{C}$ its expansions in $\mathbb{C}^+$ and $\mathbb{C}^-$ should coincide and thus,
\[
\hat N_{11}=-\hat N_{22},\qquad \hat N_{12}=\hat N_{21},
\]
and 
\[
\hat n_{12}(y,t)=\ii \hat n_{11}\in\mathbb{R},\qquad \hat n_{21}(y,t)=\ii \hat n_{22}\in\mathbb{R}.
\]

Therefore,
\begin{equation}\label{psi-inf}
\hat\Phi_y\hat\Phi^{-1}= \frac{\lambda A_2}{2} \begin{pmatrix}
0&1\\-1&0
\end{pmatrix} - \frac{A_2 \alpha(y,t)}{2}
\begin{pmatrix}
1 & 0 \\ 0 & -1
\end{pmatrix}+\ord(\lambda^{-1}),\qquad \lambda\to\infty
\end{equation}
with $\alpha(y,t)=-\ii\sqrt{2}(\hat n_{11}(y,t)+ \hat n_{22}(y,t))$.

\item Let
$
\lambda_0\in
\left\{
\pm A_1^{-1},\pm A_2^{-1}
\right\}$.
As $\lambda\to\lambda_0$,
$
\hat n,\ \hat n_y
=
o\left(
|\lambda-\lambda_0|^{-1/2}
\right)$.
Since $\det\hat n\equiv1$, we also have
$
\hat n^{-1}
=
o\left(
|\lambda-\lambda_0|^{-1/2}
\right)$.
Moreover,
$
\hat f_{2y}(\lambda)=\ord(1)$ as
$\lambda\to\lambda_0$.
Therefore,
$
\hat\Phi_y\hat\Phi^{-1}
=
o\left(
|\lambda-\lambda_0|^{-1}
\right)$.

The function $\hat\Phi_y\hat\Phi^{-1}$ is holomorphic in a punctured
neighborhood of $\lambda_0$. The estimate
$
\hat\Phi_y\hat\Phi^{-1}
=
o\left(
|\lambda-\lambda_0|^{-1}
\right)$
excludes every negative-power term in its Laurent expansion. Hence
$\lambda_0$ is a removable singularity, and
\begin{equation}\label{psi-sing}
\hat\Phi_y\hat\Phi^{-1}=\ord(1),
\qquad
\lambda\to\pm\frac{1}{A_j}.
\end{equation}

\end{enumerate}

Combining \eqref{psi-inf} and \eqref{psi-sing}, we conclude that the function
\[
\hat\Phi_y\hat\Phi^{-1}
-\frac{\lambda A_2}{2}
\begin{pmatrix}
0&1\\
-1&0
\end{pmatrix}
-\frac{A_2\alpha(y,t)}{2}
\begin{pmatrix}
1&0\\
0&-1
\end{pmatrix}
\]
is entire in $\lambda$ and tends to zero as $\lambda\to\infty$.
Therefore, by Liouville's theorem, it vanishes identically.
\end{proof}

\begin{remark}\label{remark:nonvan}
Alternatively, $\hat\Phi_y\hat\Phi^{-1}$ can be evaluated at
$\lambda=0_+$ using
$l_2(0_+)= \frac{\ii}{A_2}$.
Indeed, by the symmetry relations \eqref{sym-N_right}, there exist real-valued functions
$\hat a_1$ and $\hat b_1$ such that
\[
\hat n(y,t,0_+)
=
\ii
\begin{pmatrix}
0&\hat a_1(y,t)\\
\hat b_1(y,t)&0
\end{pmatrix}.
\]
Since $\hat n(y,t,0_+)$ is finite and
$\det\hat n(y,t,0_+)=1$, we have
\[
\hat a_1(y,t)\hat b_1(y,t)=1.
\]
Hence $\hat a_1(y,t)\neq0$ and
$\hat b_1(y,t)=\hat a_1^{-1}(y,t)$.
Therefore,
\begin{equation*}
    \label{hat_n_0}
    \hat n(y,t,0_+)=\ii\begin{pmatrix}
0 & \hat a_1(y,t)\\
\hat a^{-1}_1(y,t) & 0
\end{pmatrix},
\end{equation*}
and a direct calculation yields
\[
(\hat\Phi_y \hat\Phi^{-1})(y,t,0_+)=\begin{pmatrix}
\frac{\hat a_{1y}(y,t)}{\hat a_{1}(y,t)}-\frac{1}{2} & 0\\
0 & \frac{1}{2}-\frac{\hat a_{1y}(y,t)}{\hat a_{1}(y,t)}
\end{pmatrix}
\]

Comparing this expression with Proposition~\ref{prop-y}, we obtain
\begin{equation}\label{alpha_inf}
   A_2\alpha=1-2\frac{\hat a_{1y}(y,t)}{\hat a_{1}(y,t)}. 
\end{equation}
\end{remark}

\begin{proposition}\label{prop-t}
The function $\hat\Phi(y,t,\lambda)$ defined by \eqref{hatpsiy} satisfies the differential equation
\begin{equation}\label{psi-t-i}
\hat\Phi_t=\doublehat{V}\hat\Phi
\end{equation}
with 
\[\doublehat{V}=\frac{1}{\lambda^2}\begin{pmatrix}
1&0\\0&-1
\end{pmatrix} +\frac{1}{\lambda}\begin{pmatrix}0&\beta(y,t)\\\gamma(y,t)& 0\end{pmatrix}.\] 
Here 
\begin{equation}
    \label{beta-gamma}
    \beta(y,t)=-2\hat a_2(y,t)\hat a_1(y,t),\qquad \gamma(y,t)=2\hat a_3(y,t)\hat a_1^{-1}(y,t),
\end{equation}
where $\hat a_j(y,t)\in\mathbb{R}$, $j=1,2,3$, are
 determined evaluating $\hat N(y,t,\lambda)$ as $\lambda\to 0$, $\lambda\in\mathbb{C}^+$,
 \begin{equation}\label{M-hat-expand}
\hat N(y,t,\lambda)=-\sqrt{\frac{1}{2}}\left(\ii\begin{pmatrix}
\ii \hat a_1^{-1}(y,t)& \hat a_1(y,t)\\ \hat a_1^{-1}(y,t)&\ii \hat a_1(y,t)
\end{pmatrix}+\ii\lambda\begin{pmatrix}
\hat a_2(y,t)&\ii \hat a_3(y,t)\\\ii \hat a_2(y,t)&\hat a_3(y,t)
\end{pmatrix}  \right)
+\ord(\lambda^2).
\end{equation}
Here $\hat a_1$ is nonvanishing by \eqref{remark:nonvan}.

\end{proposition}

\begin{proof} As in Proposition \ref{prop-y}, we observe that 
\[
\hat\Phi_t \hat\Phi^{-1}=\hat n_t\hat n^{-1}-\hat f_{2t}\hat n\sigma_3\hat n^{-1}
\]
has no jump across the contour. It therefore extends to a meromorphic function of
$\lambda$, whose only possible singularities are located at
 $\lambda=\infty$, $\lambda=\pm \frac{1}{A_j}$, and $\lambda=0$. The possible singularity at  $\lambda=0$  arises from the behavior of
$\hat f_{2t}$, namely
\begin{equation}\label{psi-t-i-1}
\hat f_{2t
}(\lambda)=-\frac{\ii A_2l_2(\lambda)}{\lambda^2}=\begin{cases}
\frac{1}{\lambda^2}+\ord(1), \qquad\lambda\to0,\qquad\lambda\in\mathbb{C}^+,\\
-\frac{1}{\lambda^2}+\ord(1), \qquad\lambda\to0,\qquad\lambda\in\mathbb{C}^-.
\end{cases}
\end{equation}
We now evaluate the local behavior of $\hat\Phi_t\hat\Phi^{-1}$ at each of these points.

\begin{enumerate}[(i)]
\item
As $\lambda\to\infty$, we have $\hat f_{2t}(\lambda)=\ord(\lambda^{-1})$ and hence
\begin{equation}\label{psi-inf-t}
\hat\Phi_t\hat\Phi^{-1}(\lambda)=\ord(\lambda^{-1}),\qquad\lambda\to\infty.
\end{equation}

\item  Let
$
\lambda_0\in
\left\{
\pm A_1^{-1},\pm A_2^{-1}
\right\}$.
As $\lambda\to\lambda_0$,
$
\hat n,\ \hat n_t
=
o\left(
|\lambda-\lambda_0|^{-1/2}
\right)$.
Since $\det\hat n\equiv1$, we also have
$
\hat n^{-1}
=
o\left(
|\lambda-\lambda_0|^{-1/2}
\right)$.
Moreover,
$
\hat f_{2t}(\lambda)=\ord(1)$ as
$\lambda\to\lambda_0$.
Therefore,
$
\hat\Phi_t\hat\Phi^{-1}
=
o\left(
|\lambda-\lambda_0|^{-1}
\right)$.

The function $\hat\Phi_t\hat\Phi^{-1}$ is holomorphic in a punctured
neighborhood of $\lambda_0$. The estimate
$
\hat\Phi_t\hat\Phi^{-1}
=
o\left(
|\lambda-\lambda_0|^{-1}
\right)$
excludes every negative-power term in its Laurent expansion. Hence
$\lambda_0$ is a removable singularity, and
\begin{equation}\label{psi-sing-t}
\hat\Phi_t\hat\Phi^{-1}=\ord(1),
\qquad
\lambda\to\pm\frac{1}{A_j}.
\end{equation}

\item Finally, let us consider $\lambda\to 0$.

First, notice that the symmetries \eqref{sym-N_right_2} imply that
as $\lambda\to 0$, $\lambda\in\mathbb{C}^+$,
\[
\hat N(y,t,\lambda)=-\sqrt{\frac{1}{2}}\left(\ii\begin{pmatrix}
\ii \hat a_1^{-1}(y,t)& \hat a_1(y,t)\\ \hat a_1^{-1}(y,t)&\ii \hat a_1(y,t)
\end{pmatrix}+\ii\lambda\begin{pmatrix}
\hat a_2(y,t)&\ii \hat a_3(y,t)\\\ii \hat a_2(y,t)&\hat a_3(y,t)
\end{pmatrix}  \right)
+\ord(\lambda^2)
\]
where $\hat a_j(y,t)\in\mathbb{R}$, $j=1,2,3$. 
Therefore, taking into account \eqref{psi-t-i-1} and the fact that $\hat\Phi_t\hat\Phi^{-1}(\lambda)$ is meromorphic in $\mathbb{C}$, we obtain 
\begin{equation}\label{hathatVi}
\hat\Phi_t\hat\Phi^{-1}(\lambda)=\frac{1}{\lambda^2}\begin{pmatrix}
1&0\\0&-1
\end{pmatrix}+\frac{1}{\lambda}\begin{pmatrix}
                  0 & -2\hat a_2 \hat a_1 \\
                 2 \hat a_3 \hat a_1^{-1} & 0
\end{pmatrix}+\ord(1),\qquad \lambda\to0.
\end{equation}
\end{enumerate}

Combining \eqref{psi-inf-t}, \eqref{psi-sing-t}, and \eqref{hathatVi},
we conclude that the function
\begin{align*}
\hat\Phi_t\hat\Phi^{-1}(\lambda)-\frac{1}{\lambda^2}\begin{pmatrix}
1&0\\0&-1
\end{pmatrix}-\frac{1}{\lambda}\begin{pmatrix}
                  0 & \beta \\
                 \gamma & 0
\end{pmatrix}	
\end{align*}			
with $\beta=-2\hat a_2\hat a_1$ and $\gamma=2\hat a_3\hat a_1^{-1}$
is holomorphic in the whole complex $\lambda$-plane and, moreover, vanishes as $\lambda\to\infty$. Then, by Liouville's theorem, it vanishes identically. Therefore, we arrive at \eqref{psi-t-i}.
\end{proof}

Finally, by comparing the coefficients in the expansions of the
compatibility condition for \eqref{Lax-hat-hat}
\begin{equation}\label{compat}
\doublehat{U}_t - \doublehat{V}_y + [\doublehat{U},\doublehat{V}]=0
\end{equation}
as $\lambda\to\infty$ and
$\lambda\to0$, we obtain algebraic and differential relations among
$\alpha$, $\beta$, and $\gamma$, that can be reduced to \eqref{ch_y}.

\begin{proposition}

Let $\alpha(y,t)$, $\beta(y,t)$, and $\gamma(y,t)$ be the functions determined in terms of $\hat N(y,t,\lambda)$ as in Propositions \ref{prop-y} and \ref{prop-t}. Then they satisfy the following equations:
\begin{subequations}\label{rel}
\begin{align}\label{rel-a}
&-\alpha_{t}+\beta+\gamma= 0;\\
\label{rel-b}
&\beta_y+A_2(\alpha\beta+1)=0;\\
\label{rel-c}
&\gamma_y + A_2(-\alpha\gamma+1)=0.
\end{align}
\end{subequations}
\end{proposition}

\begin{proof} 

(a)
Evaluating the l.h.s. of \eqref{compat} as $\lambda\to\infty$, the main term (of order $\ord(1)$) is
\[
\frac{A_2}{2}\left(-\alpha_{t}+\beta+\gamma\right)\sigma_3,
\]
from which \eqref{rel-a} follows.

(b), (c)
Evaluating the l.h.s.\ of \eqref{compat} as $\lambda\to 0$, the main term (of order $\ord(\lambda^{-1})$)
is
\[
\begin{pmatrix}
0&-\beta_y+ A_2(-\alpha\beta-1)\\
-\gamma_y - A_2(-\alpha\gamma+1)&0
\end{pmatrix},
\]
from which \eqref{rel-b} and \eqref{rel-c} follow.
\end{proof}

 

\begin{corollary}\label{reduce}
Introducing  $\hat m(y,t)$, $\hat u(y,t)$, and 
$\hat v(y,t)$
in terms of $\alpha$, $\beta$, and $\gamma$ by
\begin{equation}\label{hmbbgg}
\hat m = \alpha^{-1},\quad\hat u=\frac{\gamma-\beta}{2},\quad 
\hat v=\frac{\gamma+\beta}{2},
\end{equation}
 the system \eqref{rel} reduce to \eqref{ch_y}.
\end{corollary}

Now assume that $\hat m(y,t)>0$ and $\hat m(\cdot,t)-A_2\in L^1(\mathbb{R}^+)$, and define the inverse change of variables $(y,t)\mapsto(x,t)$ by 
\begin{equation}\label{x_y_}  x(y,t):=y-\int_{y}^{+\infty}\left(\frac{A_2}{\hat m}-1\right)(\xi,t)\dd \xi+A_2^2t.
\end{equation}
Then $x_y(y,t)=\frac{A_2}{\hat m(y,t)}$, and for each fixed $t$, the map $y\mapsto x(y,t)$ is locally invertible. Moreover, Proposition \ref{prop:hatu_u} implies that the functions $u(x,t)=\hat u(y(x,t),t)$ and $m(x,t)=\hat m(y(x,t),t)$ satisfy \eqref{mCH-1}. 

\begin{theorem}\label{thm_diff_hat} Let $r(\lambda)=\ord(\frac{1}{\lambda^2})$ as $\lambda\to\infty$. Then the functions
     $\hat u(y,t)$, $\hat v(y,t)$, $\hat m^{-1}(y,t)$ defined in \eqref{hmbbgg} are $C^1$ in $y$ and $C^1$ in $t$. 
  
\end{theorem} 

\begin{proof}
By Theorem~\ref{thm:N-differentiability} and the expansions
\eqref{alpha-M} and \eqref{M-hat-expand}, the functions
$\alpha$, $\hat a_1$, $\hat a_2$, and $\hat a_3$ are $C^1$ in $y$
and $t$. Since  the function $\hat a_1$ does not
vanish, $\hat a_1^{-1}$ is also $C^1$. Therefore, $\beta$ and
$\gamma$ are $C^1$, and the result follows from \eqref{hmbbgg}.
\end{proof}

\subsection{From a solution of Left RH problem to a solution of the mCH equation}

Alternatively, we can start with the reformulation of the mCH equation \eqref{mCH-1} in the $(\tilde y,t)$ variables, where
 \begin{equation}\label{shkala-l_left}
    \tilde y(x,t)=x+\frac{1}{A_1}\int_{-\infty}^x(m(\xi,t)-A_1)\dd\xi-A_1^2 t.
\end{equation} 

\begin{proposition}
    The Lax pair \eqref{Lax} in $(\tilde y,t)$ takes the form
     \begin{subequations}\label{hatphieq_l}
\begin{alignat}{4}\label{hatphieq-x_l}
\hat\Psi_{\tilde y}&=\left( \frac{\lambda A_1}{2} \begin{pmatrix}
0&1\\-1&0
\end{pmatrix} - \frac{A_1}{2\hat m}
\begin{pmatrix}
1 & 0 \\ 0 & -1
\end{pmatrix}\right)\hat\Psi,\\
\hat\Psi_t&=\left(\frac{1}{\lambda^2}\begin{pmatrix}
1&0\\0&-1
\end{pmatrix} +\frac{1}{\lambda}\begin{pmatrix}0&\hat v-\hat u\\\hat v+\hat u& 0\end{pmatrix}\right)\hat\Psi.
\end{alignat}
\end{subequations}
where $\hat u(\tilde y,t)= u(x(\tilde y,t),t)$, $\hat v(\tilde y,t)= u_x(x(\tilde y,t),t)$, and
$\hat m(\tilde y,t)= m(x(\tilde y,t),t)$.
\end{proposition}

\begin{proposition}\label{mCH_(y,t)_}
Let $u(x,t)$ and $m(x,t)>0$ satisfy \eqref{mCH-1} and let $\tilde y(x,t)$ be defined by \eqref{shkala-l_left}.

Then the mCH equation \eqref{mCH-1} in the $(\tilde y,t)$ variables reads as the following system of equations
with respect to $\hat{\tilde m} (\tilde y,t):=m(x(\tilde y,t),t)$, $\hat {\tilde u}(\tilde y,t):=u(x(\tilde y,t),t)$, and 
$\hat {\tilde v}(\tilde y,t):=u_x(x(\tilde y,t),t)$:
\begin{subequations}\label{ch_y_}
\begin{align}\label{eq_y_}
(\hat {\tilde m}^{-1})_t(\tilde y,t)&= 2\hat {\tilde v}(\tilde y,t),\\
\hat {\tilde v}(\tilde y,t)&=\hat {\tilde u}_{\tilde y}(\tilde y,t)\frac{\hat {\tilde m}(\tilde y,t)}{A_1},\label{u-1_}
\\
\label{hatmy_}
\hat {\tilde m}(\tilde y,t)&=\hat {\tilde u}(\tilde y,t) - \left(\hat {\tilde v}\right)_{\tilde y}(\tilde y,t) \frac{\hat {\tilde m}(\tilde y,t)}{A_1}.
\end{align}
\end{subequations}
\end{proposition}

\begin{proposition}\label{prop:hatu_u_til}
    Let $\hat{\tilde  u}(\tilde y,t)$, $\hat{\tilde  m}(\tilde y,t)$,  and $\hat{\tilde v}(\tilde y,t)$  satisfy 
    \eqref{ch_y_} and let $x(\tilde y,t)$ be 
such that $x_{\tilde y}(\tilde y,t)=\frac{A_1}{\hat {\tilde m}(\tilde y,t)}$ and $x_t(\tilde y,t)=\hat{\tilde u}^2(\tilde y,t)-\hat{\tilde v}^2(\tilde y,t)$ hold. Define $ u(x,t):=\hat{\tilde  u}(\tilde y(x,t),t)$, $ m(x,t):=\hat{{\tilde  m}}(\tilde y(x,t),t)$ and $ u_x(x,t):=\hat{\tilde v} (\tilde y(x,t),t)$.
Then \eqref{mCH-1} holds for $ u(x,t)$ and $ m(x,t)$.
\end{proposition}

\begin{proposition}\label{prop-y_}
Let $\hat {\tilde N}(\tilde y,t,\lambda)$ be the solution of Left RH problem and $\hat {\tilde n}(\tilde y,t,\lambda)=\sqrt{\frac{1}{2}}\begin{pmatrix}
-1&\ii \\\ii &-1
\end{pmatrix}\hat {\tilde N}(\tilde y,t,\lambda)$. Define
\begin{equation}\label{hatpsiy_}
\hat{\tilde\Phi}(\tilde y,t,\lambda):=\hat {\tilde n}(\tilde y,t,\lambda)\eul^{-\hat f_1(\tilde y,t,\lambda)\sigma_3},
\end{equation}
where $\hat f_1(\tilde y,t,\lambda)=\frac{\ii A_1l_1(\lambda)}{2}(\tilde y-\frac{2t}{\lambda^2})$. Then $\hat{\tilde \Phi}(\tilde y,t,\lambda)$ satisfies the differential equation
\[
\hat{\tilde \Phi}_{\tilde y}=\doublehat{\tilde U}\hat{\tilde \Phi}
\]
with 
\[\doublehat{\tilde U}=\frac{\lambda A_1}{2} \begin{pmatrix}
0&1\\-1&0
\end{pmatrix} - \frac{A_1 \tilde \alpha(\tilde y,t)}{2}
\begin{pmatrix}
1 & 0 \\ 0 & -1
\end{pmatrix},\]
where $\tilde  \alpha(\tilde y,t)\in\mathbb{R}$ can be obtained from the large $\lambda$ expansion of $\hat N(\tilde y,t,\lambda)$: 
\begin{equation}\label{alpha_}
\tilde\alpha(\tilde y,t)=2{\tilde N}_{12}(\tilde y,t),
\end{equation}
where
\begin{equation}\label{alpha-M_}
\hat {\tilde N}(\tilde y,t,\lambda)= I+\frac{1}{\lambda}
\begin{pmatrix}\hat {\tilde N}_{11}(\tilde y,t) & \hat {\tilde N}_{12}(\tilde y,t)\\\hat {\tilde N}_{12}(\tilde y,t)&-\hat {\tilde N}_{11}(\tilde y,t)\end{pmatrix}+\ord\left(\frac{1}{\lambda^2}\right), \qquad \lambda\in\mathbb{C}^+.
\end{equation}
\end{proposition}

\begin{remark}
On the other hand,  $\hat{\tilde \Phi}_{\tilde y} \hat{\tilde \Phi}^{-1}$ can be evaluated  at $\lambda=0_+$ in by using 
\[
l_1(0_+)= \frac{\ii}{A_1}
\]
 and 
\[
\hat {\tilde n}(\tilde y,t,0_+)=\ii\begin{pmatrix}
0 & \hat b_1(\tilde y,t)\\
\hat b^{-1}_1(\tilde y,t) & 0
\end{pmatrix}
\] 
 where the latter identity follows from \eqref{sym-til-N}.
A direct computation yields
\[
(\hat{\tilde \Phi}_{\tilde y} \hat{\tilde \Phi}^{-1})(\tilde y,t,0_+)=\begin{pmatrix}
\frac{\hat b_{1\tilde y}(\tilde y,t)}{\hat b_{1}(\tilde y,t)}-\frac{1}{2} & 0\\
0 & \frac{1}{2}-\frac{\hat b_{1\tilde y}(\tilde y,t)}{\hat b_{1}(\tilde y,t)}
\end{pmatrix}.
\]

  Comparing this expression with Proposition \ref{prop-y_}, we obtain
\begin{equation}\label{alpha_inf_}
   A_1\tilde \alpha=1-2\frac{\hat b_{1\tilde y}(\tilde y,t)}{\hat b_{1}(\tilde y,t)}. 
\end{equation}
\end{remark}

\begin{proposition}\label{prop-t_}
The function $\hat{\tilde \Phi}(\tilde y,t,\lambda)$ defined by \eqref{hatpsiy_} satisfies the differential equation
\begin{equation}\label{psi-t-i_}
\hat{\tilde \Phi}_t=\doublehat{\tilde V}\hat{\tilde\Phi}
\end{equation}
with 
\[\doublehat{\tilde V}=\frac{1}{\lambda^2}\begin{pmatrix}
1&0\\0&-1
\end{pmatrix} +\frac{1}{\lambda}\begin{pmatrix}0&\tilde \beta(\tilde y,t)\\\tilde \gamma(\tilde y,t)& 0\end{pmatrix}.\] 
Here 
$\tilde \beta(\tilde y,t)=-2\hat b_2(\tilde y,t)\hat b_1(\tilde y,t)$ and $\tilde \gamma=2\hat b_3(\tilde y,t)\hat b_1^{-1}(\tilde y,t)$,
where $\hat b_j(\tilde y,t)$, $j=1,2,3$ are
 determined evaluating $\hat {\tilde N}(\tilde y,t,\lambda)$ as $\lambda\to 0$, $\lambda\in\mathbb{C}^+$, 
 \begin{equation}\label{tilM-hat-expand}
\hat {\tilde N}(\tilde y,t,\lambda)=-\sqrt{\frac{1}{2}}\left(\ii\begin{pmatrix}
\ii \hat b_1^{-1}(\tilde y,t)& \hat b_1(\tilde y,t)\\ \hat b_1^{-1}(\tilde y,t)&\ii \hat b_1(\tilde y,t)
\end{pmatrix}+\ii\lambda\begin{pmatrix}
\hat b_2(\tilde y,t)&\ii \hat b_3(\tilde y,t)\\\ii \hat b_2(\tilde y,t)&\hat b_3(\tilde y,t)
\end{pmatrix}  \right)
+\ord(\lambda^2).
\end{equation}

\end{proposition}

\begin{proposition}

Let $\tilde \alpha(\tilde y,t)$, $\tilde \beta(\tilde y,t)$, and $\tilde \gamma(\tilde y,t)$ be the functions determined in terms of $\hat{\tilde  n}(\tilde y,t,\lambda)$ as in Propositions \ref{prop-y_} and \ref{prop-t_}. Then they satisfy the following equations:
\begin{subequations}\label{rel_}
\begin{align}\label{rel-a_}
&-\tilde \alpha_{t}+\tilde \beta+\tilde \gamma= 0,\\
\label{rel-b_}
&\tilde \beta_{\tilde y}+ A_1(\tilde \alpha\tilde \beta+1)=0,\\
\label{rel-c_}
&\tilde \gamma_{\tilde y} + A_1(-\tilde \alpha\tilde \gamma+1)=0,
\end{align}
\end{subequations}
which introducing  $\hat {\tilde m}(\tilde y,t)$, $\hat {\tilde u}(\tilde y,t)$, and 
$\hat {\tilde v}(\tilde y,t)$
in terms of $\tilde\alpha$, $\tilde\beta$, and $\tilde\gamma$ by
\begin{equation}\label{hmbbgg_left}
\hat {\tilde m} = \tilde\alpha^{-1},\quad\hat {\tilde u}=\frac{\tilde\gamma-\tilde\beta}{2},\quad 
\hat {\tilde v}=\frac{\tilde\gamma+\tilde\beta}{2},
\end{equation}
 reduce to \eqref{ch_y_}.
\end{proposition}

Now assume that $\hat {\tilde m}>0$ and $\hat {\tilde m}-A_1\in L^1(\mathbb{R}^-)$, and define the inverse change of variable $(\tilde y,t)\mapsto (x,t)$ by
\begin{equation}\label{x_y_left}  x(\tilde y,t):=\tilde y+\int_{-\infty}^{\tilde y}\left(\frac{A_1}{\hat {\tilde m}}-1\right)(\xi,t)\dd \xi+A_1^2 t.
\end{equation}
Then $x_{\tilde y}(\tilde y,t)=\frac{A_1}{\hat {\tilde m}(\tilde y,t)}$, and for each fixed $t$, the map $\tilde y\mapsto x(\tilde y,t)$ is locally invertible. Moreover, by Proposition \ref{prop:hatu_u_til} the functions $u(x,t)=\hat {\tilde u}(\tilde y(x,t),t)$ and $m(x,t)=\hat {\tilde m}(\tilde y(x,t),t)$ satisfy \eqref{mCH-1}. 

\begin{remark}\label{rem:corr_ext} Assume that $r(\lambda)$ and $\tilde r(\lambda)$ are constructed from a given initial function $u_0(x)$. Introduce $\tilde y(y)$ by \eqref{ch_y_til_y}.
The analysis in Appendix~\ref{app:A.6} shows that
\[
\hat N(y,t,\lambda)
=
\hat{\tilde N}(\tilde y(y),t,\lambda)\,P(y,t,\lambda),
\]
where $P$ is given by \eqref{connection_P}.
Moreover, as $\lambda\to\infty$,
\[
P(y,t,\lambda)
=
I+\frac{1}{\lambda}
\begin{pmatrix}
p_{\infty}^1(y,t) & 0\\
0 & -p_{\infty}^1(y,t)
\end{pmatrix}
+O\left(\frac{1}{\lambda^2}\right).
\]
Comparing the corresponding large-$\lambda$ expansions of
$\hat N$ and $\hat{\tilde N}$  yields
\[
\alpha(y,t)
=
\tilde\alpha(\tilde y(y),t).
\]
Moreover, the analysis in Section \ref{app:Eigen_0-l} (cf. the analogous property of
$\tilde s_{11}(\lambda)$ established in \cite{KST22}), yields 
\[
c_{11}(\lambda)
\eul^{\hat f_1(\tilde y(y),t,\lambda)-\hat f_2(y,t,\lambda)}
=
q_0(y,t)+O(\lambda^2),
\qquad \lambda\to0.
\]
Consequently,
\[
P(y,t,\lambda)
=
\begin{pmatrix}
q_0(y,t) & 0\\
0 & q_0^{-1}(y,t)
\end{pmatrix}
+O(\lambda^2),
\qquad \lambda\to0.
\]
Comparing the corresponding small-$\lambda$ expansions  gives
\[
\left(\frac{\gamma\pm\beta}{2}\right)(y,t)
=
\left(\frac{\tilde\gamma\pm\tilde\beta}{2}\right)
(\tilde y(y),t).
\]
\end{remark}
In particular, Remark \ref{rem:corr_ext} implies

\begin{theorem} Assume that $\tilde r(\lambda)=\ord(\frac{1}{\lambda^2})$ as $\lambda\to\infty$. Then the functions
     $\hat {\tilde u}(\tilde y,t)$, $\hat {\tilde v}(\tilde y,t)$, $\hat{\tilde m} ^{-1}(\tilde y,t)$ defined in \eqref{hmbbgg} are $C^1$ in $\tilde y$ and $C^1$ in $t$. 
  
\end{theorem}


\section{Existence of solution of the mCH equation.} \label{sec:6}

\subsection{Existence of the solution of the mCH in the $(y,t)$ scale.}
By Proposition~\ref{mCH_(y,t)}, the mCH equation in the
$(y,t)$-variables takes the form
\begin{subequations}\label{mch_y}
\begin{align}
(\hat m^{-1})_t&=2\hat v,\\
\hat u_y&=A_2\hat v\hat m^{-1},\\
\hat v_y&=A_2(\hat u\hat m^{-1}-1).
\end{align}
\end{subequations}

Moreover, combining Proposition~\ref{reduce} with
\eqref{alpha}, \eqref{M-hat-expand}, and \eqref{hmbbgg}, we obtain
the following reconstruction formulas:
\begin{subequations}\label{hat-u-reconstruction}
\begin{align}
\hat u(y,t)
&=\hat a_1(y,t)\hat a_2(y,t)
  +\hat a_1^{-1}(y,t)\hat a_3(y,t),\\
\hat v(y,t)
&=-\hat a_1(y,t)\hat a_2(y,t)
  +\hat a_1^{-1}(y,t)\hat a_3(y,t),\\
\hat m^{-1}(y,t)
&=2\hat N_{12}(y,t).
\end{align}
\end{subequations}

\subsubsection{Differentiability} By a classical solution of \eqref{mch_y} in the $(y,t)$-variables, we
mean a solution for which $\hat u$, $\hat v$, and $\hat m^{-1}$ are
continuously differentiable with respect to both $y$ and $t$.
To establish the required regularity of the reconstructed
functions, we
combine Proposition \ref{prop_r} and Theorem \ref{thm_diff_hat}, and obtain the following result.

\begin{theorem}\label{thm:sol_diff_y}
     Let $u_0(x)-A_1\in H^{4,2}(\mathbb{R}^-)$, $u_0(x)-A_2\in H^{4,2}(\mathbb{R}^+)$, and $m_0(x):= u_0(x)-u_{0xx}(x)>0$. Then the functions
     $\hat u(y,t)$, $\hat v(y,t)$, $\hat m^{-1}(y,t)$ defined in \eqref{hmbbgg} are $C^1$ in $y$ and $C^1$ in $t$. 
  
\end{theorem}

\subsubsection{Initial condition} Set
\begin{equation}\label{y_0}
y_0(x)
=
x-\frac{1}{A_2}\int_x^{+\infty}
\bigl(m_0(\xi)-A_2\bigr)\dd\xi,
\qquad
m_0=u_0-u_{0xx},
\end{equation}
and denote by $x_0(y)$ its inverse. Since
$y_0'(x)=m_0(x)/A_2>0$, the inverse $x_0$ is well defined. The
initial condition in the $(y,t)$-variables is therefore
\begin{equation}\label{initial-y}
\hat u_0(y)=u_0(x_0(y)),\qquad \hat m_0(y)=m_0(x_0(y)).
\end{equation}

In order to prove that  $\hat u(y,t)$, $\hat v(y,t)$, $\hat m^{-1}(y,t)$ from Theorem \ref{thm:sol_diff_y} 
are consistent with the initial data $u_0$, 
we first construct an RH problem determined directly by  $u_0$.

Let $\Psi_j(x,0,\lambda)$, $j=1,2$, be the Jost solutions associated
with the initial data, and let $c_{11}(\lambda)$ and
$c_{22}(\lambda)$ be the corresponding scattering coefficients.
Define
\begin{equation}\label{P_0-Jost}
\hat Q_0(y,\lambda)
=
-\sqrt{\frac{1}{2}}
\begin{pmatrix}
1&\ii\\
\ii&1
\end{pmatrix}
\begin{cases}
\left(
\dfrac{\Psi_1^{(1)}(x_0(y),0,\lambda)}{c_{11}(\lambda)},
\Psi_2^{(2)}(x_0(y),0,\lambda)
\right)
\eul^{\frac{\ii A_2l_2(\lambda)}{2}y\sigma_3},
&
\lambda\in\mathbb C^+,
\\[4mm]
\left(
\Psi_2^{(1)}(x_0(y),0,\lambda),
\dfrac{\Psi_1^{(2)}(x_0(y),0,\lambda)}{c_{22}(\lambda)}
\right)
\eul^{\frac{\ii A_2l_2(\lambda)}{2}y\sigma_3},
&
\lambda\in\mathbb C^-.
\end{cases}
\end{equation}

The function
$\hat Q_0(y,\lambda)$ satisfies the following \textbf{``initial-value'' RH problem} in the
$L^2$ sense:
find a $2\times2$ matrix-valued function
$\hat Q_0(y,\lambda)$ such that
\[
\hat Q_0(y,\lambda)-I
\in
\dot E^2(\mathbb C\setminus\mathbb R),
\]
and whose non-tangential boundary values satisfy the following jump
condition:
\[
\hat Q_{0+}(y,\lambda)=\hat  Q_{0-}(y,\lambda) \hat G_0(y,\lambda),\qquad
\text{for a.e. }\lambda\in\mathbb R,
\]
where
\[
 \hat G_0(y,\lambda)=\begin{cases}
\eul^{-\frac{\ii A_2 l_2(\lambda) y}{2}\sigma_3}
G_0(\lambda)
\eul^{\frac{\ii A_2 l_2(\lambda) y}{2}\sigma_3} \quad \lambda\in\dot\Sigma_2,\\
-\ii\sigma_1, \quad \lambda\in\dot\Gamma_2,
\end{cases}  
\]
with $G_0(\lambda)$ defined by \eqref{G_0}.

Moreover, 
\begin{enumerate}[(i)]
    \item expanding $\hat Q_0(y,\lambda)$ as $\lambda\to\infty$ using Neumann series, we obtain
\begin{equation}\label{inf_Q_(x)}
    \hat Q_0(y,\lambda)=I+\frac{1}{\lambda}\begin{pmatrix}
    \zeta_1(y)&\eta_1(y)\\
    \eta_1(y)&-\zeta_1(y)
\end{pmatrix}+O(\frac{1}{\lambda^2}),\quad\lambda\to\infty,
\end{equation}
with 
\[
\eta_1(y) = \frac{1}{2 \hat m_0(y)};
\]

\item expanding $\hat Q_0(y,\lambda)$ as $\lambda\to0$,
$\lambda\in\mathbb C^+$, by means of the Neumann series, we obtain
\begin{equation}\label{0_Q_(x)}
    \hat Q_0(y,\lambda)=-\sqrt{\frac{1}{2}}\left(\ii\begin{pmatrix}
\ii q_1^{-1}(y)&  q_1(y)\\ q_1^{-1}(y)&\ii  q_1(y)
\end{pmatrix}+\ii\lambda\begin{pmatrix}
q_2(y)&\ii q_3(y)\\\ii q_2(y)& q_3(y)
\end{pmatrix}\right)+O(\lambda^2),\quad \lambda\to0,\quad\lambda\in\mathbb C^+,
\end{equation}
with $q_j(y)$ satisfying
\[
\hat u_0(y) = q_1(y)q_2(y)+ q_1^{-1}(y)q_3(y).
\]

\end{enumerate}

We now identify the solution of the initial-value RH problem with the
solution of  Right RH problem at $t=0$.

\begin{proposition}\label{prop:initial-condition}
For every $y\in\mathbb R$,
\begin{equation}\label{N-Q0}
\hat N(y,0,\lambda)=\hat Q_0(y,\lambda).
\end{equation}
Consequently,
\begin{equation}\label{initial-m}
\hat m(y,0)=m_0(x_0(y))
\end{equation}
and
\begin{equation}\label{initial-u}
\hat u(y,0)=u_0(x_0(y)).
\end{equation}
\end{proposition}

\begin{proof}
At $t=0$,
\begin{equation*}
\hat f_2(y,0,\lambda)
=
\frac{\ii A_2l_2(\lambda)}{2}y.
\end{equation*}
Therefore, $\hat N(y,0,\lambda)$ and $\hat Q_0(y,\lambda)$ satisfy
the same jump condition on $\mathbb R$. Moreover,
\begin{equation*}
\hat N(y,0,\cdot)-I,\qquad
\hat Q_0(y,\cdot)-I
\in
\dot E^2(\mathbb C\setminus\mathbb R).
\end{equation*}
The uniqueness of the $L^2$ RH problem therefore gives
\eqref{N-Q0}.

Comparing the large-$\lambda$ expansions
\eqref{alpha-M} and \eqref{inf_Q_(x)}, we obtain
$\hat N_{12}(y,0)
=
\eta_1(y)
=
\frac{1}{2\hat m_0(y)}$.
Since
$\hat m^{-1}(y,t)=2\hat N_{12}(y,t)$,
it follows that
\begin{equation*}
\hat m^{-1}(y,0)
=
\hat m_0^{-1}(y),
\end{equation*}
which proves \eqref{initial-m}.

Next, comparing the expansions \eqref{M-hat-expand} and
\eqref{0_Q_(x)} gives
$\hat a_j(y,0)=q_j(y)$,
$j=1,2,3$.
Hence, 
\begin{align*}
\hat u(y,0)
=
u_0(x_0(y)).
\end{align*}
This proves \eqref{initial-u}.
\end{proof}

\subsubsection{Behavior as $y \to +\infty$ }\label{sec:+inf}

The idea is as follows: (i) construct the ``outer parametrix'', which solves the model RH problem obtained by retaining the non-oscillatory jump on $\Gamma_2$; (ii) ``open lenses'' , that is, deform the contour so that the oscillatory factors are moved into regions where they decay, and decompose the reflection coefficient into analytic and remainder parts; (iii) construct ``local parametrices'', which provide a local model for the RH problem near the branch points, and match them with the outer model; (iv) estimate the resulting error terms using Fourier analysis to obtain the required weighted decay.

First we construct an outer parametrix --- a model associated with a non-oscillatory jump on $\Gamma_2$.
Consider the following model Riemann--Hilbert problem: find a $2\times 2$ matrix-valued function $\hat{M}_{mod}(\lambda)$ such that
\begin{equation}
\label{M_mod_E2}
\hat{ M}_{mod}(\lambda)-I
\in
\dot E^2(\mathbb C\setminus \Gamma_2),
\end{equation}
and whose non-tangential boundary values satisfy
\begin{equation}
\label{jump_mod}
\hat{ M}_{mod}^{+}(\lambda)
=
\hat{ M}_{mod}^{-}(\lambda)
(-\ii \sigma_1),
\qquad
\text{for a.e. }\lambda\in\Gamma_2.
\end{equation}
The solution of this Riemann--Hilbert problem is given explicitly by
\begin{equation}
\label{solmod}
\hat{ M}_{mod}(\lambda)
=
\frac{1}{2}
\begin{pmatrix}
\delta_2(\lambda)+\delta_2^{-1}(\lambda)
&
\delta_2(\lambda)-\delta_2^{-1}(\lambda)
\\[2mm]
\delta_2(\lambda)-\delta_2^{-1}(\lambda)
&
\delta_2(\lambda)+\delta_2^{-1}(\lambda)
\end{pmatrix},
\end{equation}
where
\begin{equation}
\label{delta_mod_2}
\delta_2(\lambda)
=
\left(
\frac{\lambda+\frac{1}{A_2}}
     {\lambda-\frac{1}{A_2}}
\right)^{\frac14},
\end{equation}
and the branch is chosen so that
$
\delta_2(\lambda)\to1$ as
$\lambda\to\infty$.

In particular, \begin{equation*}
\label{Mmod-infinity}
\hat M_{mod}(\lambda)
=
I+\frac{\sigma_1}{2A_2\lambda}
+O(\lambda^{-2}),
\qquad
\lambda\to\infty,
\end{equation*}
\begin{equation*}\label{solmod_at_0}
  -\sqrt{\frac{1}{2}}
    \begin{pmatrix}
    1&-\ii\\
    -\ii &1
    \end{pmatrix}
    \hat{ M}_{mod}(\lambda)=\ii\begin{pmatrix}
0& 1\\1&0
\end{pmatrix}+\ii\lambda\begin{pmatrix}
\frac{A_2}{2}&0\\0&\frac{A_2}{2}
\end{pmatrix}+\ord(\lambda^2),\quad \lambda\to 0,\quad \lambda\in\mathbb{C}^+,
\end{equation*}
whereas near $\lambda=\pm\frac{1}{A_2}$,
\begin{equation}
\label{Pinf-right-end}
\hat{M}_{mod}(\lambda)
=
O((\lambda\mp A_2^{-1})^{-1/4}).
\end{equation}

Using $|r|=1$ on $\Sigma_0$, the oscillatory jump on the whole
exterior spectrum $\Sigma_2$ can be written as 
\begin{equation}
\label{UL-factorization}
\hat G=UL,
\qquad
L
=
I+r(\lambda)e^{2\hat f_2}E_{21},
\qquad
U
=
I-\overline{r(\lambda)}e^{-2\hat f_2}E_{12}.
\end{equation}
At the right endpoint $\frac{1}{A_2}$, the ``frozen'' lower- and
upper-triangular factors are
\[
L_{\frac{1}{A_2}}
=
e^{-\widehat f_2\sigma_3}
C_u
e^{\widehat f_2\sigma_3}=
I-e^{2\widehat f_2}E_{21},\qquad U_{\frac{1}{A_2}}
=
e^{-\widehat f_2\sigma_3}
C_\ell
e^{\widehat f_2\sigma_3}=I+e^{-2\widehat f_2}E_{12}
\]
with 
\begin{equation*}
\label{Cul-right}
C_u
=
\begin{pmatrix}
1&0\\
-1&1
\end{pmatrix},
\qquad
C_\ell
=
\begin{pmatrix}
1&1\\
0&1
\end{pmatrix}, \qquad E_{12}
=
\begin{pmatrix}0&1\\0&0\end{pmatrix},
\qquad
E_{21}
=
\begin{pmatrix}0&0\\1&0\end{pmatrix}.
\end{equation*}
Notice that these factors satisfy the relation
\begin{equation*}
C_\ell^{-1} \ii\sigma_1=\ii\sigma_1 C_u,
\end{equation*}
which expresses the algebraic compatibility between the two frozen lens factors and the jump matrix $\ii\sigma_1$ on $\Gamma_2$.

For the subsequent analytic/non-analytic decomposition, it is convenient
to replace the constants $\mp1$ by an explicit analytic model with the
required behavior at both the origin and infinity. Choose $\kappa>0$
sufficiently large so that the poles introduced below lie outside all lens
domains, and define
\begin{equation}
\label{Phi-def}
\phi(k)
=
\frac{1-\frac{4\ii k}{\kappa}}
     {\left(1-\frac{\ii k}{\kappa}\right)^4}.
\end{equation}
By construction,
\begin{equation}
\label{Phi-properties}
\phi(0)=1,
\qquad
\phi'(0)=0,
\qquad
\phi(k)=O(k^{-3}),
\quad k\to\infty,
\end{equation}
while on the real axis $\Phi$ satisfies the Schwarz-type symmetry
\[
\phi(-k)=\overline{\phi(\bar k)},
\]
which, on the real axis, reads as
\[
\phi(-k)=\overline{\phi( k)},
\qquad k\in\mathbb R.
\]

On $\Sigma_2$, we introduce the signed uniformizing coordinate
\begin{equation*}
\label{k-coordinate}
k
=
\mathrm{sign}(\lambda)
\sqrt{\lambda^2-\frac1{A_2^2}}
\end{equation*}
In particular, $k=0_+$ corresponds to $\lambda=A_2^{-1}$ and $k=0_-$ to
$\lambda=-A_2^{-1}$, and
$\frac{d\lambda}{dk}
=
\frac{k}{\lambda(k)}
=
\frac{|k|}{\sqrt{A_2^{-2}+k^2}}$.
Moreover,
$\lambda\mp A_2^{-1}
=
\pm\frac{A_2}{2}k^2+O(k^4)$ as 
 $k\to0_\pm$.
In this coordinate
\begin{equation*}
\label{phase-k-form}
2\hat f_2(\lambda(k))
=
\ii A_2yk
-
\frac{2\ii A_2tk}{A_2^{-2}+k^2},
\end{equation*}
so that
\begin{equation*}
\label{omega-def}
\eul^{2\hat f_2}
=
\eul^{\ii A_2yk}\omega_t(k),
\qquad
\omega_t(k)
=
\eul^{
-\frac{2\ii A_2tk}{A_2^{-2}+k^2}
}.
\end{equation*}

We will open lenses around the two exterior components
$
\left[A_2^{-1},\infty\right)$ and 
$\left(-\infty,-A_2^{-1}\right]$.
The \emph{upper lens} is the region between the real exterior
spectrum and the contour obtained by moving it into the upper
half-plane; the \emph{lower lens} is defined analogously in the
lower half-plane. Thus, there are four lens regions:
$\Omega_{R}^{u},~
\Omega_{R}^{\ell},~
\Omega_{L}^{u},~
\Omega_{L}^{\ell}$.

The lenses are chosen in the $\lambda$-plane so that, under
$k=l_2(\lambda)$, the upper and lower lens regions are mapped into
$\Im k>0$ and $\Im k<0$, respectively. Consequently, uniformly for
$t$ in bounded intervals and sufficiently large $y$,
$e^{2\hat f_2}$ decays on the upper lens boundaries and
$e^{-2\hat f_2}$ on the lower ones. Near the branch points, their
images approach $k=0$ nontangentially, that is,
$|\Im l_2(\lambda)|\ge c_0|l_2(\lambda)|$ for some $c_0>0$.
The lenses and $\kappa$ are also chosen so that the corresponding
$k$-domains avoid the poles of $\phi(\pm k)$ and the singularities
$k=\pm\ii A_2^{-1}$ of $\omega_t(k)$.

For $k\in\mathbb{R}$, define the function
\begin{equation*}
\label{Rsharp}
R^\sharp(k)
=
\begin{cases}
-\phi(k),&k>0,\\
 \phi(k),&k<0.
\end{cases}
\end{equation*}
Notice that $R^\sharp$ is not intended to be the boundary value of
one function analytic in a neighborhood of $k=0$. Instead, we define analytic functions of $\lambda$ in the upper
lens regions by
\begin{equation*}
\label{rsharp-upper}
r_u^\sharp(\lambda)
=
\begin{cases}
-\phi(l_2(\lambda)),
&\lambda\in\Omega_R^u,
\\[1mm]
\phantom{-}\phi(l_2(\lambda)),
&\lambda\in\Omega_L^u,
\end{cases}
\end{equation*}
and in the lower lenses by Schwarz reflection,
\begin{equation*}
\label{rsharp-lower}
r_\ell^\sharp(\lambda)
=
\overline{r_u^\sharp(\bar\lambda)}.
\end{equation*}

Set
\begin{equation*}
\label{R-rho}
R(k)=r(\lambda(k)),
\qquad
\rho(k)=R(k)-R^\sharp(k).
\end{equation*}
Then, the properties in \eqref{Phi-properties}, together with Remark \ref{ratbranchpoints} and the identity $R'(k)
=
r_\lambda(\lambda(k))\frac{k}{\lambda(k)}$ imply that
\begin{equation*}
\label{rho-zero}
\rho(k)=O(k^2),
\qquad
\rho'(k)=O(k),
\qquad
k\to0_\pm.
\end{equation*}
Consequently,
\begin{equation*}
\label{rho-H1}
\rho\in H^1(\mathbb R),
\qquad
h(k):=\frac{\rho(k)}{k}\in H^1(\mathbb R),
\qquad
h(0)=0.
\end{equation*}

Next, absorb the $t$-dependent part of the phase by setting
\[
H_t(k):=h(k)\omega_t(k).
\]
Uniformly for $t$ in bounded $t$-intervals,
\begin{equation*}
\label{Ht-H1}
H_t,\ kH_t\in H^1(\mathbb R),
\qquad
H_t(0)=0.
\end{equation*}

We now introduce a smooth Fourier decomposition of $H_t$ into an analytic part and a small remainder. Consider the Fourier transform of $H_t$
\[
\widehat H_t(\xi)
=
\frac1{2\pi}
\int_{\mathbb R}H_t(k)e^{-\ii k\xi}\,dk.
\]
Choose $\nu\in C^\infty(\mathbb R)$ such that
\[
\nu(s)=0\quad(s\le-3/4),
\qquad
\nu(s)=1\quad(s\ge-1/2),
\]
and put
\[
\nu_y(\xi)=\nu\left(\frac{\xi}{A_2y}\right).
\]
Define
\begin{align*}
    H_a(y,t,k)
&=
\int_{\mathbb R}
\nu_y(\xi)\widehat H_t(\xi)e^{\ii k\xi}\,d\xi,
\qquad \Im k>0,\\
H_r(y,t,k)
&=
\int_{\mathbb R}
(1-\nu_y(\xi))
\widehat H_t(\xi)e^{\ii k\xi}\,d\xi,\qquad k\in \mathbb{R}.
\end{align*}
On the real axis, put
\begin{equation*}
\label{rho-ar}
\rho_a(y,t,k)
=
k\omega_t(k)^{-1}H_a(y,t,k),
\qquad
\rho_r(y,t,k)
=
k\omega_t(k)^{-1}H_r(y,t,k).
\end{equation*}
Since $H_t=H_a+H_r$, we have
\begin{equation}
\label{R-decomposition}
R(k)
=
R^\sharp(k)+\rho_a(y,t,k)+\rho_r(y,t,k),\qquad k\in\mathbb{R}.
\end{equation}

The analytic function used in the upper lenses is
\begin{equation*}
\label{rho-upper}
\rho_{a,u}(y,t,\lambda)
=
l_2(\lambda)
\omega_t(l_2(\lambda))^{-1}
H_a(y,t,l_2(\lambda)),
\end{equation*}
whereas in the lower lenses we set
\begin{equation*}
\label{rho-lower}
\rho_{a,\ell}(y,t,\lambda)
=
\overline{\rho_{a,u}(y,t,\bar\lambda)}.
\end{equation*}

Since
$1-\nu_y$ is supported in $\xi\le-cy$, Plancherel's theorem gives
\begin{equation}
\label{remainder-estimates}
\|H_r\|_2+\|kH_r\|_2=O(y^{-1}),
\qquad
\|H_r\|_\infty+\|\rho_r\|_\infty=O(y^{-1/2}).
\end{equation}
Furthermore, since $H_t(0)=0$, for every $0<\beta<1/2$,
\begin{equation}
\label{analytic-estimate}
\left|
H_a(y,t,k)e^{\ii A_2yk}
\right|
\le
C\left(y^{-1/2}+|k|^\beta\right)e^{-cy|k|}
\end{equation}
on the upper lens boundaries near $k=0$; the lower estimate follows
by Schwarz reflection.

We now perform the lens transformation. In accordance with
\eqref{UL-factorization}, define
\begin{align*}
L^{an}(\lambda)
&=
I+
\bigl(
r_u^\sharp(\lambda)
+
\rho_{a,u}(\lambda)
\bigr)
e^{2\hat f_2}E_{21},
\\
U^{an}(\lambda)
&=
I-
\bigl(
r_\ell^\sharp(\lambda)
+
\rho_{a,\ell}(\lambda)
\bigr)
e^{-2\hat f_2}E_{12}.
\end{align*}
Set
\begin{equation*}
\label{S-def}
S=
\begin{cases}
\hat N(L^{an})^{-1},
&\text{in an upper lens},
\\[1mm]
\hat N U^{an},
&\text{in a lower lens},
\\[1mm]
\hat N,
&\text{elsewhere}.
\end{cases}
\end{equation*}
On $\Sigma_2$, the remaining jump is therefore
\begin{equation}
\label{JS-real}
J_S
=
\left(
I-\overline{\rho_r(k)}
e^{-2\hat f_2}E_{12}
\right)
\left(
I+\rho_r(k)e^{2\hat f_2}E_{21}
\right).
\end{equation}

We next construct the local parametrix near $\lambda=\pm\frac{1}{A_2}$.
Let
\[
D_R=D(A_2^{-1},\varepsilon),
\qquad
D_L=D(-A_2^{-1},\varepsilon)
\]
be fixed sufficiently small disks and denote by $D_j^u$, $D_j^0$, and $D_j^\ell$ the upper, central,
and lower sectors of $D_j$, respectively. For $j=R,L$, define
\[
L_j^\sharp
=
I+r_u^\sharp e^{2\hat f_2}E_{21},
\qquad
U_j^\sharp
=
I-r_\ell^\sharp e^{-2\hat f_2}E_{12},
\]
and
\begin{equation}
\label{local-parametrix}
P^{(j)}(\lambda)
=
\begin{cases}
\hat M_{mod}(\lambda)L_j^\sharp(\lambda),
&\lambda\in D_j^u,
\\[1mm]
\hat M_{mod}(\lambda),
&\lambda\in D_j^0,
\\[1mm]
\hat M_{mod}(\lambda)(U_j^\sharp(\lambda))^{-1},
&\lambda\in D_j^\ell.
\end{cases}
\end{equation}

Then $P^{(j)}$ (a local parametrix) solves the corresponding local RH problem in $D_j$:
it is analytic away from the portions of the lens boundaries and
$\Gamma_2$ contained in $D_j$, has the prescribed jumps
$L_j^\sharp$ and $U_j^\sharp$ on the upper and lower lens boundaries,
respectively, and satisfies
$P^{(j)}_+(\lambda)
=
P^{(j)}_-(\lambda)(-\ii\sigma_1)$,
$\lambda\in\Gamma_2\cap D_j$.
Since the additional factors are
bounded at the branch points, the local parametrices have the same
quarter-root behavior as $\hat M_{mod}$.

We finally verify the matching with the outer model.  Set
\[
\Gamma_j^u=\overline{\partial D_j\cap D_j^u},
\qquad
\Gamma_j^0=\overline{\partial D_j\cap D_j^0},
\qquad
\Gamma_j^\ell=\overline{\partial D_j\cap D_j^\ell},
\qquad j=R,L.
\]
On the central arc,
\[
P^{(j)}(\lambda)\hat M_{mod}(\lambda)^{-1}=I,
\qquad
\lambda\in\Gamma_j^0.
\]
By the choice of the lens contours, the upper and lower arcs are
contained strictly in the corresponding decay regions. Hence, for
some $d_j>0$,
\[
\Im l_2(\lambda)\ge d_j
\quad\text{on }\Gamma_j^u,
\qquad
\Im l_2(\lambda)\le-d_j
\quad\text{on }\Gamma_j^\ell.
\]
Since
\[
2\hat f_2
=
\ii A_2y\,l_2(\lambda)
-
\frac{2\ii A_2t\,l_2(\lambda)}
     {A_2^{-2}+l_2(\lambda)^2},
\]
the second term is uniformly bounded on these fixed arcs for $t$ in
a bounded interval. Therefore
\[
\Re(2\hat f_2)\le -A_2d_jy+C_T
\quad\text{on }\Gamma_j^u,
\qquad
\Re(-2\hat f_2)\le -A_2d_jy+C_T
\quad\text{on }\Gamma_j^\ell.
\]
Thus, for sufficiently large $y$,
\[
|e^{2\hat f_2}|\le Ce^{-cy}
\quad\text{on }\Gamma_j^u,
\qquad
|e^{-2\hat f_2}|\le Ce^{-cy}
\quad\text{on }\Gamma_j^\ell.
\]

For example, on the upper arc of $\partial D_R$,
\[
P^{(R)}\hat M_{mod}^{-1}-I
=
-\phi(l_2(\lambda))e^{2\hat f_2}
\hat M_{mod}E_{21}\hat M_{mod}^{-1},
\]
and the analogous formula on the lower arc contains
$e^{-2\hat f_2}$. Since $\partial D_j$ is separated from the branch
point, the factors involving $\phi$ and $\hat M_{mod}$ are uniformly
bounded there. Consequently,
\begin{equation}
\label{matching-local}
P^{(j)}(\lambda)\hat M_{mod}(\lambda)^{-1}
=
I+O(e^{-cy}),
\qquad
\lambda\in\partial D_j,
\qquad j=R,L,
\end{equation}
uniformly for $t$ in bounded intervals.

Define the global parametrix
\[
P(\lambda)
=
\begin{cases}
P^{(R)}(\lambda),
&\lambda\in D_R,
\\
P^{(L)}(\lambda),
&\lambda\in D_L,
\\
\hat M_{mod}(\lambda),
&\lambda\notin
\overline{D_R\cup D_L},
\end{cases}
\]
and set
\begin{equation*}
\label{error-def}
M^{(err)}=SP^{-1}.
\end{equation*}

Inside the endpoint disks, the local parametrices reproduce exactly
the $r^\sharp$-part of the lens jumps. Since $E_{21}^2=E_{12}^2=0$,
\[
L^{an}(L_j^\sharp)^{-1}
=
I+\rho_{a,u}e^{2\hat f_2}E_{21},
\qquad
U^{an}(U_j^\sharp)^{-1}
=
I-\rho_{a,\ell}e^{-2\hat f_2}E_{12}.
\]
Hence the corresponding error jumps are
\begin{equation*}
\label{error-upper}
J_{err}-I
=
\hat M_{mod}
\left(
\rho_{a,u}e^{2\hat f_2}E_{21}
\right)
\hat M_{mod}^{-1},
\end{equation*}
on the upper lens boundary, and
\begin{equation*}
\label{error-lower}
J_{err}-I
=
-\hat M_{mod}
\left(
\rho_{a,\ell}e^{-2\hat f_2}E_{12}
\right)
\hat M_{mod}^{-1},
\end{equation*}
on the lower lens boundary.

Near either branch point,
\[
\hat M_{mod}E_{ij}\hat M_{mod}^{-1}=O(|k|^{-1}),
\qquad
|d\lambda|\le C|k|\,|dk|.
\]
Since
\[
\rho_{a,u}e^{2\hat f_2}
=
kH_a(k)e^{\ii A_2yk},
\]
the factor $k$ cancels the singularity of the outer-model
conjugation. Hence, by \eqref{analytic-estimate},
\[
|J_{err}-I|
\le
C\left(y^{-1/2}+|k|^\beta\right)e^{-cy|k|},
\qquad 0<\beta<\frac12.
\]
Therefore
\[
\|J_{err}-I\|_\infty=O(y^{-\beta}),
\]
while, with $s=|k|$,
\[
\|J_{err}-I\|_2^2
\le
C\int_0^\varepsilon
\left(y^{-1/2}+s^\beta\right)^2e^{-2cys}s\,ds
=
O(y^{-2-2\beta}),
\]
and thus
\begin{equation}
\label{lens-jump-estimates}
\|J_{err}-I\|_\infty=O(y^{-\beta}),
\qquad
\|J_{err}-I\|_2=O(y^{-1-\beta}).
\end{equation}
The lower lens boundaries are treated by Schwarz reflection, and
away from the branch points, the lens jumps are exponentially small.

On the real exterior spectrum,
\eqref{JS-real} gives the corresponding linear terms
$\rho_r e^{2\hat f_2}E_{21},
$ $
-\overline{\rho_r}e^{-2\hat f_2}E_{12}$,
and a quadratic term. Since
$\rho_r e^{2\hat f_2}
=
kH_r e^{\ii A_2yk}$,
the same factor $k$ cancels the endpoint singularity. Using
\eqref{remainder-estimates} and
$d\lambda\asymp |k|\,dk$, we obtain
\begin{equation}
\label{real-jump-estimates}
\|J_{err}-I\|_\infty
=
O(y^{-1/2}),
\qquad
\|J_{err}-I\|_2
=
O(y^{-1}).
\end{equation}
Together with \eqref{matching-local} and \eqref{lens-jump-estimates}, \eqref{real-jump-estimates} gives globally
\begin{equation*}
\label{global-jump-estimates}
\|J_{err}-I\|_2=O(y^{-1}),
\qquad
\|J_{err}-I\|_\infty=O(y^{-\beta}),
\qquad
0<\beta<\frac12.
\end{equation*}

Set
\[
w=J_{err}-I.
\]
Since
\[
\|w\|_{L^2(\Sigma_{err})}=O(y^{-1}),
\qquad
\|w\|_{L^\infty(\Sigma_{err})}=O(y^{-\beta}),
\qquad
0<\beta<\frac12,
\]
the small-norm RH theory applies for sufficiently large $y$. Hence
\[
\|\mu-I\|_2=O(y^{-1}),
\qquad
M^{(err)}(\lambda)
=
I+\frac1{2\pi\ii}
\int_{\Sigma_{err}}
\frac{\mu(s)w(s)}{s-\lambda}\,ds,
\]
and
\[
M_1^{(err)}
=
-\frac1{2\pi\ii}
\left(
\int_{\Sigma_{err}}w\,ds
+
\int_{\Sigma_{err}}(\mu-I)w\,ds
\right).
\]
The nonlinear term satisfies
\[
\left|
\int_{\Sigma_{err}}(\mu-I)w\,ds
\right|
\le
\|\mu-I\|_2\|w\|_2
=
O(y^{-2}),
\]
so its contribution belongs to $L^{2,1}(Y,\infty)$.

It remains to estimate the linear term. After deforming the analytic
lens contributions to the real exterior spectrum, $\rho_a$ and
$\rho_r$ recombine to $\rho$. Using
$d\lambda=\lambda'(k)\,dk$, set
\[
B_{21}(k)
=
\lambda'(k)\hat M_{mod}E_{21}\hat M_{mod}^{-1},
\qquad
B_{12}(k)
=
\lambda'(k)\hat M_{mod}E_{12}\hat M_{mod}^{-1}.
\]
The linear terms are therefore Fourier integrals
\[
F_{21}(y,t)
=
\int_{\mathbb R}
Q_{21}(k,t)e^{\ii A_2yk}\,dk,
\qquad
Q_{21}
=
B_{21}\rho\,\omega_t,
\]
and
\[
F_{12}(y,t)
=
\int_{\mathbb R}
Q_{12}(k,t)e^{-\ii A_2yk}\,dk,
\qquad
Q_{12}
=
-B_{12}\overline{\rho}\,\omega_t^{-1}.
\]
Since $B_{ij}$ are bounded and piecewise $W^{1,\infty}$ and
$\rho(0)=0$, we have
$Q_{21},Q_{12}\in H^1(\mathbb R).$
The weighted Plancherel inequality then gives
\[
\int_Y^\infty y^2|F_{ij}(y,t)|^2\,dy
\le
C\|Q_{ij}'(\cdot,t)\|_2^2<\infty.
\]
The quadratic real-axis term is $O(y^{-2})$, while the remaining
non-real contributions are exponentially small. Consequently,
\begin{equation}
\label{Merr-weighted}
M_1^{(err)}(\cdot,t)\in L^{2,1}(Y,\infty).
\end{equation}

Finally,
\[
\hat N(\lambda)
=
I+\frac1\lambda
\left(
M_1^{(err)}+\frac{\sigma_1}{2A_2}
\right)
+o(\lambda^{-1}),
\]
and the reconstruction formula yields
\[
\frac1{\hat m(y,t)}-\frac1{A_2}
=
2(M_1^{(err)}(y,t))_{12}.
\]
Therefore, applying the same argument to the reconstruction formulas for
$\hat u$ and $\hat v$ yields the following proposition.

\begin{proposition}\label{prop:beh_hat_m_inf}
    Let $u_0(x)-A_1\in H^{6,3}(\mathbb{R}^+)$, $u_0(x)-A_2\in H^{6,3}(\mathbb{R}^+)$, and $m_0(x):= u_0(x)-u_{0xx}(x)>0$. Then
\begin{align}
   & \frac1{\hat m(\cdot,t)}-\frac1{A_2}
\in L^{2,1}(\mathbb{R^+})
\subset L^1(\mathbb{R^+}),\\
&    \hat u(\cdot,t)-A_2\in L^{2,1}(\mathbb{R^+}),\qquad \hat v(\cdot,t)
\in L^{2,1}(\mathbb{R^+}).
\end{align}
\end{proposition}

\subsubsection{Behavior as $y \to -\infty$ }

In order to study the behavior of the solution as $y\to - \infty$, we will employ the connection between Right and Left RH problems (see Remark \ref{rem:connection}). Namely, we consider the following model Riemann--Hilbert problem associated with the non-decaying background jump on $\dot\Gamma_1$: find a $2\times 2$ matrix-valued function $\check{M}_{mod}(\lambda)$ such that
\begin{equation}
\label{M_mod_E2_left}
\check{ M}_{mod}(\lambda)-I
\in
\dot E^2(\mathbb C\setminus \Gamma_1),
\end{equation}
and whose non-tangential boundary values satisfy
\begin{equation}
\label{jump_mod_left}
\check{ M}_{mod}^{+}(\lambda)
=
\check{ M}_{mod}^{-}(\lambda)
(-\ii \sigma_1)
\qquad
\text{for a.e. }\lambda\in\Gamma_1.
\end{equation}

The solution of this Riemann--Hilbert problem is given explicitly by
\begin{equation}
\label{solmod_left}
\check{ M}_{mod}(\lambda)
=
\frac{1}{2}
\begin{pmatrix}
\delta_1(\lambda)+\delta_1^{-1}(\lambda)
&
\delta_1(\lambda)-\delta_1^{-1}(\lambda)
\\[2mm]
\delta_1(\lambda)-\delta_1^{-1}(\lambda)
&
\delta_1(\lambda)+\delta_1^{-1}(\lambda)
\end{pmatrix},
\end{equation}
where
\begin{equation}
\label{delta_mod_1}
\delta_1(\lambda)
=
\left(
\frac{\lambda+\frac{1}{A_1}}
     {\lambda-\frac{1}{A_1}}
\right)^{\frac14},
\end{equation}
and the branch is chosen so that
$
\delta_1(\lambda)\to1$ as 
$\lambda\to\infty$.

In particular, \begin{equation*}
\label{Mmod-infinity_check}
\check M_{mod}(\lambda)
=
I+\frac{\sigma_1}{2A_1\lambda}
+O(\lambda^{-2}),
\qquad
\lambda\to\infty,
\end{equation*}
\begin{equation*}\label{solmod_at_0_check}
  -\sqrt{\frac{1}{2}}
    \begin{pmatrix}
    1&-\ii\\
    -\ii &1
    \end{pmatrix}
    \check{ M}_{mod}(\lambda)=\ii\begin{pmatrix}
0& 1\\1&0
\end{pmatrix}+\ii\lambda\begin{pmatrix}
\frac{A_1}{2}&0\\0&\frac{A_1}{2}
\end{pmatrix}+\ord(\lambda^2),\quad \lambda\to 0,\quad \lambda\in\mathbb{C}^+,
\end{equation*}
whereas near $\lambda=\pm\frac{1}{A_1}$,
\begin{equation*}
\label{Pinf-right-end_check}
\check{M}_{mod}(\lambda)
=
O((\lambda\mp A_1^{-1})^{-1/4}).
\end{equation*}

Proceeding similarly to the right problem, we obtain the following propositions.

\begin{proposition}\label{prop:beh_hat_m_-inf}
    Let $u_0(x)-A_1\in H^{6,3}(\mathbb{R}^-)$, $u_0(x)-A_2\in H^{6,3}(\mathbb{R}^-)$, and $m_0(x):= u_0(x)-u_{0xx}(x)>0$. Then
    \begin{align}
         &\frac1{\hat m(\cdot,t)}-\frac1{A_1}
\in L^{2,1}(\mathbb{R^-})
\subset L^1(\mathbb{R^-}),\\
    &\hat u(\cdot,t)-A_1\in L^{2,1}(\mathbb{R^-}),\qquad\hat v(\cdot,t)
\in L^{2,1}(\mathbb{R^-}).
    \end{align}
\end{proposition}

\begin{theorem}
    Let  $u_0(x)-A_1\in H^{6,3}(\mathbb{R}^-)$, $u_0(x)-A_2\in H^{6,3}(\mathbb{R}^+)$, and $m_0(x):= u_0(x)-u_{0xx}(x)>0$, and assume that the scattering data associated with $u_0$ satisfy Assumption~\ref{ass:scattering}.
    Consider  Right RH problem \eqref{right_RH_norm}--\eqref{G_right} (with $r$ corresponding to $u_0(x)$). Then for all $(y,t)\in\mathbb R\times[0,\infty)$,  Right RH problem \eqref{right_RH_norm}--\eqref{G_right} has a unique solution $\hat N(y,t,\lambda)$. 
    
    Moreover, 
    \[
   \hat u(y,t):=\hat a_1(y,t)\hat a_2(y,t)+\hat a_1^{-1}(y,t)\hat a_3(y,t),\quad  \hat v(y,t):=-\hat a_1(y,t)\hat a_2(y,t)+\hat a_1^{-1}(y,t)\hat a_3(y,t), \quad \hat m := \alpha^{-1}(y,t),
    \]
where $\hat a_i(y,t)$, $j=1,2,3$ are taken from the expansion of $\hat N(y,t,\lambda)$ near $\lambda=0_+$,
\[\hat N(y,t,\lambda)=-\sqrt{\frac{1}{2}}\left(\ii\begin{pmatrix}
\ii \hat a_1^{-1}(y,t)& \hat a_1(y,t)\\ \hat a_1^{-1}(y,t)&\ii \hat a_1(y,t)
\end{pmatrix}+\ii\lambda\begin{pmatrix}
\hat a_2(y,t)&\ii \hat a_3(y,t)\\\ii \hat a_2(y,t)&\hat a_3(y,t)
\end{pmatrix}  \right)
+\ord(\lambda^2),
\]
and $\alpha(y,t)$ is taken from the expansion of $\hat N(y,t,\lambda)$ near $\lambda=\infty$,
\[
\hat N(y,t,\lambda)= I+\frac{1}{2\lambda}
\begin{pmatrix}* & \alpha(y,t)\\\alpha(y,t)&*\end{pmatrix}+\ord\left(\frac{1}{\lambda^2}\right), \qquad \lambda\in\mathbb{C}^+,
\]
constitute a unique global solution of the Cauchy problem to the mCH equation in the $(y,t)$-scale. Namely, 
\begin{enumerate}
    \item they satisfy the system of equations \eqref{ch_y} in the classical sense,

    \item $\hat u(y,t)$ satisfies the initial condition $\hat u(y,0)=u_0(x_0(y))$, where $x_0=y_0^{-1}$ is the inverse transformation introduced
in Section \ref{y_0},

    \item for all $t\geq 0$, $\frac1{\hat m(\cdot,t)}-\frac1{A_1}\in L^{1}(\mathbb{R}^-)$ and $\frac1{\hat m(\cdot,t)}-\frac1{A_2}\in L^{1}(\mathbb{R}^+)$.
\end{enumerate}
    
\end{theorem}

\subsection{Existence of the solution of the mCH in the $(x,t)$ scale.}

Let $\alpha$, $\beta$, and $\gamma$ be determined from the solution of  Right RH problem \eqref{right_RH_norm}--\eqref{G_right}, with $\alpha$ given by \eqref{alpha} and $\beta$, $\gamma$ given by \eqref{beta-gamma}.


Introduce the change of variables by the following formula
\begin{equation}\label{x-from-y}
  x(y,t)=y-A_2\int_y^{\infty}\left(\alpha(\xi,t)-\frac{1}{A_2}\right)d\xi+A_2^2 t.  
\end{equation}
It is a straightforward calculation to see that
\[x_y(y,t)=A_2\alpha(y,t).\]
Moreover, Equations \eqref{mch_y} together with Proposition \ref{prop:beh_hat_m_inf} imply
\[x_t(y,t)=\hat u^2(y,t)-\hat v^2(y,t).\]
Notice also that $x(y,t)\to \infty$ as $y\to\infty$ for each fixed $t$.

\begin{remark}
There is an equivalent representation of the change of variables in terms of
the coefficient $\hat a_1(y,t)$ arising from the small-$\lambda$ expansion of
the RH solution. Define
\[
\tilde x(y,t)
=
y-2\ln \hat a_1(y,t)+A_2^2t.
\]
Since $\hat m^{-1}=\alpha$, relation \eqref{alpha_inf} gives
\[
\tilde x_y(y,t)
=
1-2\frac{\hat a_{1y}(y,t)}{\hat a_1(y,t)}
=
A_2\alpha(y,t).
\]
Hence
\[
\frac{\partial}{\partial y}
\bigl(\tilde x(y,t)-x(y,t)\bigr)=0,
\]
so, for each fixed $t$, the difference $\tilde x(y,t)-x(y,t)$ is independent
of $y$.
By the asymptotic analysis in Section~\ref{sec:+inf},
\[
\hat a_1(y,t)\to1,
\qquad y\to+\infty,
\]
and therefore
\[
\tilde x(y,t)-y-A_2^2t\to0,
\qquad y\to+\infty.
\]
The definition of $x(y,t)$ and
$\alpha(\cdot,t)-\frac1{A_2}\in L^1(\mathbb R^+)$
give the same normalization,
\[
x(y,t)-y-A_2^2t\to0,
\qquad y\to+\infty.
\]
Consequently,
\[
\tilde x(y,t)\equiv x(y,t).
\]
\end{remark}

We now discuss the global invertibility of the transformation
$y\mapsto x(y,t)$. 
It follows from \eqref{x-from-y} that
\begin{equation}\label{x-asymp-plus}
x(y,t)
=
y+A_2^2t+o(1),
\qquad y\to+\infty,
\end{equation}
whereas
\begin{equation}\label{x-asymp-minus}
x(y,t)
=
\frac{A_2}{A_1}y+O(1),
\qquad y\to-\infty.
\end{equation}
In particular,
\[
x(y,t)\to+\infty
\quad\text{as }y\to+\infty,
\qquad
x(y,t)\to-\infty
\quad\text{as }y\to-\infty.
\]

Therefore, if
\begin{equation}\label{q-positive}
\alpha>0,
\qquad (y,t)\in\mathbb R\times[0,T),
\end{equation}
then for each fixed $t\in[0,T)$, the map
$y\longmapsto x(y,t)$
is strictly increasing. Together with
\eqref{x-asymp-plus}--\eqref{x-asymp-minus}, this shows that it is a
$C^1$-diffeomorphism from $\mathbb R$ onto $\mathbb R$.

In this way, we arrive at the following theorem.

\begin{theorem}\label{main:theorem} 
Assume that the initial data $u_0$ satisfies the following conditions:

\begin{enumerate}[(i)]
    \item 
    $u_0(x)-A_1\in H^{6,3}(\mathbb{R}^-)$, $u_0(x)-A_2\in H^{6,3}(\mathbb{R}^+)$, and $m_0(x):= u_0(x)-u_{0xx}(x)>0$;
    \item 
    the scattering data associated with $u_0$ satisfy Assumption~\ref{ass:scattering}.
    \end{enumerate}

Let $\hat N(y,t,\lambda)$ be a solution of  Right RH problem \eqref{right_RH_norm}--\eqref{G_right}, whose scattering data are determined by the initial data $u_0(x)$ through the direct scattering transform. Let $\alpha(y,t)$ be determined from $\hat N(y,t,\lambda)$ via \eqref{alpha}, and let $\beta(y,t)$ and $\gamma(y,t)$ be given by \eqref{beta-gamma}. Assume that
    \begin{enumerate}[(iii)]
    \item 
    $\alpha(y,t)$ satisfies $ \alpha(y,t)>0$ for $(y,t)\in\mathbb R\times[0,T)$, $T\leq+\infty$.
    \end{enumerate}
Then,  for every $t\in[0,T]$, the function $x(y,t)$ defined by
\eqref{x-from-y} is a $C^1$-diffeomorphism of $\mathbb R$ onto
$\mathbb R$. Moreover, its inverse $y=y(x,t)$ determines a classical solution in $\mathbb{R}\times[0,T)$ (in the sense of Definition \ref{def:cl})
of the mCH Cauchy problem \eqref{mCH-1}--\eqref{mCH-bc} by 
\[
u(x,t)=\hat u(y(x,t),t),
\qquad
m(x,t)=\hat m(y(x,t),t),
\]
where $\hat u(y,t)$ and $\hat m(y,t)$ are defined via \eqref{hmbbgg} in terms of $\alpha(y,t)$, $\beta(y,t)$, and $\gamma(y,t)$.

In particular, if
\[
 \alpha(y,t)>0
\qquad
\text{for all }(y,t)\in\mathbb R\times[0,\infty),
\]
then the reconstructed solution exists globally in the $(x,t)$ scale.

\end{theorem}


\section{Concluding Remarks}

Geometrically, the condition $\alpha(y,t)>0$ is precisely the condition that
the change of variables $y\mapsto x(y,t)$  remains strictly monotone and hence invertible. Its
failure corresponds to the vanishing of the derivative $x_y$ and signals wave breaking: if $\alpha(y,t^*)$ vanishes,
then the solution $u$ remains bounded but $u_{xx}$  blows up in a finite time $t^*$.

At $t=0$, the condition $\alpha(y,0)>0$ is equivalent to $m_0(x)>0$, which
is one of the standing assumptions on the initial data. 
Then Theorem~\ref{main:theorem} presents a condition (expressed in terms of the solution of a globally solvable Riemann--Hilbert problem) that provides the existence of the solution for a certain time interval. 
An important open
problem is to identify checkable conditions on the initial data which
guarantee that $\alpha(y,t)$ remains positive for all time.

On the other hand, the failure of the positivity condition $\alpha(y,t)>0$ leads to a
different class of questions concerning wave breaking and continuation
beyond the first breaking time. In particular, the loss of invertibility
need not be permanent. One possible scenario is that there exist two times
$t_1<t_2$ such that the map $y\mapsto x(y,t)$ is one-to-one for $t<t_1$,
loses injectivity for $t_1<t<t_2$, and becomes one-to-one again for
$t>t_2$. During the intermediate time interval, several values of $y$ may
correspond to the same physical point $x$, so that the parametric
representation develops several branches and no longer determines a
single-valued classical solution in the physical variables. At a later
time, these branches may recombine and the map $y\mapsto x(y,t)$ may become
invertible again, recovering a single physical branch.

Such a scenario raises the question of whether the solution can be
continued through the whole breaking interval $t_1<t<t_2$, despite the
temporary loss of invertibility of the hodograph map. One may ask whether
the parametric evolution itself remains meaningful during this interval,
whether a suitable single-valued continuation can be selected from the
multiple branches, for instance in an appropriate weak sense, and whether
such a continuation is compatible with the single-valued solution that is
recovered after $t_2$.

\section{Acknowledgment}

This research was funded in part by the Austrian Science Fund (FWF), grant no. 10.55776/ESP691, and the Research Council of Norway, project 361083 (Ukraina-NASTRAN cooperation). IK and DS acknowledge the support from Ukrainian-Austrian Joint Programme of Scientific and Technological Cooperation Project (supported by OeAD, Project WTZ No. UA 08/2025 and Ministry of Education and Science of Ukraine, project no. 0125U003485).


\appendix
\section{Riemann-Hilbert problems  with inner cuts}\label{app:A}

It was shown in \cite{KST22} that, provided a solution of the Cauchy
problem \eqref{mCH1-ic}--\eqref{mCH-bc} exists, it can be recovered
from the solution of an associated Riemann--Hilbert problem whose data
are determined by $u_0(x)$. The construction of the RH problem in
\cite{KST22} employs the square root functions
\begin{equation*}
  k_j(\lambda) = \sqrt{\lambda^2 - \frac{1}{A_j^2}}, \qquad j = 1, 2,
\end{equation*}
defined with branch cuts along the half-lines
$\Sigma_j = (-\infty, -\frac{1}{A_j}] \cup [\frac{1}{A_j}, \infty)$
(henceforth, \emph{outer cuts}).

In the present paper, we adopt an alternative choice of these roots,
denoted $l_j(\lambda)$, with branch cuts placed instead along the
segments $[-\frac{1}{A_j}, \frac{1}{A_j}]$ (\emph{inner cuts}). This
choice is motivated by the analytic structure of the Jost solutions of the Lax pair: with inner cuts, two columns of each Jost solution are
analytic in the upper half-plane, and the remaining two are analytic in the lower half-plane, a property more directly aligned with standard usage in the literature.

The resulting RH problem is constructed along the same lines as in
\cite{KST22}, and we provide the relevant details below, restricting attention to the solitonless case. Since we want to control the behavior of the solution both as $x \to -\infty$ and as
$x \to +\infty$, we present two formulations of the RH problem, which
we refer to as the \emph{right} and the \emph{left} RH problem,
respectively.

\begin{notations*} \begin{itemize}
\item $l_j(\lambda)=\sqrt{\lambda^2-\frac{1}{A_j^2}}$, $j=1,2$, where the branch cut for $l_j$ is the segment $[-\frac{1}{A_j},\frac{1}{A_j}]$, and $l_j(\lambda) \sim \lambda $
as $\lambda\to\infty$. 
    \item 
We introduce the following notations for various intervals of the real axis:
\[\Sigma_j=(-\infty,-\frac{1}{A_j}]\cup[\frac{1}{A_j},\infty),\qquad\dot{\Sigma}_j=(-\infty,-\frac{1}{A_j})\cup(\frac{1}{A_j},\infty),\]
\[\Sigma_0=[-\frac{1}{A_1},-\frac{1}{A_2}]\cup[\frac{1}{A_2},\frac{1}{A_1}],\qquad\dot{\Sigma}_0=(-\frac{1}{A_1},-\frac{1}{A_2})\cup(\frac{1}{A_2},\frac{1}{A_1}),\]
\[\Gamma_j=[-\frac{1}{A_j},\frac{1}{A_j}],\qquad \dot{\Gamma}_j=(-\frac{1}{A_j},\frac{1}{A_j}).\]

\item For $\lambda\in \Gamma_j$ we denote by $\lambda_+$ ($\lambda_-$) the point of the upper (lower) side of $\Gamma_j$ (i.e. $\lambda_\pm=\lim_{\epsilon\downarrow 0}\lambda\pm\ii\epsilon$).

\end{itemize}

\end{notations*}

Observe the following symmetry relations for $l_j(\lambda)$: 
\begin{subequations}\label{sym_l_i}
\begin{alignat}{4}\label{sym_l_i-a}
l_j(-\lambda)&=-l_j(\lambda),\quad &\lambda\in\mathbb{C}\setminus\Gamma_j,\\ \label{sym_l_i-b}
l_j(\lambda_+)&=l_j((-\lambda)_+),\quad &\lambda\in\Gamma_j,\\\label{sym_l_i-c}
\overline{l_j(\overline{\lambda})}&=l_j(\lambda),\quad &\lambda\in\mathbb{C}\setminus\Gamma_j,\\ \label{sym_l_i-d}
\overline{l_j(\lambda_+)}&=-l_j(\lambda_+),\quad &\lambda\in\Gamma_j,\\
\label{sym_l_i-e}
l_j(\lambda_-)&=-l_j(\lambda_+),\quad &\lambda\in\Gamma_j.
\end{alignat}
\end{subequations}

The Lax pair for the mCH equation \eqref{mCH-1}  has the form \cite{Q06}:
\begin{subequations}\label{Lax_App}
\begin{alignat}{4} \label{Lax-x_App}
    \Psi_x(x,t,\lambda)&=U(x,t,\lambda)\Psi(x,t,\lambda), \\
    \Psi_t(x,t,\lambda)&=V(x,t,\lambda)\Psi(x,t,\lambda), \label{Lax-t_App}
\end{alignat}
where  the coefficients $U$ and $V$ are  given by
\begin{alignat}{4} \label{U_App}
U&=\frac{1}{2}\begin{pmatrix} -1 & \lambda  m \\
-\lambda m & 1\end{pmatrix},\\ \label{V_App}
V&=\begin{pmatrix}\lambda^{-2}+\frac{u^2-u_x^2}{2} &
-\lambda^{-1}(u-u_x)-\frac{\lambda(u^2-u_x^2)m}{2}\\
\lambda^{-1}(u+u_x)+\frac{\lambda(u^2- u_x^2)m}{2} &
		 -\lambda^{-2}-\frac{u^2-u_x^2}{2}\end{pmatrix},
\end{alignat}
\end{subequations}
with $m(x,t)= u(x,t)-u_{xx}(x,t)$. 

As in Section 2 of \cite{KST22}, we introduce two gauge transformations associated with the limits $x\to(-1)^j\infty$ and $m\to A_j$; the only difference here is that $l_j(\lambda)$ is used in place of $k_j(\lambda)$. Specifically, we set
\begin{equation}
\label{hatPsi-Psi}
\hat\Psi_j(x,t,\lambda):= H_j(\lambda)\Psi(x,t,\lambda),
\end{equation}
where
\[
H_j(\lambda):= \sqrt{\frac{1}{2}}\sqrt{\frac{1}{\ii A_jl_j(\lambda)}-1}\begin{pmatrix}
\frac{\lambda A_j}{1-\ii A_jl_j(\lambda)} & - 1  \\
- 1 & \frac{\lambda A_j}{1-\ii A_jl_j(\lambda)} \\
\end{pmatrix}
\]
with inverse
\[
H_j^{-1}(\lambda):= \sqrt{\frac{1}{2}}\sqrt{\frac{1}{\ii A_jl_j(\lambda)}-1}\begin{pmatrix}
\frac{\lambda A_j}{1-\ii A_jl_j(\lambda)} & 1  \\
1 & \frac{\lambda A_j}{1-\ii A_jl_j(\lambda)} \\
\end{pmatrix}.
\]
The scalar factor
$\sqrt{\frac{1}{2}}\sqrt{\frac{1}{\ii A_jl_j(\lambda)}-1}$
ensures that $\det H_j(\lambda)=1$ for all $\lambda$ in its domain of definition.
We use the branch of the square root on $\mathbb C\setminus[0,\infty)$ defined by
$\sqrt{z}=|z|^{1/2}\eul^{\ii\arg z/2}$, where $\arg z\in(0,2\pi)$.
In particular, $\sqrt{-1}=\ii$. With this convention,
$\sqrt{\frac{1}{\ii A_jl_j(\lambda)}-1}$
is well defined for $\lambda\in\mathbb C\setminus\Gamma_j$, together with its boundary values on both sides of $\Gamma_j$.

\begin{remark} It might seem that $H_j^{-1}$ has a singularity at $\lambda=0$. Indeed,
\[
l_j(\lambda)=\begin{cases}
\frac{\ii}{A_j}-\frac{\ii A_j}{2}\lambda^2+\ord(\lambda^4),\quad\lambda\in\mathbb{C}^+,\\
-\frac{\ii}{A_j}+\frac{\ii A_j}{2}\lambda^2+\ord(\lambda^4),\quad\lambda\in\mathbb{C}^-,
\end{cases}
\]
and hence 
\[1-\ii A_jl_j(\lambda)=\begin{cases}
2-\frac{A_j^2}{2}\lambda^2+\ord(\lambda^4),\quad\lambda\in\mathbb{C}^+,\\
\frac{A_j^2}{2}\lambda^2+\ord(\lambda^4),\quad\lambda\in\mathbb{C}^-.
\end{cases}\]
However, since 
\[\sqrt{\frac{1}{\ii A_jl_j(\lambda)}-1}=\begin{cases}
\sqrt{2}\ii+\ord(\lambda^2),\quad\lambda\in\mathbb{C}^+,\\
-\frac{A_j\lambda}{\sqrt{2}}+\ord(\lambda^3),\quad\lambda\in\mathbb{C}^-,
\end{cases}\]
the singularity in the lower half-plane is removable, and we obtain
\begin{equation}
    \label{H_at_0}
    H_j^{-1}(\lambda)=\begin{cases}\begin{pmatrix}
0&\ii\\\ii&0
\end{pmatrix}+\lambda\ii\frac{A_j}{2}\begin{pmatrix}
1&0\\0&1
\end{pmatrix}+\ord(\lambda^2),\quad\lambda\in\mathbb{C}^+,\\
-\begin{pmatrix}
1&0\\0&1
\end{pmatrix}-\lambda\frac{A_j}{2}\begin{pmatrix}
0&1\\1&0
\end{pmatrix}+\ord(\lambda^2),\quad\lambda\in\mathbb{C}^-.
\end{cases}
\end{equation}
This shows that the presence of the factor $\sqrt{\frac{1}{\ii A_jl_j(\lambda)}-1}$ in the definition of $H_j^{-1}$ is essential: it cancels the apparent singularity of the matrix entries at $\lambda=0$, leaving a well-defined limit from each half-plane.
\end{remark}

Introducing
\begin{equation}\label{zam}
\tilde  \Psi_j=\hat\Psi_j\eul^{R_j},
\end{equation}
where
\begin{subequations}\label{Qp-l}
\begin{equation}\label{Q-l}
R_j(x,t,\lambda):= f_j(x,t,\lambda)\sigma_3,
\end{equation}
with
\begin{equation}\label{p_i-l}
f_j(x,t,\lambda):= \ii A_j l_j(\lambda)\left(\frac{1}{2A_j}\int^x_{(-1)^j\infty} ( m(\xi,t)-A_j)\dd\xi+\frac{x}{2}-t\big(\frac{1}{\lambda^2}+\frac{A_j^2}{2}\big)\right),
\end{equation}
\end{subequations}
we arrive at the following Lax pair equations
\begin{equation}\label{comsys}
\begin{cases}
\tilde  \Psi_{jx}+[R_{jx},\tilde  \Psi_j]=\hat W_j\tilde  \Psi_j,&\\
\tilde  \Psi_{jt}+[R_{jt},\tilde  \Psi_j]=\hat V_j\tilde  \Psi_j,&
\end{cases}
\end{equation}
where $[\,\cdot\,,\,\cdot\,]$ denotes the commutator, 
\begin{equation}\label{U_i-hat-l}
\hat W_j=\frac{\lambda(m-A_j)}{2 A_j l_j(\lambda)}
\sigma_2
+\frac{m-A_j}{2\ii A_j^2l_j(\lambda)}\sigma_3
\end{equation}
and 
\begin{equation}\label{hat-V_i}
\begin{aligned}
\hat V_j&=
-\frac{1}{2  A_j l_j(\lambda)}\left(\lambda (u^2-u_x^2)( m - A_j)
+\frac{2 (u-A_j)}{\lambda}\right)
\sigma_2
+\frac{ u_x}{\lambda}  \sigma_1 \\
&\quad-\frac{1}{\ii A_j l_j(\lambda)}\left(A_j( u-A_j) +\frac{1}{2A_j}(u^2-u_x^2)( m - A_j)\right) \sigma_3.
\end{aligned}
\end{equation}

We define the eigenfunctions (Jost solutions) $\tilde  \Psi_j$ as the solutions of the associated Volterra integral equations
\begin{equation}\label{eq_l}
\tilde  \Psi_{j}(x,t,\lambda)=I+\int_{(-1)^j\infty}^x
	\eul^{\frac{\ii l_j(\lambda)}{2}\int^\xi_x m(\tau,t)\dd\tau\sigma_3}\hat W_j(\xi,t,\lambda)\tilde  \Psi_{j}(\xi,t,\lambda)
			\eul^{-\frac{\ii l_j(\lambda)}{2}\int^\xi_x m(\tau,t)\dd\tau\sigma_3}\dd\xi.
\end{equation}

The functions $\Psi_j$ can then be characterized as the solutions of
\begin{equation}\label{eq_phi-l}
\Psi_{j}(x,t,\lambda)=\Psi_{0,j}(x,t,\lambda)+\int_{(-1)^j\infty}^x
	\Psi_{0,j}(x,t,\lambda)\Psi_{0,j}^{-1}(\xi,t,\lambda)\frac{m(\xi,t)-A_j}{2A_j} \sigma_3\Psi_{j}(\xi,t,\lambda)\dd\xi,
\end{equation}
where $\Psi_{0,j}(x,t,\lambda)=H_j^{-1}(\lambda)\eul^{-R_j(x,t,\lambda)}$.

\begin{remark}\label{rem:background-kernel-entire}
The kernel in the Volterra integral equation \eqref{eq_phi-l} is an entire function of
$\lambda$. Indeed, set
\[
L(x,y,t):=\int_y^x m(\xi,t)\,d\xi,
\qquad
\theta_j(x,y,t,\lambda)
:=
\frac{l_j(\lambda)}{2}L(x,y,t).
\]
Since
\[
\Psi_{0,j}(x,t,\lambda)
=
H_j^{-1}(\lambda)\eul^{-R_j(x,t,\lambda)},
\]
we have
\[
\Psi_{0,j}(x,t,\lambda)\Psi_{0,j}^{-1}(y,t,\lambda)
=
H_j^{-1}(\lambda)
\eul^{-R_j(x,t,\lambda)+R_j(y,t,\lambda)}
H_j(\lambda).
\]
Moreover,
\[
R_j(x,t,\lambda)-R_j(y,t,\lambda)
=
\ii\theta_j(x,y,t,\lambda)\sigma_3.
\]
A direct multiplication, using the explicit formulas for
$H_j^{-1}$ and $H_j$, gives
\[
\Psi_{0,j}(x,t,\lambda)\Psi_{0,j}^{-1}(y,t,\lambda)
 =
\begin{pmatrix}
\displaystyle
\cos\theta_j(x,y,t,\lambda)-
\frac{\sin\theta_j(x,y,t,\lambda)}{A_jl_j(\lambda)}
&
\displaystyle
\frac{\lambda\sin\theta_j(x,y,t,\lambda)}{l_j(\lambda)}
\\[3mm]
\displaystyle
-\frac{\lambda\sin\theta_j(x,y,t,\lambda)}{l_j(\lambda)}
&
\displaystyle
\cos\theta_j(x,y,t,\lambda)+
\frac{\sin\theta_j(x,y,t,\lambda)}{A_jl_j(\lambda)}
\end{pmatrix}.
\]
Although the entries in this expression apparently contain
$l_j(\lambda)$ in the denominator, the singularities at
$\lambda=\pm A_j^{-1}$ are removable. In fact,
\[
\cos\theta_j
=
\sum_{n=0}^{\infty}
\frac{(-1)^n}{(2n)!}
\left(\frac{L(x,y,t)}{2}\right)^{2n}
\left(\lambda^2-\frac{1}{A_j^2}\right)^n
\]
and
\[
\frac{\sin\theta_j}{l_j(\lambda)}
=
\sum_{n=0}^{\infty}
\frac{(-1)^n}{(2n+1)!}
\left(\frac{L(x,y,t)}{2}\right)^{2n+1}
\left(\lambda^2-\frac{1}{A_j^2}\right)^n.
\]
Both series are entire functions of $\lambda$. Consequently,
\[
\Psi_{0,j}(x,t,\lambda)\Psi_{0,j}^{-1}(y,t,\lambda)
\]
is entire with respect to $\lambda$.

In particular, at the branch points one has
\[
\left.
\Psi_{0,j}(x,t,\lambda)\Psi_{0,j}^{-1}(y,t,\lambda)
\right|_{\lambda=A_j^{-1}}
=
I+
\frac{L(x,y,t)}{2A_j}
\begin{pmatrix}
-1&1\\
-1&1
\end{pmatrix},
\]
and
\[
\left.
\Psi_{0,j}(x,t,\lambda)\Psi_{0,j}^{-1}(y,t,\lambda)
\right|_{\lambda=-A_j^{-1}}
=
I+
\frac{L(x,y,t)}{2A_j}
\begin{pmatrix}
-1&-1\\
1&1
\end{pmatrix}.
\]
Thus the Volterra kernel does not introduce any additional branch
singularities; the branch behavior of $\Psi_j$ is inherited from that
of the background solution $\Psi_{0,j}$.
\end{remark}

\begin{remark} Observe that $l_j(\lambda)$ coincides with $k_j(\lambda)$ in the upper half-plane, that $l_j(\lambda)$ coincides with $k_j(\lambda_+)$ on $\Sigma_j$, and that $l_j(\lambda_+)$ coincides with $k_j(\lambda)$ on $\Gamma_j$. Consequently, $\Psi_j(x,t,\lambda)$ (resp. $\tilde  \Psi_j(x,t,\lambda)$) coincides with $\Phi_j(x,t,\lambda)$ (resp. $\tilde  \Phi_j(x,t,\lambda)$) in the upper half-plane. Moreover, $\Psi_j(x,t,\lambda)$ (resp. $\tilde  \Psi_j(x,t,\lambda)$) coincides with $\Phi_j(x,t,\lambda_+)$ (resp. $\tilde  \Phi_j(x,t,\lambda_+)$) on $\Sigma_j$, and $\Psi_j(x,t,\lambda_+)$ (resp. $\tilde  \Psi_j(x,t,\lambda_+)$) coincides with $\Phi_j(x,t,\lambda)$ (resp. $\tilde  \Phi_j(x,t,\lambda)$) on $\Gamma_j$ (cf.\ \cite{KST22}).

\end{remark}

Notice that the analytic properties of the eigenfunctions $\tilde  \Psi_j$ differ substantially from those of $\tilde  \Phi_j$ (corresponding to the case of an outer cut; cf.\ \cite{KST22}). In particular, assuming that
$
m(\cdot,t)-A_1\in L^{1}(\mathbb{R}^-) $ and $
m(\cdot,t)-A_2\in L^{1}(\mathbb{R}^+),
$
for all $t\ge 0$, the Neumann series expansions imply the following:
\begin{enumerate}[\textbullet]
\item
$\tilde  {\Psi}_1^{(1)}(x,t,\lambda)$ and $\tilde  {\Psi}_2^{(2)}(x,t,\lambda)$ are analytic in the upper half-plane, while $\tilde  {\Psi}_1^{(2)}(x,t,\lambda)$ and $\tilde  {\Psi}_2^{(1)}(x,t,\lambda)$ are analytic in the lower half-plane;
\item $\tilde  {\Psi}_j^{(k)}(x,t,\lambda)$, $j,k=1,2$, have continuous extensions to $\dot\Sigma_j$ from their respective domains of analyticity;
\item $\tilde  \Psi_j^{(j)}(x,t,\lambda)$, $j=1,2$, have continuous extensions to the upper side of $\dot{\Gamma}_j$, while $\tilde  \Psi_1^{(2)}(x,t,\lambda)$ has a continuous extension to the lower side of $\dot{\Gamma}_1$ and $\tilde  \Psi_2^{(1)}(x,t,\lambda)$ has a continuous extension to the lower side of $\dot{\Gamma}_2$;
\item
$\det\tilde  \Psi_j(x,t,\lambda)\equiv 1$.
\end{enumerate}
Regarding the values of $\tilde  \Psi_j(x,t,\lambda)$ at particular points in the $\lambda$-plane, \eqref{eq_l} implies the following:
\begin{enumerate}[\textbullet]
\item
$\left(\begin{smallmatrix}
\tilde  \Psi_1^{(1)} &
\tilde  \Psi_2^{(2)}\end{smallmatrix}\right)\to I$ as $\lambda\to\infty$ (since diagonal part of $\hat W_j$ is $\ord(\frac{1}{\lambda})$ and off-diagonal part is bounded);
\item
$\left(\begin{smallmatrix}
\tilde  \Psi_2^{(1)} &
\tilde  \Psi_1^{(2)}\end{smallmatrix}\right)\to I$ as $\lambda\to\infty$ (since diagonal part of $\hat W_j$ is $\ord(\frac{1}{\lambda})$ and off-diagonal part is bounded);
\item $\tilde  \Psi_j(x,t,\lambda)$, $j=1,2$, have singularities at  $\lambda=\pm\frac{1}{A_j}$  of order $\frac{1}{2}$. 
\end{enumerate}

Introduce the scattering matrix $c(\lambda)$ for $\lambda\in\dot\Sigma_1$ as a matrix relating $\Psi_1$ and $\Psi_2$:
\begin{equation}\label{scat-l}
\Psi_1(x,t,\lambda)=\Psi_2(x,t,\lambda)c(\lambda),\qquad\lambda\in\dot{\Sigma}_1,
\end{equation}
with $\det c(\lambda)=1$. In turn, $\tilde\Psi_1$ and $\tilde\Psi_2$ are related by:
\begin{equation}\label{scat_-l}
H_1^{-1}(\lambda)\tilde\Psi_1(x,t,\lambda)=H_2^{-1}(\lambda)\tilde\Psi_2(x,t,\lambda)\tilde c(x,t,\lambda),\qquad\lambda\in\dot{\Sigma}_1,
\end{equation}
where $\tilde c(x,t,\lambda)=\eul^{-R_2(x,t,\lambda)}c(\lambda)\eul^{R_1(x,t,\lambda)}$.

Notice that the scattering coefficients ($c_{ij}(\lambda)$) can be expressed as follows:
\begin{subequations}\label{scatcoeff-l}
\begin{alignat}{4}\label{scatcoeff-l-a}
c_{11}&=\det(\Psi_1^{(1)},\Psi_2^{(2)}),\\\label{scatcoeff-l-b}
c_{12}&=\det(\Psi_1^{(2)},\Psi_2^{(2)}),\\\label{scatcoeff-l-c}
c_{21}&=\det(\Psi_2^{(1)},\Psi_1^{(1)}),\\\label{scatcoeff-l-d}
c_{22}&=\det(\Psi_2^{(1)},\Psi_1^{(2)}).
\end{alignat}
\end{subequations}

Then \eqref{scatcoeff-l-a} implies that $c_{11}$ may be analytically continued to the upper half-plane and extended to the upper side of $\dot\Gamma_1$; \eqref{scatcoeff-l-b} implies that $c_{12}$ may be extended to the lower side of $\Sigma_0$; \eqref{scatcoeff-l-c} implies that $c_{21}$ may be extended to the upper side of $\Sigma_0$; \eqref{scatcoeff-l-d} implies that $c_{22}$ may be analytically continued to the lower half-plane and extended to the lower side of $\dot\Gamma_1$. It follows that \eqref{scat-l}, restricted to the second column, holds also on the lower side of $\dot\Sigma_0$, and \eqref{scat-l}, restricted to the first column, holds also on the upper side of $\dot\Sigma_0$, namely,
\begin{subequations}\label{scat-col-l}
\begin{alignat}{4}\label{scat-col-l-a}
    \Psi_1^{(1)}(x,t,\lambda_+)=c_{11}(\lambda_+)\Psi_2^{(1)}(x,t,\lambda_+)+c_{21}(\lambda_+)\Psi_2^{(2)}(x,t,\lambda_+), \qquad\lambda\in\dot\Sigma_0,\\\label{scat-col-l-b}
    \Psi_1^{(2)}(x,t,\lambda_-)=c_{12}(\lambda_-)\Psi_2^{(1)}(x,t,\lambda_-)+c_{22}(\lambda_-)\Psi_2^{(2)}(x,t,\lambda_-), \qquad\lambda\in\dot\Sigma_0,
\end{alignat}
\end{subequations}
and, respectively,
\begin{subequations}\label{scat-col-l_}
\begin{alignat}{4}\label{scat-col_-l-a}
    \hspace{-5mm}(H_1^{-1}\tilde\Psi_1^{(1)})(\lambda_+)=\tilde c_{11}(\lambda_+)(H_2^{-1}\tilde\Psi_2^{(1)})(\lambda_+)+\tilde c_{21}(\lambda_+)(H_2^{-1}\tilde\Psi_2^{(2)})(\lambda_+), \qquad\lambda\in\dot\Sigma_0,\\\label{scat-col_-l-b}
    \hspace{-5mm}(H_1^{-1}\tilde\Psi_1^{(2)})(\lambda_-)=\tilde c_{12}(\lambda_-)(H_2^{-1}\tilde\Psi_2^{(1)})(\lambda_-)+\tilde c_{22}(\lambda_-)(H_2^{-1}\tilde\Psi_2^{(2)})(\lambda_-), \qquad\lambda\in\dot\Sigma_0.
\end{alignat}
\end{subequations}

\begin{remark}
$c_{11}(\lambda)$  coincides with $s_{11}(\lambda)$ in the upper half-plane. Moreover, $c_{11}(\lambda)$ coincides with $s_{11}(\lambda_+)$ on $\Sigma_1$, and $c_{11}(\lambda_+)$ coincides with $s_{11}(\lambda)$ on $\Gamma_1$. Similarly, $c_{21}(\lambda_+)$  coincides with $s_{21}(\lambda_+)$ on $\Sigma_0$, and $c_{21}(\lambda)$  coincides with $s_{21}(\lambda_+)$ on $\Sigma_1$. (cf.\ \cite{KST22})
\end{remark}

\subsection{Symmetries}  \label{sec:symmetries-l}

We next analyze the symmetry relations among the eigenfunctions and the scattering coefficients. To simplify the notation, we will omit the dependence on $x$ and $t$ (e.g., $U(\lambda)\equiv U(x,t,\lambda)$).

Notice that
\begin{equation}\label{Psi_tildePsi}
    H_j^{-1}(\lambda)\tilde\Psi_j(\lambda)=\Psi_j\eul^{R_j(\lambda)}.
\end{equation}

\textit{First symmetry: $\lambda \longleftrightarrow -\lambda$.}

\begin{proposition}\label{prop:sym_Phi_minus_sigma_i-l} The eigenfunctions $\Psi_j$ satisfy the following symmetry relation:
    \begin{equation}\label{sym_Phi_minus_sigma_i-l}
    \Psi_j(\lambda)=\sigma_3\Psi_j(-\lambda)\sigma_2, \qquad \lambda\in\dot\Sigma_j.
\end{equation}

\end{proposition}
\begin{proof}
Observe that $\sigma_3U(\lambda)\sigma_3\equiv U(-\lambda)$ and $\sigma_3V(\lambda)\sigma_3\equiv V(-\lambda)$. Hence, if $\Psi_j(\lambda)$ solves \eqref{Lax_App}, so does $\sigma_3 \Psi_j(-\lambda)$. Comparing asymptotic behavior as $x\to(-1)^j\infty$ and, using \eqref{sym_l_i-a} and $\sqrt{-\frac{1}{\ii A_jl_j(\lambda)}-1} \sqrt{\frac{1}{\ii A_jl_j(\lambda)}-1}=-\frac{l_j(\lambda)}{\lambda}$ for $\lambda\in\Sigma_j$, we obtain \eqref{sym_Phi_minus_sigma_i-l}.
\end{proof}

\begin{corollary}\label{cor:C7} The following symmetry relations hold:
\begin{enumerate}
    \item \ifshort
    \begin{equation}
       c(\lambda)=\sigma_2 c(-\lambda) \sigma_2,\quad \lambda\in\dot\Sigma_1. 
    \end{equation}
    \fi   
      \item $c_{11}(\lambda)=c_{22}(-\lambda)$ for $\lambda\in\mathbb{C}^+$. In particular, $c_{11}(\lambda_+)=c_{22}((-\lambda)_-)$ for $\lambda\in\dot\Gamma_1$.
     
    \item
    \ifshort
    \begin{equation}\label{sym_Phi_minus_sigma_i-l_elem_2}
    \Psi_1^{(1)}(\lambda)=\ii\sigma_3\Psi_1^{(2)}(-\lambda),\qquad
    \Psi_2^{(2)}(\lambda)=-\ii\sigma_3\Psi_2^{(1)}(-\lambda),\qquad\lambda\in\mathbb{C}^+.    
    \end{equation}
    \fi
    \item \begin{equation}\label{sym_minus-l} 
 (H_j^{-1}\tilde\Psi_j)(\lambda)=\sigma_3(H_j^{-1}\tilde\Psi_j)(-\lambda)\sigma_2, \quad \lambda\in\dot\Sigma_j.   
\end{equation}

\item 
\begin{subequations}\label{sym_D_tildePhi_minus_sigma_i-l_elem}
\begin{alignat}{4}
    (H_1^{-1}\tilde\Psi_1)^{(1)}(\lambda)&=\ii\sigma_3(H_1^{-1}\tilde\Psi_1)^{(2)}(-\lambda),\qquad \lambda\in\mathbb{C}^+,\\
    (H_2^{-1}\tilde\Psi_2)^{(2)}(\lambda)&=-\ii\sigma_3(H_2^{-1}\tilde\Psi_2)^{(1)}(-\lambda),\qquad\lambda\in\mathbb{C}^+.
\end{alignat}
\end{subequations}
    
 \end{enumerate}
\end{corollary}

\begin{proof} \begin{enumerate}
    \item Follows by substituting \eqref{sym_Phi_minus_sigma_i-l} into \eqref{scat-l}
    
    \item  Consider $g(\lambda)=c_{11}(\lambda)-c_{22}(-\lambda).$ We know that $g(\lambda)$ is analytic in $\mathbb{C}^+$ and $g(\lambda)=0$ on $\Sigma_1$. Hence, using the Schwarz reflection principle, $g(\lambda)$ can be analytically continued in $\mathbb{C}^-$. Now, using the uniqueness theorem, we can conclude that $g(\lambda)\equiv 0$ in $\mathbb{C}^+$. 
    
    \item  
    
    \ifshort
    Rewriting \eqref{sym_Phi_minus_sigma_i-l} columnwise, we have 
    \begin{alignat*}{4}
    \Psi_1^{(1)}(\lambda)=\ii\sigma_3\Psi_1^{(2)}(-\lambda),\quad\lambda\in\dot\Sigma_1;\qquad
    \Psi_2^{(2)}(\lambda)=-\ii\sigma_3\Psi_2^{(1)}(-\lambda),\quad\lambda\in\dot\Sigma_2.
    \end{alignat*}
    Recall that $ \Psi_1^{(1)}(\lambda),~\Psi_2^{(2)}(\lambda)$ can be analytically continued in the upper half-plane, while $ \Psi_1^{(2)}(\lambda),~\Psi_2^{(1)}(\lambda)$ can be analytically continued in the lower half-plane. 
    
    Then, using the Schwarz reflection principle, we obtain \eqref{sym_Phi_minus_sigma_i-l_elem_2}. Indeed, let $g(\lambda)=\Psi_1^{(1)}(\lambda)-\ii\sigma_3\Psi_1^{(2)}(-\lambda)$. Then $g(\lambda)$ is an analytic vector-valued function in $\mathbb{C}^+$ and $g(\lambda)=0\in\mathbb{R}^2$ on $\Sigma_1$. Hence, using the Schwarz reflection principle, $g(\lambda)$ can be analytically continued in $\mathbb{C}^-$. Now, using the uniqueness theorem, we conclude that $g(\lambda)\equiv 0$ in $\mathbb{C}^+$. This proves the first equality. The second one can be proved analogously.
    \fi
    
    \item Note that due to \eqref{sym_l_i-a}, we have $R_j(-\lambda)=-R_j(\lambda)$. Then \eqref{sym_minus-l} follows from \eqref{Psi_tildePsi}.
    
    \item It follows again from the Schwarz reflection principle, similarly to \textit{(3)}.
\end{enumerate}

\end{proof}

\textit{\textit{Second symmetry: $\lambda \longleftrightarrow \overline\lambda$.} }

\begin{proposition}\label{prop:sym_Phi_barbar_sigma_i-l} The eigenfunctions $\Psi_j$ satisfy the following symmetry relation:
    \begin{equation}\label{sym_Phi_barbar_sigma_i-l}
    \Psi_j(\lambda)=-\ii\overline{\Psi_j(\overline\lambda)}\sigma_1, \qquad \lambda\in\dot\Sigma_j.
\end{equation}

\end{proposition}
\begin{proof}
Observe that $\overline{U(\lambda)}\equiv U(\lambda)$ and $\overline{V(\lambda)}\equiv V(\lambda)$ for $\lambda\in\Sigma_j$. Hence, if $\Psi_j(\lambda)$ solves \eqref{Lax_App}, so does $ \overline{\Psi_j(\lambda)}$. Comparing asymptotic behavior as $x\to(-1)^j\infty$ and, using \eqref{sym_l_i-c}, we get \eqref{sym_Phi_barbar_sigma_i-l}.
\end{proof}

\begin{corollary}\label{cor:C-sym}
The following symmetry relations hold:
\begin{enumerate}
    \item $c(\lambda)=\sigma_1 \overline{c(\lambda)} \sigma_1$ for $\lambda\in\dot\Sigma_1$ or, more precisely, for $\lambda\in\dot\Sigma_1$
        \begin{subequations}\label{sym_s_barbar_sigma1-l}
        \begin{alignat}{3}\label{sym_s_barbar_sigma1-l-a}
        c_{11}(\lambda)&=\overline{c_{22}(\lambda)},\\
        c_{21}(\lambda)&=\overline{c_{12}(\lambda)}.
        \end{alignat}
        \end{subequations} 
        
       \item $1=|c_{11}(\lambda)|^2-|c_{21}(\lambda)|^2$ for $\lambda\in\dot\Sigma_1$.
       
      \item $c_{11}(\lambda)=\overline{c_{22}(\overline\lambda)}$ for $\lambda\in\mathbb{C}^+$. In particular, $c_{11}(\lambda_+)=\overline{c_{22}(\lambda_-)}$ for $\lambda\in\dot\Gamma_1$.

    \item 
    \ifshort
    \begin{equation}\label{sym_Phi_barbar_sigma_i-l_col_2}
    \Psi_1^{(1)}(\lambda)=-\ii\overline{\Psi_1^{(2)}(\overline{\lambda})}, \quad \Psi_2^{(2)}(\lambda)=-\ii\overline{\Psi_2^{(1)}(\overline{\lambda})},\quad \lambda\in\mathbb{C}^+.   
    \end{equation}
    \fi
    
    \item \begin{equation}\label{sym_barbar-l} 
 (H_j^{-1}\tilde\Psi_j)(\lambda)=-\ii\overline{(H_j^{-1}\tilde\Psi_j)(\lambda)}\sigma_1, \quad \lambda\in\dot\Sigma_j.   
\end{equation}

\item \ifshort
\begin{equation}\label{sym_D_tildePhi_barbar_sigma_i-l_col}
(H_1^{-1}\tilde\Psi_1)^{(1)}(\lambda)=-\ii\overline{(H_1^{-1}\tilde\Psi_1)^{(2)}(\overline{\lambda})},\quad 
(H_2^{-1}\tilde\Psi_2)^{(2)}(\lambda)=-\ii\overline{(H_2^{-1}\tilde\Psi_2)^{(1)}(\overline{\lambda})},\quad \lambda\in\mathbb{C}^+.   
\end{equation}
\fi
 \end{enumerate}
\end{corollary}
\begin{proof}
\begin{enumerate}
    \item Follows by substituting \eqref{sym_Phi_barbar_sigma_i-l} into \eqref{scat-l}
    
    \item Since $\det c(\lambda)=1$ on $\dot\Sigma_1$.
    
    \item Consider $g(\lambda)=c_{11}(\lambda)-\overline{c_{22}(\overline\lambda)}.$ We know that $g(\lambda)$ is analytic in $\mathbb{C}^+$ and $g(\lambda)=0$ on $\Sigma_1$. Hence, using the Schwarz reflection principle, $g(\lambda)$ can be analytically continued in $\mathbb{C}^-$. Now, using the uniqueness theorem, we can conclude that $g(\lambda)\equiv 0$ in $\mathbb{C}^+$. 
    
    \item Rewriting \eqref{sym_Phi_barbar_sigma_i-l} columnwise, we have
    \begin{alignat*}{4}
    \Psi_1^{(1)}(\lambda)&=-\ii\overline{\Psi_1^{(2)}(\lambda)},\quad\lambda\in\dot\Sigma_1;\qquad
    \Psi_2^{(2)}(\lambda)&=-\ii\overline{\Psi_2^{(1)}(\lambda)},\quad\lambda\in\dot\Sigma_2.
    \end{alignat*}
    \ifshort
    Since $ \Psi_1^{(1)}(\lambda),~\Psi_2^{(2)}(\lambda)$ can be analytically continued in the upper half-plane, $ \Psi_1^{(2)}(\lambda),~\Psi_2^{(1)}(\lambda)$ can be analytically continued in the lower half-plane, \eqref{sym_Phi_barbar_sigma_i-l_col_2} follows from the Schwarz reflection principle similarly to the proof of the item \textit{(3)} of Corollary \ref{cor:C7}. 
    \fi
    
        \item Note that due to \eqref{sym_l_i-c}, we have $R_j(\lambda)=-\overline{R_j(\overline{\lambda})}$. Then \eqref{sym_barbar-l} follows from \eqref{Psi_tildePsi}.
    
    \item It follows again from the Schwarz reflection principle, similarly to \textit{(4)}.
\end{enumerate}

\end{proof}

\textit{Third symmetry: $\lambda_+ \longleftrightarrow (-\lambda)_+$.} 

\begin{proposition}\label{prop:sym-Phi-(minus)_+} For $\lambda\in\dot\Gamma_j$, the eigenfunctions satisfy the symmetries
\begin{subequations}\label{sym_Phi_(minus)_+}
\begin{alignat}{4}
\Psi_1^{(1)}(\lambda_+)=-\sigma_3\Psi_1^{(1)}((-\lambda)_+), \qquad \lambda\in\dot\Gamma_1,\\
\Psi_2^{(2)}(\lambda_+)=\sigma_3\Psi_2^{(2)}((-\lambda)_+), \qquad \lambda\in\dot\Gamma_2.
\end{alignat}
\end{subequations}
\end{proposition}
\begin{proof}
Observe that $\sigma_3U(\lambda_+)\sigma_3\equiv U((-\lambda)_+)$ and $\sigma_3V(\lambda_+)\sigma_3\equiv V((-\lambda)_+)$ for $\lambda_+\in\Gamma_j$ (since $U$ and $V$ are single valued with respect to $\lambda$, and $u$, $u_x$, and $m$ are real-valued). Hence, if $\Psi_j^{(j)}(\lambda_+)$ solves \eqref{Lax_App}, so does $\sigma_3\Psi_j^{(j)}((-\lambda)_+)$. Comparing asymptotic behavior as $x\to(-1)^j\infty$, using \eqref{sym_l_i-b}, we get \eqref{sym_Phi_(minus)_+}.
\end{proof}

\begin{corollary}
The scattering coefficient $c_{11}$ satisfies the symmetry
    \begin{equation}\label{sym_s11-(minus)-l-+}
        c_{11}(\lambda_+)=c_{11}((-\lambda)_+),\quad\lambda\in\dot\Gamma_2
    \end{equation}

\end{corollary}

\begin{proof}
Substitute \eqref{sym_Phi_(minus)_+} into \eqref{scatcoeff-l-a}.
\end{proof}

\begin{proposition}\label{prop:sym-Phi-(minus)_-} For $\lambda\in\dot\Gamma_j$, the eigenfunctions satisfy the symmetries
\begin{subequations}\label{sym_Phi_(minus)_-}
\begin{alignat}{4}
\Psi_2^{(1)}(\lambda_-)=-\sigma_3\Psi_2^{(1)}((-\lambda)_-), \qquad \lambda\in\dot\Gamma_2,\\
\Psi_1^{(2)}(\lambda_-)=\sigma_3\Psi_1^{(2)}((-\lambda)_-), \qquad \lambda\in\dot\Gamma_1.
\end{alignat}
\end{subequations}
\end{proposition}
\begin{proof}
Observe that $\sigma_3U(\lambda_-)\sigma_3\equiv U((-\lambda)_-)$ and $\sigma_3V(\lambda_-)\sigma_3\equiv V((-\lambda)_-)$ for $\lambda_-\in\Gamma_j$ (since $U$ and $V$ are single valued with respect to $\lambda$, and $u$, $u_x$, and $m$ are real-valued). Hence, if $\Psi_1^{(2)}(\lambda_-)$ solves \eqref{Lax_App}, so does $\sigma_3\Psi_1^{(2)}((-\lambda)_-)$ (resp., $\Psi_2^{(1)}(\lambda_-)$ solves \eqref{Lax_App}, so does $\sigma_3\Psi_2^{(1)}((-\lambda)_-)$). Comparing asymptotic behavior as $x\to(-1)^j\infty$, using \eqref{sym_l_i-b}, we get \eqref{sym_Phi_(minus)_-}.
\end{proof}

\begin{corollary}
The scattering coefficient $c_{22}$ satisfies the symmetry
    \begin{equation}\label{sym_s11-(minus)-l--}
        c_{22}(\lambda_-)=c_{22}((-\lambda)_-),\quad\lambda_-\in\dot\Gamma_2
    \end{equation}
\end{corollary}

\begin{proof}
Substitute \eqref{sym_Phi_(minus)_-} into \eqref{scatcoeff-l-d}.
\end{proof}

\textit{Fourth symmetry: $\lambda_+ \longleftrightarrow \lambda_+$.} 

\begin{proposition}\label{prop:sym-Phi-bar_+} For $j=1,2$, the following symmetry relation holds:
\begin{equation}\label{sym_Phi_bar_+}
\Psi_j^{(j)}(\lambda_+)=-\overline{\Psi_j^{(j)}(\lambda_+)}, \qquad \lambda\in\dot\Gamma_j.
\end{equation}
\end{proposition}
\begin{proof}
Observe that $\overline{U(\lambda_+)}\equiv U(\lambda_+)$ and $\overline{V(\lambda_+)}\equiv V(\lambda_+)$ for $\lambda_+\in\Gamma_j$ (since $U$ and $V$ are single valued with respect to $\lambda$, and $u$, $u_x$, and $m$ are real-valued). Hence, if $\Psi_j^{(j)}(\lambda_+)$ solves \eqref{Lax_App}, so does $\overline{\Psi_j^{(j)}(\lambda_+)}$. Comparing asymptotic behavior as $x\to(-1)^j\infty$, using \eqref{sym_l_i-d} and the fact that  $\overline{\sqrt{\frac{1}{\ii A_jl_j(\lambda_+)}-1}}=-\sqrt{\frac{1}{\ii A_jl_j(\lambda_+)}-1}$ for $\lambda_+\in\dot\Gamma_j$ (our choice of branch), we get \eqref{sym_Phi_bar_+}.
\end{proof}

\begin{corollary} The following symmetry relations hold:
\begin{enumerate}
    \item \begin{equation}\label{sym_s11-bar-l-+}
        c_{11}(\lambda_+)=\overline{c_{11}(\lambda_+)},\quad\lambda_+\in\dot\Gamma_2.
    \end{equation}
    
    \item \begin{equation}\label{sym_D_Phi_bar_+}
    (H_j^{-1}\tilde\Psi_j)^{(j)}(\lambda_+)=-\overline{(H_j^{-1}\tilde\Psi_j)^{(j)}(\lambda_+)}, \qquad \lambda\in\dot\Gamma_j.
    \end{equation}
    \end{enumerate}

\end{corollary}

\begin{proof}
\begin{enumerate}
    \item Substitute \eqref{sym_Phi_bar_+} into \eqref{scatcoeff-l-a}.
    
    \item Note that due to \eqref{sym_l_i-d}, we have $R_j(\lambda_+)=\overline{R_j(\lambda_+)}$. Then \eqref{sym_D_Phi_bar_+} follows from \eqref{Psi_tildePsi}.
\end{enumerate}
\end{proof}

\begin{proposition}\label{prop:sym-Phi-bar_-} For $\lambda\in\dot\Gamma_j$, the eigenfunctions satisfy the symmetries
\begin{subequations}\label{sym_Phi_bar_-}
\begin{alignat}{4}
\Psi_1^{(2)}(\lambda_-)=\overline{\Psi_1^{(2)}(\lambda_-)}, \qquad \lambda\in\dot\Gamma_1,\\
\Psi_2^{(1)}(\lambda_-)=\overline{\Psi_2^{(1)}(\lambda_-)}, \qquad \lambda\in\dot\Gamma_2.
\end{alignat}
\end{subequations}
\end{proposition}
\begin{proof}
Observe that $\overline{U(\lambda_-)}\equiv U(\lambda_-)$ and $\overline{V(\lambda_-)}\equiv V(\lambda_-)$ for $\lambda_-\in\Gamma_j$ (since $U$ and $V$ are single valued with respect to $\lambda$, and $u$, $u_x$, and $m$ are real-valued). Hence, if $\Psi_1^{(2)}(\lambda_-)$ solves \eqref{Lax_App}, so does $\overline{\Psi_1^{(2)}(\lambda_-)}$ (resp., $\Psi_2^{(1)}(\lambda_-)$ solves \eqref{Lax_App}, so does $\overline{\Psi_2^{(1)}(\lambda_-)}$). Comparing asymptotic behavior as $x\to(-1)^j\infty$, using \eqref{sym_l_i-d} and the fact that  $\overline{\sqrt{\frac{1}{\ii A_jl_j(\lambda_-)}-1}}=\sqrt{\frac{1}{\ii A_jl_j(\lambda_-)}-1}$ for $\lambda_-\in\dot\Gamma_j$ (our choice of branch), we get \eqref{sym_Phi_bar_-}.
\end{proof}

\begin{corollary}
The following symmetry relations hold:
\begin{enumerate}
    \item 
    \begin{equation}\label{sym_s11-bar-l--}
        c_{22}(\lambda_-)=\overline{c_{22}(\lambda_-)},\quad\lambda_-\in\dot\Gamma_2.
    \end{equation}
    
    \item 
        \begin{subequations}\label{sym_Phi_barbar_sigma_i-l_col_2-comb}
    \begin{alignat}{4}\label{sym_Phi_barbar_sigma_i-l_col_2-comb-a}
    \Psi_1^{(1)}(\lambda_+)&=-\ii\Psi_1^{(2)}(\lambda_-),\quad\lambda\in\dot\Gamma_1\\\label{sym_Phi_barbar_sigma_i-l_col_2-comb-b}
    \Psi_2^{(2)}(\lambda_+)&=-\ii\Psi_2^{(1)}(\lambda_-),\quad\lambda\in\dot\Gamma_2.
    \end{alignat}
    \end{subequations}
    \item 
    \begin{subequations}\label{sym_s_gamma}
    \begin{alignat}{4}\label{sym_s_gamma-a}
    c_{11}(\lambda_+)=c_{22}(\lambda_-), \quad\lambda_+\in\dot\Gamma_2,\\\label{sym_s_gamma-b}
    c_{11}(\lambda_+)=-\ii c_{12}(\lambda_-), \quad\lambda_+\in\dot\Sigma_0=\dot\Gamma_1\setminus\Gamma_2,\\
    \label{sym_s_gamma-c}
    c_{22}(\lambda_-)=\ii c_{21}(\lambda_+), \quad\lambda_+\in\dot\Sigma_0=\dot\Gamma_1\setminus\Gamma_2.
    \end{alignat}
    \end{subequations}
    
    \item \begin{subequations}\label{sym_D_Phi_bar_-}
\begin{alignat}{4}
(H_1^{-1}\tilde\Psi_1)^{(2)}(\lambda_-)=\overline{(H_1^{-1}\tilde\Psi_1)^{(2)}(\lambda_-)}, \qquad \lambda\in\dot\Gamma_1,\\
(H_2^{-1}\tilde\Psi_2)^{(1)}(\lambda_-)=\overline{(H_2^{-1}\tilde\Psi_2)^{(1)}(\lambda_-)}, \qquad \lambda\in\dot\Gamma_2.
\end{alignat}
\end{subequations}

\item  \begin{subequations}\label{sym_D_tildePhi_barbar_sigma_i-l_col-gamma}
    \begin{alignat}{4}\label{sym_D_tildePhi_barbar_sigma_i-l_col-gamma-a}
    (H_1^{-1}\tilde\Psi_1)^{(1)}(\lambda_+)&=-\ii(H_1^{-1}\tilde\Psi_1)^{(2)}(\lambda_-),\qquad \lambda\in\dot\Gamma_1\\\label{sym_D_tildePhi_barbar_sigma_i-l_col-gamma-b}
    (H_2^{-1}\tilde\Psi_2)^{(2)}(\lambda_+)&=-\ii(H_2^{-1}\tilde\Psi_2)^{(1)}(\lambda_-), \qquad \lambda\in\dot\Gamma_2.
    \end{alignat}
    \end{subequations}

\end{enumerate}
\end{corollary}

\begin{proof}
\begin{enumerate}
    \item Substitute \eqref{sym_Phi_bar_-} into \eqref{scatcoeff-l-d}.
    
    \item Combine \eqref{sym_Phi_bar_-} and \eqref{sym_Phi_barbar_sigma_i-l_col_2}.
    
    \item Use the definition of $c_{ij}(\lambda)$ via determinant and \eqref{sym_Phi_barbar_sigma_i-l_col_2-comb}. Note that for the second and third equalities, we use the facts that $\Psi_2^{(2)}(\lambda_+)=\Psi_2^{(2)}(\lambda_-)=\Psi_2^{(2)}(\lambda)$ and $\Psi_2^{(1)}(\lambda_+)=\Psi_2^{(1)}(\lambda_-)=\Psi_2^{(1)}(\lambda)$ on $\dot\Sigma_0$.
    
    \item Note that due to \eqref{sym_l_i-d}, we have $R_j(\lambda_-)=\overline{R_j(\lambda_-)}$. Then \eqref{sym_D_Phi_bar_-} follows from \eqref{Psi_tildePsi}.
    
    \item Combine \eqref{sym_D_Phi_bar_-} and \eqref{sym_D_tildePhi_barbar_sigma_i-l_col}
\end{enumerate}
\end{proof}

\subsection{Behavior at the branch points}
\label{sec:branch-point-behavior-l}

To describe the behavior at the branch points, we introduce
notations analogous to those used in \cite{KST22}, with the branches
adapted to the inner cuts. Let
\begin{equation}
\label{omega_j-l}
\omega_j^+(\lambda)
=
\sqrt{\lambda-\frac{1}{A_j}},
\qquad
\omega_j^-(\lambda)
=
\sqrt{\lambda+\frac{1}{A_j}},
\end{equation}
where $\omega_j^+$ has the branch cut
$\left(-\infty,\frac{1}{A_j}\right]$
and is positive for $\lambda>A_j^{-1}$, whereas $\omega_j^-$ has the
branch cut
$\left(-\infty,-\frac{1}{A_j}\right]$
and is positive for $\lambda>-A_j^{-1}$. With these conventions,
\begin{equation}
\label{l-omega-factorization}
l_j(\lambda)
=
\omega_j^+(\lambda)\omega_j^-(\lambda),
\qquad
\lambda\in\mathbb C\setminus\Gamma_j.
\end{equation}

We also introduce
\begin{equation}
\label{nu_j-l}
\nu_j^+(\lambda)
=
\left(\lambda-\frac{1}{A_j}\right)^{1/4},
\qquad
\nu_j^-(\lambda)
=
\left(\lambda+\frac{1}{A_j}\right)^{1/4},
\end{equation}
with the same respective cuts as $\omega_j^\pm$ and with
$\nu_j^+(\lambda)>0,
~ \lambda>A_j^{-1}$,
and
$\nu_j^-(\lambda)>0,
~ \lambda>-A_j^{-1}$.

Thus,
\begin{equation}
\label{nu-omega-relation-l}
\big(\nu_j^\pm(\lambda)\big)^2
=
\omega_j^\pm(\lambda).
\end{equation}

Arguing similarly to \cite{KST22}, and using Remark
\ref{rem:background-kernel-entire}, one obtains the following result.

\begin{proposition}\label{prop:Psi-branch-expansion}
Let $j=1,2$. There exist real-valued functions $a_j(x,t)$ and
$b_j(x,t)$ such that
\begin{align}
\Psi_j(x,t,\lambda)
&=
\frac{1}{\sqrt{2\ii A_jl_j(\lambda)}}
\begin{pmatrix}
a_j(x,t) & a_j(x,t)\\
b_j(x,t) & b_j(x,t)
\end{pmatrix}
+\ord\!\left(\sqrt{l_j(\lambda)}\right),
\qquad
\lambda\to\frac{1}{A_j},
\\
\Psi_j(x,t,\lambda)
&=
\frac{1}{\sqrt{2\ii A_jl_j(\lambda)}}
\begin{pmatrix}
-a_j(x,t) & a_j(x,t)\\
b_j(x,t) & -b_j(x,t)
\end{pmatrix}
+\ord\!\left(\sqrt{l_j(\lambda)}\right),
\qquad
\lambda\to-\frac{1}{A_j}.
\end{align}
\end{proposition}

The symmetries
\eqref{sym_Phi_minus_sigma_i-l_elem_2},
\eqref{sym_Phi_barbar_sigma_i-l_col_2},
\eqref{sym_Phi_(minus)_+}, and
\eqref{sym_Phi_bar_+}, together with
Proposition~\ref{prop:Psi-branch-expansion} and the determinant
representations \eqref{scatcoeff-l}, imply the following result.

\begin{corollary}\label{cor:c11-c21-branch-expansion}
There exist constants $\alpha_1\in\mathbb{C}$ and
$\gamma_2\in\mathbb{R}$ with the following properties.

As $\lambda\to A_1^{-1}$, $\lambda\in\dot{\Sigma}_1$,
\begin{align}
c_{11}(\lambda)
&=
\frac{\alpha_1}
{\sqrt{2\ii A_1l_1(\lambda)}}
+\ord\!\left(\sqrt{l_1(\lambda)}\right),\\
c_{21}(\lambda)
&=
\frac{\ii\overline{\alpha_1}}
{\sqrt{2\ii A_1l_1(\lambda)}}
+\ord\!\left(\sqrt{l_1(\lambda)}\right),
\end{align}
whereas, as $\lambda\to A_1^{-1}$,
$\lambda\in\dot{\Sigma}_0$,
\begin{align}
c_{11}(\lambda_+)
&=
\frac{\alpha_1}
{\sqrt{2\ii A_1l_1(\lambda_+)}}
+\ord\!\left(\sqrt{l_1(\lambda_+)}\right),\\
c_{21}(\lambda_+)
&=
\frac{\ii\overline{\alpha_1}}
{\sqrt{2\ii A_1l_1(\lambda_+)}}
+\ord\!\left(\sqrt{l_1(\lambda_+)}\right).
\end{align}
In particular, the corresponding normalized limits coincide:
\begin{align}
\lim_{\substack{\lambda\to A_1^{-1}\\
\lambda\in\dot{\Sigma}_1}}
\sqrt{2\ii A_1l_1(\lambda)}\,c_{11}(\lambda)
&=
\lim_{\substack{\lambda\to A_1^{-1}\\
\lambda\in\dot{\Sigma}_0}}
\sqrt{2\ii A_1l_1(\lambda_+)}\,c_{11}(\lambda_+)
=\alpha_1,\\
\lim_{\substack{\lambda\to A_1^{-1}\\
\lambda\in\dot{\Sigma}_1}}
\sqrt{2\ii A_1l_1(\lambda)}\,c_{21}(\lambda)
&=
\lim_{\substack{\lambda\to A_1^{-1}\\
\lambda\in\dot{\Sigma}_0}}
\sqrt{2\ii A_1l_1(\lambda_+)}\,c_{21}(\lambda_+)
=\ii\overline{\alpha_1}.
\end{align}

As $\lambda\to-A_1^{-1}$, $\lambda\in\dot{\Sigma}_1$,
\begin{align}
c_{11}(\lambda)
&=
-\frac{\overline{\alpha_1}}
{\sqrt{2\ii A_1l_1(\lambda)}}
+\ord\!\left(\sqrt{l_1(\lambda)}\right),\\
c_{21}(\lambda)
&=
-\frac{\ii\alpha_1}
{\sqrt{2\ii A_1l_1(\lambda)}}
+\ord\!\left(\sqrt{l_1(\lambda)}\right),
\end{align}
whereas, as $\lambda\to-A_1^{-1}$,
$\lambda\in\dot{\Sigma}_0$,
\begin{align}
c_{11}(\lambda_+)
&=
-\frac{\overline{\alpha_1}}
{\sqrt{2\ii A_1l_1(\lambda_+)}}
+\ord\!\left(\sqrt{l_1(\lambda_+)}\right),\\
c_{21}(\lambda_+)
&=
-\frac{\ii\alpha_1}
{\sqrt{2\ii A_1l_1(\lambda_+)}}
+\ord\!\left(\sqrt{l_1(\lambda_+)}\right).
\end{align}
Thus, the corresponding normalized limits also coincide:
\begin{align}
\lim_{\substack{\lambda\to-A_1^{-1}\\
\lambda\in\dot{\Sigma}_1}}
\sqrt{2\ii A_1l_1(\lambda)}\,c_{11}(\lambda)
&=
\lim_{\substack{\lambda\to-A_1^{-1}\\
\lambda\in\dot{\Sigma}_0}}
\sqrt{2\ii A_1l_1(\lambda_+)}\,c_{11}(\lambda_+)
=-\overline{\alpha_1},\\
\lim_{\substack{\lambda\to-A_1^{-1}\\
\lambda\in\dot{\Sigma}_1}}
\sqrt{2\ii A_1l_1(\lambda)}\,c_{21}(\lambda)
&=
\lim_{\substack{\lambda\to-A_1^{-1}\\
\lambda\in\dot{\Sigma}_0}}
\sqrt{2\ii A_1l_1(\lambda_+)}\,c_{21}(\lambda_+)
=-\ii\alpha_1.
\end{align}

As $\lambda\to A_2^{-1}$,
$\lambda\in\dot{\Sigma}_0$,
\begin{align}
c_{11}(\lambda_+)
&=
\frac{\ii\gamma_2}
{\sqrt{2\ii A_2l_2(\lambda_+)}}
+\ord\!\left(\sqrt{l_2(\lambda_+)}\right),\\
c_{21}(\lambda_+)
&=
-\frac{\ii\gamma_2}
{\sqrt{2\ii A_2l_2(\lambda_+)}}
+\ord\!\left(\sqrt{l_2(\lambda_+)}\right),
\end{align}
whereas, as $\lambda\to-A_2^{-1}$,
$\lambda\in\dot{\Sigma}_0$,
\begin{align}
c_{11}(\lambda_+)
&=
\frac{\ii\gamma_2}
{\sqrt{2\ii A_2l_2(\lambda_+)}}
+\ord\!\left(\sqrt{l_2(\lambda_+)}\right),\\
c_{21}(\lambda_+)
&=
\frac{\ii\gamma_2}
{\sqrt{2\ii A_2l_2(\lambda_+)}}
+\ord\!\left(\sqrt{l_2(\lambda_+)}\right).
\end{align}
\end{corollary}

\subsection{Pre-RH problems parametrised by $(x,t)$}  \label{sec:RH-l}

The identity
\[
1=|c_{11}(\lambda)|^2-|c_{21}(\lambda)|^2,
\qquad \lambda\in\dot{\Sigma}_1,
\]
implies that $\alpha_1\neq0$. In contrast, for the inner branch points we
restrict attention to the generic case
\[
\gamma_2\neq0.
\]

\subsubsection{Right pre-RH problem}

\begin{notations*}
We denote
\begin{equation}\label{rho-l}
    r(\lambda)=\begin{cases}
    \frac{c_{21}(\lambda)}{c_{11}(\lambda)},\quad\lambda\in\dot\Sigma_1,\\
    \frac{c_{21}(\lambda_+)}{c_{11}(\lambda_+)},\quad\lambda\in\dot\Sigma_0
    \end{cases}
\end{equation}
The function $r(\lambda)$ is the reflection coefficient for  Right RH problem associated with the initial profile $u_0(x)$. It is determined through the direct scattering transform: the scattering coefficients $c_{11}(\lambda)$ and $c_{21}(\lambda)$ are computed from the eigenfunctions $\Psi_j(x,0,\lambda)$ of the $x$-equation of the Lax pair corresponding to $u_0(x)$, and $r(\lambda)$ is then defined by \eqref{rho-l}.
\end{notations*}

\begin{remark}\label{ratbranchpoints}   
The reflection coefficient $r(\lambda)$ satisfies the following properties
\begin{enumerate}   
        \item $r(\lambda)\in C(\Sigma_2)$;

        \item at the endpoints $\pm A_2^{-1}$, the one-sided limits are
      $  \lim_{\substack{\lambda\to\pm A_2^{-1}\\
        \lambda\in\dot\Sigma_0}}
         r(\lambda)=\mp 1$.

        \item $r(\lambda)=\ord(\frac{1}{\lambda})$ as $\lambda\to\infty$;
    
        \item $r(\lambda)=-\overline{r(-\lambda)}$, $\lambda\in\dot\Sigma_1$;

        \item $r(\lambda)=-\frac{1}{r(-\lambda)}$ and $|r(\lambda)|=1$, $\lambda\in\dot\Sigma_0$.
    \end{enumerate}    

    Moreover, if we additionally assume $u_0(x)-A_1\in H^{6,3}(\mathbb{R}^-)$, $u_0(x)-A_2\in H^{6,3}(\mathbb{R}^+)$, one can prove proceeding as in \cite{L12} that $r(\lambda)\in C^1(\Sigma_2)$ and $\frac{\dd}{\dd \lambda}r(\lambda)\in L^2(\Sigma_2)$.

    Indeed, differentiating the columnwise Volterra equations
\eqref{eq_l} with respect to the spectral parameter and applying the
standard Volterra resolvent estimates, the weighted assumptions provide
the spatial moments required to obtain uniform and $L^2$ bounds for the
differentiated Jost solutions. Substitution of these estimates into the
determinant representations \eqref{scatcoeff-l} yields the stated
regularity of the reflection coefficient.
\end{remark}

Introducing the matrix-valued function 
\begin{equation}\label{n}
N(x,t,\lambda)=-\sqrt{\frac{1}{2}}
    \begin{pmatrix}
    1&\ii\\
    \ii&1
    \end{pmatrix}\begin{cases}
\left( \frac{\Psi_1^{(1)}(\lambda)}{c_{11}(\lambda)},\Psi_2^{(2)}(\lambda)\right)\eul^{f_2(\lambda)\sigma_3},\quad\lambda\in\mathbb{C}^+,\\
\left( \Psi_2^{(1)}(\lambda),\frac{\Psi_1^{(2)}(\lambda)}{c_{22}(\lambda)}\right)\eul^{f_2(\lambda)\sigma_3},\quad\lambda\in\mathbb{C}^-,
\end{cases}
\end{equation} 
with $f_2$ defined in \eqref{p_i-l}, we obtain a matrix-valued function satisfying the following properties:

\begin{enumerate}[\textbullet]
\item The \emph{jump} property
\begin{subequations}\label{jump}
\begin{equation}\label{jump_N}
    N^+(x,t,\lambda) =N^-(x,t,\lambda) G(x,t,\lambda), \quad \lambda\in\mathbb{R}\setminus\left\{\pm\frac{1}{A_1},\pm\frac{1}{A_2}\right\}
\end{equation}
where 
\begin{equation}\label{G}
    G(x,t,\lambda)=\begin{cases}\begin{pmatrix}
\eul^{-f_2(x,t,\lambda)}&0\\
0&\eul^{f_2(x,t,\lambda)}
\end{pmatrix}G_0\begin{pmatrix}
\eul^{f_2(x,t,\lambda)}&0\\
0&\eul^{-f_2(x,t,\lambda)}
\end{pmatrix},\quad \lambda\in\dot\Sigma_2\setminus\{\pm\frac{1}{A_1}\}\\
-\ii\sigma_1, \quad \lambda\in\dot\Gamma_2
\end{cases}
\end{equation}
with
\begin{equation}\label{G_0_A}
G_0(\lambda)=\begin{cases}
\begin{pmatrix}
1-|r(\lambda)|^2&-\overline{r(\lambda)}\\
r(\lambda)&1
\end{pmatrix}, \quad \lambda\in\dot\Sigma_1,\\
\begin{pmatrix}
0&-\frac{1}{r(\lambda)}\\
r(\lambda)&1
\end{pmatrix}, \quad \lambda\in\dot\Sigma_0.
\end{cases}
\end{equation}
\end{subequations}

\item
The \emph{normalization} property:
\begin{equation}\label{norm-n}
N(x,t,\lambda)=I+\ord(\frac{1}{\lambda}), \quad \lambda\to\infty.
\end{equation}

\item
\emph{Singularity} property:
the singularities of $
N(x,t,\lambda)$ at $\pm\frac{1}{A_j}$ are of order not greater than $\frac{1}{4}$  (cf.\ \cite{KST22}).
\end{enumerate}

\subsubsection{Left pre-RH problem }

\begin{notations*}
We denote
\begin{equation}\label{rho-left}
    \tilde r(\lambda)=\begin{cases}
        \frac{\overline{c_{21}(\lambda)}}{c_{11}(\lambda)},\quad\lambda\in\dot\Sigma_1,\\
        \frac{1}{c_{21}(\lambda_+)c_{11}(\lambda_+)}=\frac{\ii}{|c_{11}(\lambda_+)|^2}, \quad\lambda\in\dot\Sigma_0
    \end{cases}
\end{equation}
The function $\tilde r(\lambda)$ serves as the reflection coefficient for  the left  RH problem. It is uniquely determined by the initial profile $u_0(x)$ through the direct scattering transform.
\end{notations*}

\begin{remark}  The reflection coefficient $\tilde r(\lambda)$ satisfies the  following properties
\begin{enumerate}\label{ratbranchpoints_left}    

         \item $\tilde r(\lambda)$ is continuous on $\dot\Sigma_1\cup\dot\Sigma_0$;

             \item at the endpoints $\pm A_1^{-1}$, the one-sided limits are
        \[
        \lim_{\substack{\lambda\to A_1^{-1}\\
        \lambda\in\dot\Sigma_1}}
        \tilde r(\lambda)=1,
        \qquad
        \lim_{\substack{\lambda\to-A_1^{-1}\\
        \lambda\in\dot\Sigma_1}}
        \tilde r(\lambda)=-1,
        \]
        whereas
        \[
        \lim_{\substack{\lambda\to\pm A_1^{-1}\\
        \lambda\in\dot\Sigma_0}}
        \tilde r(\lambda)=0;
        \]

        \item  at the endpoints $\pm A_2^{-1}$, the one-sided limits are
      $  \lim_{\substack{\lambda\to\pm A_2^{-1}\\
        \lambda\in\dot\Sigma_0}}
        \tilde r(\lambda)=0$.

            \item $\tilde r(\lambda)=\ord(\frac{1}{\lambda})$ as $\lambda\to\infty$;
            
        \item $\tilde r(\lambda)=-\overline{\tilde r(-\lambda)}$, $\lambda\in\dot\Sigma_1$;

        \item $\tilde r(\lambda)=\tilde r(-\lambda)$, $\lambda\in\dot\Sigma_0$.
    \end{enumerate}
\end{remark}

To construct  the left  RH problem, we introduce the matrix-valued function 
\begin{equation}\label{til n}
\tilde N(x,t,\lambda)=-\sqrt{\frac{1}{2}}
    \begin{pmatrix}
    1&\ii\\
    \ii&1
    \end{pmatrix}\begin{cases}
\left(\Psi_1^{(1)}(\lambda), \frac{\Psi_2^{(2)}(\lambda)}{c_{11}(\lambda)}\right)\eul^{f_1(\lambda)\sigma_3},\quad\lambda\in\mathbb{C}^+,\\
\left( \frac{\Psi_2^{(1)}(\lambda)}{c_{22}(\lambda)}, \Psi_1^{(2)}(\lambda)\right)\eul^{f_1(\lambda)\sigma_3},\quad\lambda\in\mathbb{C}^-
\end{cases}
\end{equation} 
with $f_1$ defined in \eqref{p_i-l}.

The matrix-valued function $\tilde N(x,t,\lambda)$   satisfies the following properties:
\begin{enumerate}[\textbullet]
\item The \emph{jump} property
\begin{subequations}\label{jump_N-left}
\begin{equation}\label{jump_N-left_}
    \tilde N^+(x,t,\lambda) =\tilde N^-(x,t,\lambda) \tilde G(x,t,\lambda), \quad \lambda\in\mathbb{R}\setminus\left\{\pm\frac{1}{A_1},\pm\frac{1}{A_2}\right\}
\end{equation}
where 
\begin{equation}\label{tilde_G}
    \tilde G(x,t,\lambda)=\begin{cases}\begin{pmatrix}
\eul^{-f_1(x,t,\lambda)}&0\\
0&\eul^{f_1(x,t,\lambda)}
\end{pmatrix}\begin{pmatrix}
1&-\tilde r(\lambda)\\
\overline{\tilde r(\lambda)}&1-|\tilde r(\lambda)|^2
\end{pmatrix}\begin{pmatrix}
\eul^{f_1(x,t,\lambda)}&0\\
0&\eul^{-f_1(x,t,\lambda)}
\end{pmatrix},\quad \lambda\in\dot\Sigma_1,\\
-\ii\begin{pmatrix}
0&1\\
1&e^{-2f_1(x,t,\lambda_+)} \tilde r(\lambda)
\end{pmatrix},\quad \lambda\in\dot\Sigma_0,\\
-\ii\sigma_1, \quad \lambda\in\dot\Gamma_2.
\end{cases}
\end{equation}
\end{subequations}

\item
The \emph{normalization} property:
\begin{equation}\label{norm-til-n}
\tilde N(x,t,\lambda)=I+\ord(\frac{1}{\lambda}), \quad \lambda\to\infty.
\end{equation}

\item
\emph{Singularity} property:
the singularities of $
\tilde N(x,t,\lambda)$ at $\pm\frac{1}{A_j}$ are of order not greater than $\frac{1}{4}$  (cf.\ \cite{KST22}).
\end{enumerate}

\begin{remark}\label{rem_con_n} The matrix-valued functions $N(x,t,\lambda)$ and $\tilde N(x,t,\lambda)$ are related as

\begin{equation*}
    N(x,t,\lambda)=\tilde N(x,t,\lambda)\begin{cases}
     \begin{pmatrix}\frac{1}{c_{11}(\lambda)e^{f_1(x,t,\lambda)-f_2(x,t,\lambda)}}&0\\0&c_{11}(\lambda)e^{f_1(x,t,\lambda)-f_2(x,t,\lambda)}\end{pmatrix},\quad\lambda\in\mathbb{C}^+,\\
         \begin{pmatrix}c_{22}(\lambda)e^{f_2(x,t,\lambda)-f_1(x,t,\lambda)}&0\\0&\frac{1}{c_{22}(\lambda)e^{f_2(x,t,\lambda)-f_1(x,t,\lambda)}}\end{pmatrix},\quad\lambda\in\mathbb{C}^-.
    \end{cases}
\end{equation*}

\end{remark}


\subsection{The right RH problem in the $\BS{(y,t)}$ scale}
The construction of jump matrix involves $f_2(x,t,\lambda)$ (see \eqref{p_i-l}), which in turn involves $u(x,t)$, suggesting the introduction of a new spatial variable
\begin{equation}\label{shkala_}
    y(x,t)=x-\frac{1}{A_2}\int_x^{+\infty}(m(\xi,t)-A_2)\dd\xi-A_2^2 t,
\end{equation}
and defining $\hat N(y,t,\lambda)$ by $N(x,t,\lambda)=\hat N(y(x,t),t,\lambda)$, the jump condition \eqref{jump_N} can be rewritten in the $(y,t)$ variables as
\begin{subequations} \label{Gp-y}
\begin{equation}\label{jump-y-l}
\hat N^+(y,t,\lambda)=\hat N^-(y,t,\lambda)\hat G(y,t,\lambda),\quad\lambda\in \mathbb{R}\setminus\left\{\pm\frac{1}{A_1},\pm\frac{1}{A_2}\right\},
\end{equation}
where
\begin{equation}\label{G-_A}
 \hat G(y,t,\lambda)=\begin{cases}
\hat G_{\Sigma_1}(y,t,\lambda), \quad \lambda\in\dot\Sigma_1,\\
\hat G_{\Sigma_0}(y,t,\lambda), \quad \lambda\in\dot\Sigma_0,\\
\hat G_{\Gamma_2}, \quad \lambda\in\dot\Gamma_2.
\end{cases}  
\end{equation}
with
\begin{equation}\label{G_Gamma-l-}
  \hat G_{\Gamma_2}=-\ii\sigma_1\end{equation}
and
\begin{equation}\label{G_Sigma-l-_A}
\hat G_{\Sigma_i}(y,t,\lambda)=
\begin{pmatrix}
\eul^{-\hat f_2(y,t,\lambda)}&0\\
0&\eul^{\hat f_2(y,t,\lambda)}
\end{pmatrix}G_0(\lambda)
\begin{pmatrix}
\eul^{\hat f_2(y,t,\lambda_+)}&0\\
0&\eul^{-\hat f_2(y,t,\lambda_+)},
\end{pmatrix}
\end{equation}
where $G_0(\lambda)$ is defined by \eqref{G_0_A} and
\begin{equation}\label{f-y_A}
\hat f_2(y,t,\lambda) := \frac{\ii A_2 l_2(\lambda)}{2}\left(y-\frac{2t}{\lambda^2}\right).
\end{equation}
\end{subequations}

Noticing that the normalization condition \eqref{norm-n} and the singularity conditions at $\lambda=\pm \frac{1}{A_j}$ hold when using the new scale $(y,t)$, we arrive at the basic right RH problem.


\textbf{Basic right RH problem.}
Given $r(\lambda)$ for $\lambda\in \dot\Sigma_2\}$,  find a piecewise (with respect to~$\dot\Sigma_2$) analytic, $2\times 2$ matrix-valued function $\hat N(y,t,\lambda)$ satisfying the following conditions:
\begin{enumerate}
\item
The jump condition \eqref{Gp-y} across $\mathbb{R}\setminus\left\{\pm\frac{1}{A_1},\pm\frac{1}{A_2}\right\}$.

\item The \emph{normalization} condition:
\begin{equation}\label{norm-n-hat_A}
\hat N(y,t,\lambda)=I+\ord(\frac{1}{\lambda}), \quad \lambda\to\infty.
\end{equation}

\item
\emph{Singularity} conditions:
the singularities of $
\hat N(y,t,\lambda)$ at $\pm\frac{1}{A_j}$ are of order not greater than $\frac{1}{4}$.
\end{enumerate}

\begin{remark} If a solution of the RH problem exists, then it is unique and satisfies the following properties:
\begin{itemize}

\item 
The \emph{determinant} condition $\det \hat N\equiv 1$. 
    \item 
\emph{Symmetries}:
\begin{subequations}\label{sym-N_hat}
\begin{alignat}{4}
\hat N(\lambda)&=\sigma_2 \hat N                 (-\lambda)\sigma_2,\qquad \hat N(\lambda)=\sigma_1\overline{\hat N(\overline{\lambda})}\sigma_1,\qquad \lambda\in\mathbb{C}^+,\\
\hat N(\lambda_+)&=-\sigma_2 \hat N((-\lambda)_+)\sigma_3,\qquad \hat N(\lambda_+)=-\ii\sigma_1\overline{\hat N(\lambda_+)},\qquad \lambda\in\dot\Gamma_2,\\
\hat N(\lambda_-)&=-\sigma_2 \hat N((-\lambda)_-)\sigma_3,\qquad \hat N(\lambda_-)=\ii\sigma_1\overline{\hat N(\lambda_-)},\qquad \lambda\in\dot\Gamma_2.
\end{alignat}
\end{subequations}
where $\hat N(\lambda)\equiv \hat N(y,t,\lambda)$.

\end{itemize}
\end{remark}

\subsection{Left RH problem in the $\BS{(\tilde y,t)}$ scale}
The construction of jump matrix involves $f_1(x,t,\lambda)$ (see \eqref{p_i-l}), which in turn involves $u(x,t)$, suggesting the introducing the new space variable $\tilde y(x,t)$ by \begin{equation}\label{shkala-l_left_}
    \tilde y(x,t)=x+\frac{1}{A_1}\int_{-\infty}^x(m(\xi,t)-A_1)\dd\xi-A_1^2 t,
\end{equation} 
and defining $\hat{\tilde N}(\tilde y,t,\lambda)$ by $\tilde N(x,t,\lambda)=\hat{\tilde N}(\tilde y(x,t),t,\lambda)$, the jump condition \eqref{jump_N-left} becomes
\begin{subequations} \label{Gp-y_left}
\begin{equation}\label{jump-y-l_left}
\hat {\tilde N}^+(\tilde y,t,\lambda)=\hat  {\tilde N}^-(\tilde y,t,\lambda)\hat {\tilde G}(\tilde y,t,\lambda),\quad\lambda\in \mathbb{R}\setminus\left\{\pm\frac{1}{A_1},\pm\frac{1}{A_2}\right\},
\end{equation}
where
\begin{equation}\label{tilde_G__A}
    \hat {\tilde G}(\tilde y,t,\lambda)=\begin{cases}\begin{pmatrix}
\eul^{-\hat f_1(\tilde y,t,\lambda)}&0\\
0&\eul^{\hat f_1(\tilde y,t,\lambda)}
\end{pmatrix}\begin{pmatrix}
1&-\tilde r(\lambda)\\
\overline{\tilde r(\lambda)}&1-|\tilde r(\lambda)|^2
\end{pmatrix}\begin{pmatrix}
\eul^{\hat f_1(\tilde y,t,\lambda)}&0\\
0&\eul^{-\hat f_1(\tilde y,t,\lambda)}
\end{pmatrix},\quad \lambda\in\dot\Sigma_1,\\
-\ii\begin{pmatrix}
\eul^{-\hat f_1(\tilde y,t,\lambda_-)}&0\\
0&\eul^{\hat f_1(\tilde y,t,\lambda_-)}
\end{pmatrix}\begin{pmatrix}
0&1\\
1&\tilde r(\lambda_+)
\end{pmatrix}\begin{pmatrix}
\eul^{\hat f_1(\tilde y,t,\lambda_+)}&0\\
0&\eul^{-\hat f_1(\tilde y,t,\lambda_+)}
\end{pmatrix},\quad \lambda\in\dot\Sigma_0,\\
-\ii\sigma_1, \quad \lambda\in\dot\Gamma_2.
\end{cases}
\end{equation} and
\begin{equation}\label{f-y_left_A}
\hat f_1(\tilde y,t,\lambda) := \frac{\ii A_1 l_1(\lambda)}{2}\left(\tilde y-\frac{2t}{\lambda^2}\right).
\end{equation}
\end{subequations}

Noticing that the normalization condition \eqref{norm-n} and the singularity conditions at $\lambda=\pm \frac{1}{A_j}$ hold when using the new scale $(\tilde y,t)$, we arrive at the basic RH problem.

\textbf{Basic left RH problem.}
Given $\tilde r(\lambda)$ for $\lambda\in \dot\Sigma_2\setminus\{\pm\frac{1}{A_1}\}$, find a piecewise (with respect to~$\dot\Sigma_2$) analytic, $2\times 2$ matrix-valued function $\hat {\tilde N}(\tilde y,t,\lambda)$ satisfying the following conditions:
\begin{enumerate}
\item
The jump condition \eqref{Gp-y_left} across $\mathbb{R}\setminus\left\{\pm\frac{1}{A_1},\pm\frac{1}{A_2}\right\}$.
\item The \emph{normalization} condition:
\begin{equation}\label{norm-til_n-hat}
\hat {\tilde N}(\tilde y,t,\lambda)=I+\ord(\frac{1}{\lambda}), \quad \lambda\to\infty.
\end{equation}

\item
\emph{Singularity} conditions:
the singularities of $
\hat {\tilde N}(\tilde y,t,\lambda)$ at $\pm\frac{1}{A_j}$ are of order no greater than $\frac{1}{4}$.
\end{enumerate}

\begin{remark} If a solution to the RH problem exists, then it is unique and satisfies the following properties:
\begin{itemize}
\item
The \emph{determinant} condition $\det \hat {\tilde N}\equiv 1$. 
    \item 
\emph{Symmetries}:
\begin{subequations}\label{hat_sym-til_N_hat_App}
\begin{alignat}{4}
\hat {\tilde N}(\lambda)&=\sigma_2 \hat {\tilde N}                 (-\lambda)\sigma_2,\qquad \hat {\tilde N}(\lambda)=\sigma_1\overline{\hat {\tilde N}(\overline{\lambda})}\sigma_1,\qquad \lambda\in\mathbb{C}^+,\\
\hat {\tilde N}(\lambda_+)&=-\sigma_2 \hat {\tilde N}((-\lambda)_+)\sigma_3,\qquad \hat {\tilde N}(\lambda_+)=-\ii\sigma_1\overline{\hat {\tilde N}(\lambda_+)},\qquad \lambda\in\dot\Gamma_2,\\
\hat {\tilde N}(\lambda_-)&=-\sigma_2 \hat {\tilde N}((-\lambda)_-)\sigma_3,\qquad \hat {\tilde N}(\lambda_-)=\ii\sigma_1\overline{\hat {\tilde N}(\lambda_-)},\qquad \lambda\in\dot\Gamma_2.
\end{alignat}
\end{subequations}
where $\hat{\tilde N}(\lambda)\equiv \hat {\tilde N}(\tilde y,t,\lambda)$

\end{itemize}
\end{remark}

\subsection{Connection between left and right Riemann--Hilbert problems}\label{app:A.6}

We first observe that
\[
y_x=\frac{A_1}{A_2}\tilde y_x,
\qquad
y_t=\frac{A_1}{A_2}\tilde y_t.
\]
It follows that
\[
\tilde y(x,t)=\frac{A_2}{A_1} y(x,t)+C,
\]
where $C$ is independent of both $x$ and $t$. Evaluating this identity at
$x=t=0$, we obtain
\begin{equation}
\label{const}    
C=\frac{1}{A_1}\left(
\int_0^\infty (m(\xi,0)-A_2)\,d\xi
+
\int_{-\infty}^0 (m(\xi,0)-A_1)\,d\xi
\right).
\end{equation}
Consequently,
\begin{equation}
    \label{y_til_y}
    \tilde y(x,t)
=
\frac{A_2}{A_1} y(x,t)
+
\frac{1}{A_1}\left(
\int_0^\infty (m(\xi,0)-A_2)\,d\xi
+
\int_{-\infty}^0 (m(\xi,0)-A_1)\,d\xi
\right).
\end{equation}

Our goal is now to show that the solution of  the right RH problem, $\hat N(y,t,\lambda)$, is related to the solution of  the left  RH problem, $\hat{\tilde N}(\tilde y(y),t,\lambda)$, where
\begin{equation}
    \label{ch_y_til_y}
    \tilde y(y)=\frac{A_2}{A_1} y +C
\end{equation}
with
$C$ given by \eqref{const}, which can be computed from the initial condition.
In particular, 
\[
\hat f_1(\tilde y(y),t,\lambda)=\frac{\ii A_1 l_1(\lambda)}{2}\left(\frac{A_2}{A_1}y+C-\frac{2t}{\lambda^2}\right).
\]

This relation takes the form of multiplication by a matrix factor:
 \begin{subequations}\label{connection}
     \begin{align}
     &\hat N(y,t,\lambda)=\hat{\tilde N}(\tilde y(y),t,\lambda)P(y,t,\lambda)\\\label{connection_P}
         & P(y,t,\lambda)=\begin{cases}
        \begin{pmatrix}\frac{1}{c_{11}(\lambda)e^{\hat f_1(\tilde y(y),t,\lambda)-\hat f_2(y,t,\lambda)}}&0\\0&c_{11}(\lambda)e^{\hat f_1(\tilde y(y),t,\lambda)-\hat f_2(y,t,\lambda)}\end{pmatrix},\qquad \lambda\in\mathbb{C}^+\\
\begin{pmatrix}c_{22}(\lambda)e^{\hat f_2(y,t,\lambda)-\hat f_1(\tilde y(y),t,\lambda)}&0\\0&\frac{1}{c_{22}(\lambda)e^{\hat f_2(y,t,\lambda)-\hat f_1(\tilde y(y),t,\lambda)}}\end{pmatrix},\quad\lambda\in\mathbb{C}^-.
     \end{cases}
     \end{align}
 \end{subequations}

To verify this, it suffices to show that $\hat{M}(y,t,\lambda)$, defined by
 \begin{equation}
   \hat{ M}(y,t,\lambda)=\hat{\tilde N}(\tilde y(y),t,\lambda)P(y,t,\lambda),
 \end{equation}
 satisfies  the right Riemann--Hilbert problem.
 
Indeed,
  \begin{itemize}
     \item the jump condition and singularity conditions match by construction;

     \item noticing that $c_{11}(\lambda)e^{\hat f_1(\tilde y(y),t,\lambda)-\hat f_2(y,t,\lambda)}=1+O(\frac{1}{\lambda})$ in $\mathbb{C}^+$, the normalization condition is satisfied.

     \end{itemize}

     \subsection{Eigenfunctions near $\lambda=0$}\label{app:Eigen_0-l}

As in the outer-cut case, we introduce reference solutions $\tilde  \Psi_j^0(x,t,\lambda)$ solving the Volterra equation
\begin{equation}\label{inteq0_lam-l}
\begin{aligned}
&\tilde  \Psi_j^0(x,t,\lambda)=I+\int_{(-1)^j\infty}^x
\eul^{\frac{-\ii A_jl_j(\lambda)}{2}(x-\xi)\sigma_3}\hat W_j^0(\xi,t,\lambda)
\tilde  \Psi_j^0(\xi,t,\lambda)
\eul^{\frac{\ii A_jl_j(\lambda)}{2}(x-\xi)\sigma_3}\dd\xi,
\\
&\hat W_j^0(x,t,\lambda)=\frac{(m-A_j)\lambda}{2\ii l_j(\lambda)}
\begin{pmatrix}\lambda & \frac1{A_j}\\ -\frac1{A_j} & -\lambda\end{pmatrix}.
\end{aligned}
\end{equation}
They are related to $\tilde  \Psi_j(x,t,\lambda)$ by
\begin{equation}\label{inf_0_rel-l}
\tilde  \Psi_j(x,t,\lambda)=\tilde  \Psi_j^0(x,t,\lambda)\,\eul^{-f_j^0(x,t,\lambda)\sigma_3}C_j(\lambda)\,\eul^{f_j(x,t,\lambda)\sigma_3},
\end{equation}
with
\begin{equation}\label{f_j_0}
f_j^0(x,t,\lambda):=\tfrac12\ii A_jl_j(\lambda)\Big(x-2\big(\tfrac{A_j^2}2+\tfrac1{\lambda^2}\big)t\Big).  
\end{equation}

Since $f_j-f_j^0=\frac{\ii l_j(\lambda)}2\int^x_{(-1)^j\infty}(m(\xi,t)-A_j)\dd\xi$, matching asymptotics as $x\to(-1)^j\infty$ forces $C_j\equiv I$, so
\begin{equation}\label{inf_0_rel-l-2}
\tilde  \Psi_j(x,t,\lambda)=\tilde  \Psi_j^0(x,t,\lambda)\,
\eul^{\frac{\ii l_j(\lambda)}2\int_{(-1)^j\infty}^x (m(\xi,t)-A_j)\dd\xi\,\sigma_3}.
\end{equation}

To recover the solution of the mCH equation from $N$ ($\tilde N$), we use its expansion at $\lambda=0$ up to and including the $O(\lambda)$ term. Since $\hat W_j^0(x,t,0)\equiv0$, \eqref{inteq0_lam-l} gives $\tilde  \Psi_j^0(x,t,0)\equiv I$, and expanding further by Neumann series,
\[
\tilde  \Psi_j^0(x,t,\lambda)=\begin{cases}
I+\lambda\begin{pmatrix}
0&-\int_{(-1)^j\infty}^x \eul^{x-\xi}\frac{m-A_j}{2}\dd \xi\\
\int_{(-1)^j\infty}^x \eul^{-(x-\xi)}\frac{m-A_j}{2}\dd \xi&0
\end{pmatrix}+\ord(\lambda^2),&\lambda\in\mathbb C^+,\\[1mm]
I+\lambda\begin{pmatrix}
0&\int_{(-1)^j\infty}^x \eul^{-(x-\xi)}\frac{m-A_j}{2}\dd \xi\\
-\int_{(-1)^j\infty}^x \eul^{x-\xi}\frac{m-A_j}{2}\dd \xi&0
\end{pmatrix}+\ord(\lambda^2),&\lambda\in\mathbb C^-.
\end{cases}
\]
In particular, combining this with \eqref{inf_0_rel-l-2} and the expansion \eqref{H_at_0}
yields
\begin{subequations}
\begin{equation}\label{c-expand}
\tilde c_{11}(\lambda)=\eul^{-\frac{1}{2A_1}\int_{-\infty}^x(m-A_1)\dd\xi-\frac{1}{2A_2}\int_x^{\infty}(m-A_2)\dd\xi}+\ord(\lambda^2),\ \lambda\in\mathbb C^+,
\end{equation}
\begin{equation}\label{c-expand-2}
\tilde c_{22}(\lambda)=\eul^{-\frac{1}{2A_1}\int_{-\infty}^x(m-A_1)\dd\xi-\frac{1}{2A_2}\int_x^{\infty}(m-A_2)\dd\xi}+\ord(\lambda^2),\ \lambda\in\mathbb C^-.
\end{equation}
\end{subequations}


Recall that $N(x,t,\lambda)$ is defined through \eqref{n}. Therefore, we have the following expansion of $N(x,t,\lambda)$ as $\lambda\to 0$ in $\mathbb{C}^+$:

\begin{align}\label{N-at-0-+}
&N(x,t,\lambda)=-\sqrt{\frac{1}{2}}
    \begin{pmatrix}
    1&\ii\\
    \ii&1
    \end{pmatrix}\left(\ii\begin{pmatrix}0&\eul^{-\frac{1}{2A_2}\int_x^{\infty}(m-A_2)\dd\xi}\\ \eul^{\frac{1}{2A_2}\int_x^{\infty}(m-A_2)\dd\xi}&0\end{pmatrix}+
    \right.\\
&\left.+\ii\lambda\begin{pmatrix}
\big(\int_{-\infty}^x \eul^{-(x-\xi)}\frac{m-A_1}2\dd\xi+\tfrac{A_1}2\big)\eul^{\frac{1}{2A_2}\int_x^{\infty}(m-A_2)\dd\xi}&0\\
0&\big(\int_x^{\infty} \eul^{x-\xi}\frac{m-A_2}2\dd\xi+\tfrac{A_2}2\big)\eul^{-\frac{1}{2A_2}\int_x^{\infty}(m-A_2)\dd\xi}
\end{pmatrix}+\ord(\lambda^2)\right).
\end{align}

\begin{remark}\label{rem_sol-l}
notice that
\[
u(x,t):=\frac{A_1+A_2}2+\Big(\int_{-\infty}^x \eul^{-(x-\xi)}\frac{m-A_1}2\dd\xi+\int_x^{\infty} \eul^{x-\xi}\frac{m-A_2}2\dd\xi\Big)
\]
 solves $m=u-u_{xx}$ with $\lim_{x\to(-1)^j\infty}u=A_j$, moreover, as $x\to(-1)^j\infty$
\[
u_x(x,t)=\frac{A_2-A_1}2+\Big(-\int_{-\infty}^x \eul^{-(x-\xi)}\frac{m-A_1}2\dd\xi+\int_x^{\infty} \eul^{x-\xi}\frac{m-A_2}2\dd\xi\Big)\to0.
\]

\end{remark}

\subsection{Recovering $u(x,t)$ from the solution of the RH problems}\label{sec:recover}

Similarly to \cite{KST22}, in order to recover the solution of the Cauchy problem \eqref{mCH1-ic} from the solution of the associated RH problem, we analyze the behavior of the RH solution, parameterized by $x$ and $t$, as $\lambda\to0$.

\subsubsection{Recovering $u(x,t)$ from the solution of  the right RH problems}

 The right RH problem yields the following result.

\begin{proposition}\label{prop-recover_2_}
Let 
\begin{equation}\label{M-hat-expand_A}
\hat N(y,t,\lambda)=-\sqrt{\frac{1}{2}}\left(\ii\begin{pmatrix}
\ii \hat a_1^{-1}(y,t)& \hat a_1(y,t)\\ \hat a_1^{-1}(y,t)&\ii \hat a_1(y,t)
\end{pmatrix}+\ii\lambda\begin{pmatrix}
\hat a_2(y,t)&\ii \hat a_3(y,t)\\\ii \hat a_2(y,t)&\hat a_3(y,t)
\end{pmatrix}  \right)
+\ord(\lambda^2)
\end{equation}
be the expansion of the solution of the RH problem \eqref{Gp-y}--\eqref{norm-n-hat_A} near $\lambda=0$ in $\mathbb{C}^+$. Then the solution $u(x,t)$ of the Cauchy problem \eqref{mCH1-ic} has the parametric representation
\begin{subequations}\label{recover-2}
\begin{align}\label{u_(y,t)_A}
&\hat u(y,t)=\hat a_1(y,t)\hat a_2(y,t)+\hat a_1^{-1}(y,t)\hat a_3(y,t),\\
&x(y,t)=y-2\ln\hat a_1(y,t)+A_2^2 t.
\label{x(y,t)-2_A}
\end{align}
Additionally, 
\begin{equation}
\label{u_x(y,t)-2}
\hat u_x(y,t)=-\hat a_1(y,t)\hat a_2(y,t)+\hat a_1^{-1}(y,t)\hat a_3(y,t).
\end{equation}
\end{subequations}
Here $\hat u(y,t)= u(x(y,t),t)$ and $\hat u_x(y,t)= u_x(x(y,t),t)$.

\end{proposition}
Alternatively, one can express $\hat u_x(y,t)$ 
in terms of the first term in \eqref{M-hat-expand_A} only.
The price to pay is that this expression involves the derivatives of this term. 
\begin{proposition}\label{prop-recover_1}
The $x$-derivative of the solution $u(x,t)$ of the Cauchy problem \eqref{mCH1-ic} has the parametric representation
\begin{subequations}\label{recover-1}
\begin{align}\label{u_x(y,t)}
&\hat u_x(y,t)=-\frac{1}{A_2}\partial_{ty}\ln \hat a_1(y,t),\\
\label{x(y,t)}
&x(y,t)=y-2\ln\hat a_1(y,t)+A_2^2 t.
\end{align}
\end{subequations}

\end{proposition}

\subsubsection{Recovering $u(x,t)$ from the solution of  the left  RH problems}

 The left  RH problem yields the following result

\begin{proposition}\label{prop-recover_2_left}
Let 
\begin{equation}\label{tilM-hat-expand_A}
\hat {\tilde N}(\tilde y,t,\lambda)=-\sqrt{\frac{1}{2}}\left(\ii\begin{pmatrix}
\ii \hat b_1^{-1}(\tilde y,t)& \hat b_1(\tilde y,t)\\ \hat b_1^{-1}(\tilde y,t)&\ii \hat b_1(\tilde y,t)
\end{pmatrix}+\ii\lambda\begin{pmatrix}
\hat b_2(\tilde y,t)&\ii \hat b_3(\tilde y,t)\\\ii \hat b_2(\tilde y,t)&\hat b_3(\tilde y,t)
\end{pmatrix}  \right)
+\ord(\lambda^2)
\end{equation}
be the expansion of the solution of the RH problem \eqref{Gp-y_left}--\eqref{norm-til_n-hat} near $\lambda=0$ in $\mathbb{C}^+$. Then the solution $u(x,t)$ of the Cauchy problem \eqref{mCH1-ic} has the parametric representation

\begin{subequations}\label{recover-2_left}
\begin{align}\label{u_(y,t)_left_A}
&\hat u(\tilde y,t)=\hat b_1(\tilde y,t)\hat b_2(\tilde y,t)+\hat b_1^{-1}(\tilde y,t)\hat b_3(\tilde y,t),\\\label{x(y,t)-2_left_A}
&x(\tilde y,t)=\tilde y-2\ln\hat b_1(\tilde y,t)+A_1^2 t.
\end{align}
Additionally, 
\begin{equation}
\label{u_x(y,t)-2_left}
\hat u_x(\tilde y,t)=-\hat b_1(\tilde y,t)\hat b_2(\tilde y,t)+\hat b_1^{-1}(\tilde y,t)\hat b_3(\tilde y,t).
\end{equation}
\end{subequations}
Here $\hat u(\tilde y,t)= u(x(\tilde y,t),t)$ and $\hat u_x(\tilde y,t)= u_x(x(\tilde y,t),t)$.
\end{proposition}
\begin{proposition}\label{prop-recover_2}
The $x$-derivative of the solution $u(x,t)$ of the Cauchy problem \eqref{mCH1-ic} has the parametric representation
\begin{subequations}\label{recover-1_left}
\begin{align}\label{u_x(y,t)_left}
&\hat u_x(\tilde y,t)=-\frac{1}{A_1}\partial_{t\tilde y}\ln \hat b_1(\tilde y,t),\\
\label{x(y,t)_left}
&x(\tilde y,t)=\tilde y-2\ln\hat b_1(\tilde y,t)+A_1^2 t.
\end{align}
\end{subequations}

\end{proposition}



\section{Basic facts about $L^2$ Riemann--Hilbert problems}
\label{app:B}

In this section, we collect several facts concerning $L^2$
Riemann--Hilbert problems and their associated singular integral
equations; see, for example, \cites{L18,ZH89}. Particular attention is
paid to the case in which the jump matrix is only piecewise continuous.
This is relevant below because the formula defining our jump matrix
changes at the branch points
\[
\lambda=\pm A_1^{-1},\qquad \lambda=\pm A_2^{-1}.
\]

\subsection{Carleson contours and Smirnov classes}

A rectifiable arc $\Sigma\subset\mathbb C$ is called a
\emph{Carleson arc} if
\begin{equation}
\label{Carl_prop}
\sup_{z\in\Sigma}\sup_{r>0}
\frac{|\Sigma\cap B(z,r)|}{r}<\infty,
\end{equation}
where $|\cdot|$ denotes arclength measure and
\[
B(z,r)=\{\zeta\in\mathbb C:\abs{\zeta-z}<r\}.
\]

A finite union $\Gamma$ of oriented Carleson arcs, any two of which
intersect only at common endpoints, is called a
\emph{Carleson jump contour} if
\[
\mathbb C\setminus\Gamma=D_+\cup D_-,
\qquad
D_\pm=\bigcup_j D_{\pm,j},
\]
where the unions on the right contain only finitely many connected
components, each boundary $\partial D_{\pm,j}$ is a Carleson curve, and
the orientation is chosen so that
\[
\Gamma=\partial D_+=-\partial D_-.
\]
The $+$-side of $\Gamma$ is therefore the side lying to the left of the
oriented contour.

Let $D$ be a bounded component of $\mathbb C\setminus\Gamma$. An
analytic function $f$ belongs to the Smirnov class $E^2(D)$ if there
exists a sequence of rectifiable Jordan curves
\[
\{C_n\}_{n=1}^\infty\subset D
\]
such that $C_n$ eventually surrounds every compact subset of $D$ and
\begin{equation}
\label{eq:E2-definition}
\sup_{n\geq1}\int_{C_n}|f(z)|^2\,|dz|<\infty.
\end{equation}

If $D$ is unbounded, we say that $f\in E^2(D)$ if
\[
f\circ\phi^{-1}\in E^2(\phi(D)),
\qquad
\phi(z)=\frac{1}{z-z_0},
\]
for some, and hence for every, $z_0\in\mathbb C\setminus\overline D$.

We also set
\[
\dot E^2(D)
=
\left\{
f\in E^2(D): zf(z)\in E^2(D)
\right\}.
\]
Thus, when $D$ contains a neighborhood of infinity,
$f\in\dot E^2(D)$ expresses the condition that $f$ vanishes at
infinity in the $E^2$ sense. If $D$ is a finite union of components,
membership in $E^2(D)$ and $\dot E^2(D)$ is understood componentwise.

\subsection{Cauchy operators}

We denote by $\mathcal B(L^2(\Gamma))$ the space of bounded linear
operators on $L^2(\Gamma)$ and by $\mathcal F(L^2(\Gamma))$ the set of
Fredholm operators on $L^2(\Gamma)$. Recall that
\[
\operatorname{Ind}:\mathcal F(L^2(\Gamma))\longrightarrow\mathbb Z
\]
is locally constant, and hence is constant on every connected
component of $\mathcal F(L^2(\Gamma))$.

For $h\in L^2(\Gamma)$, define its Cauchy transform by
\begin{equation}
\label{Cauchy}
(Ch)(\lambda)
=
\frac{1}{2\pi\ii}
\int_\Gamma\frac{h(s)}{s-\lambda}\,ds,
\qquad
\lambda\in\mathbb C\setminus\Gamma.
\end{equation}
The non-tangential boundary values of $Ch$ from the $+$- and $-$-sides
of $\Gamma$ exist for almost every point of $\Gamma$ and are denoted by
$C_+h$ and $C_-h$, respectively. The operators
\[
C_\pm:L^2(\Gamma)\longrightarrow L^2(\Gamma)
\]
are bounded and satisfy the Sokhotski--Plemelj identity
\begin{equation}
\label{eq:Plemelj-L2}
C_+-C_-=\mathds{1}.
\end{equation}

Let
\[
w^\pm\in L^2(\Gamma)\cap L^\infty(\Gamma).
\]
The associated Beals--Coifman operator is defined by
\begin{equation}
\label{C-w}
C_wf=C_+(fw^-)+C_-(fw^+).
\end{equation}
It is a bounded operator on $L^2(\Gamma)$. Moreover,
\begin{subequations}
\label{C_w_ess}
\begin{align}
\label{C_w_ess_1}
\|C_w\|_{\mathcal B(L^2(\Gamma))}
&\leq
C\max\left\{
\|w^+\|_{L^\infty(\Gamma)},
\|w^-\|_{L^\infty(\Gamma)}
\right\},\\
\label{C_w_ess_2}
\|C_wh\|_{L^2(\Gamma)}
&\leq
C\|h\|_{L^\infty(\Gamma)}
\max\left\{
\|w^+\|_{L^2(\Gamma)},
\|w^-\|_{L^2(\Gamma)}
\right\},
\end{align}
\end{subequations}
where one may take
\[
C=
2\max\left\{
\|C_+\|_{\mathcal B(L^2(\Gamma))},
\|C_-\|_{\mathcal B(L^2(\Gamma))}
\right\}.
\]

\subsection{Equivalence with a singular integral equation}

Suppose that the jump matrix admits a pointwise factorization
\begin{equation}
\label{eq:v-factorization}
v=(v^-)^{-1}v^+,
\end{equation}
and set
\begin{equation}
\label{eq:w-factorization}
w^+=v^+-I,
\qquad
w^-=I-v^-.
\end{equation}
The singular integral equation associated with the factorization
\eqref{eq:v-factorization} is
\begin{equation}
\label{mu_int_eq}
\mu-I=C_w\mu,
\qquad
\mu\in I+L^2(\Gamma).
\end{equation}

The following result is the $p=2$ case of the standard
Beals--Coifman correspondence; see \cites{L18,ZH89}.

\begin{lemma}[RH problem and singular integral equation]
\label{Lem 5.2}
Let
\[
v^\pm:\Gamma\longrightarrow GL(n,\mathbb C)
\]
satisfy
\[
v^\pm,\ (v^\pm)^{-1}
\in I+L^2(\Gamma)\cap L^\infty(\Gamma).
\]
Let $v$, $w^+$, and $w^-$ be defined by
\eqref{eq:v-factorization}--\eqref{eq:w-factorization}.

If
\[
m\in I+\dot E^2(\mathbb C\setminus\Gamma)
\]
solves the $L^2$ RH problem determined by $(\Gamma,v)$, then
\[
\mu
=
m_+(v^+)^{-1}
=
m_-(v^-)^{-1}
\in I+L^2(\Gamma)
\]
solves \eqref{mu_int_eq}.

Conversely, if $\mu\in I+L^2(\Gamma)$ solves
\eqref{mu_int_eq}, then
\begin{equation}
\label{eq:m-from-mu}
m(\lambda)
=
I+C\!\left(\mu(w^++w^-)\right)(\lambda)
\end{equation}
belongs to
\[
I+\dot E^2(\mathbb C\setminus\Gamma)
\]
and solves the $L^2$ RH problem determined by $(\Gamma,v)$.
\end{lemma}

\begin{remark}
\label{rem:equivalence-no-continuity}
No continuity of $v$, $v^\pm$, or $w^\pm$ is required in
Lemma~\ref{Lem 5.2}. In particular, the equivalence between the RH
problem and \eqref{mu_int_eq} remains valid when the jump matrix is
piecewise continuous and has finitely many discontinuities. The values
assigned to the jump matrix at the discontinuity points themselves are
irrelevant, since both the jump condition and the singular integral
equation are understood almost everywhere.
\end{remark}

\subsection{Fredholm property: continuous and piecewise-continuous jumps}

The following result applies when the factors in
\eqref{eq:v-factorization} are continuous; see \cite{L18}.

\begin{lemma}[Continuous-factor case]
\label{Lem 5.3}
Let
\[
v^\pm:\Gamma\longrightarrow GL(n,\mathbb C)
\]
satisfy
\[
v^\pm,\ (v^\pm)^{-1}
\in I+L^2(\Gamma)\cap L^\infty(\Gamma),
\qquad
v^\pm\in C(\Gamma).
\]
Let $v$, $w^+$, and $w^-$ be defined by
\eqref{eq:v-factorization}--\eqref{eq:w-factorization}. Then:

\begin{enumerate}[(1)]
\item
The operator
\[
\mathds{1}-C_w:L^2(\Gamma)\longrightarrow L^2(\Gamma)
\]
is Fredholm.

\item
If $w^+$ and $w^-$ are nilpotent matrix functions, then
\[
\operatorname{Ind}(\mathds{1}-C_w)=0.
\]
In this case, the following statements are equivalent:
\begin{enumerate}[(a)]
\item
$\mathds{1}-C_w:L^2(\Gamma)\to L^2(\Gamma)$ is bijective;

\item
the $L^2$ RH problem determined by $(\Gamma,v)$ has a unique solution;

\item
the homogeneous $L^2$ RH problem determined by $(\Gamma,v)$ has only
the zero solution;

\item
$\mathds{1}-C_w:L^2(\Gamma)\to L^2(\Gamma)$ is injective.
\end{enumerate}
\end{enumerate}
\end{lemma}

\begin{remark}[Piecewise-continuous jumps]
\label{rem:piecewise-continuous-jump}
The continuity assumption in Lemma~\ref{Lem 5.3} cannot simply be
discarded. If $v$ or its factors have jump discontinuities, the
Fredholm property is not automatic, even when
\[
v,\ v^{-1}\in L^\infty(\Gamma)
\qquad\text{and}\qquad
v-I\in L^2(\Gamma).
\]
In this situation one uses the local-symbol theory for singular
integral operators with piecewise-continuous coefficients; see
\cites{GK71,K03,GK73,CG81}.

For clarity, suppose that $\Gamma=\mathbb R$ is oriented from left to
right and that $v$ is piecewise continuous with finitely many
discontinuity points. Thus, $v$ is continuous on each open interval
between consecutive discontinuity points and possesses finite
one-sided limits at every discontinuity point. Let $\lambda_0$ be such a point, and assume that
the one-sided limits
\[
v_\ell(\lambda_0)
=
\lim_{\lambda\uparrow\lambda_0}v(\lambda),
\qquad
v_r(\lambda_0)
=
\lim_{\lambda\downarrow\lambda_0}v(\lambda)
\]
exist and are invertible. For the standard right-RH singular integral
operator
\begin{equation}
\label{eq:Av}
\mathcal A_v h
=
C_+h-(C_-h)v,
\end{equation}
the local $L^2$ symbol at $\lambda_0$ is represented by
\begin{equation}
\label{eq:local-symbol}
\mathcal S_{\lambda_0}(s)
=
(1-s)v_\ell(\lambda_0)+s v_r(\lambda_0),
\qquad 0\leq s\leq1.
\end{equation}
The local Fredholm condition is
\begin{equation}
\label{eq:local-Fredholm-condition}
\det\mathcal S_{\lambda_0}(s)\neq0,
\qquad 0\leq s\leq1.
\end{equation}
Equivalently,
\begin{equation}
\label{eq:spectral-local-condition}
\operatorname{spec}
\left(
v_\ell(\lambda_0)^{-1}v_r(\lambda_0)
\right)
\cap(-\infty,0]
=
\varnothing.
\end{equation}

Thus, a finite discontinuity of the jump matrix is allowed, but the
mismatch between its two one-sided limits must satisfy
\eqref{eq:local-Fredholm-condition}. For a piecewise-continuous, pointwise invertible matrix jump on the
real axis, invertibility of the complete local symbol at every point,
together with the corresponding condition at infinity, implies
Fredholmness of the associated singular integral operator; see
\cites{GK71, GK73}.

The Fredholm index is determined by the complete symbol. In
particular, the identity
\[
\det v(\lambda)=1
\]
does not, by itself, imply that the index is zero in the presence of
discontinuities. The index must either be computed from the complete
symbol or shown to vanish by a homotopy that remains inside the
Fredholm set.
\end{remark}

\subsection{Uniqueness}

\begin{lemma}[Uniqueness]
\label{lem:L2-uniqueness}
Suppose that
\[
v:\Gamma\longrightarrow GL(2,\mathbb C)
\]
satisfies
\[
\det v(\lambda)=1
\]
for almost every $\lambda\in\Gamma$. If the $L^2$ RH problem determined
by $(\Gamma,v)$ has a solution, then this solution is unique and
satisfies
\[
\det m(\lambda)\equiv1.
\]
\end{lemma}

\begin{proof}
Let $m\in I+\dot E^2(\mathbb C\setminus\Gamma)$ be a solution. Since
$\det v=1$ almost everywhere,
\[
(\det m)_+=(\det m)_-
\]
almost everywhere on $\Gamma$. Moreover,
\[
\det m-1\in\dot E^1(\mathbb C\setminus\Gamma).
\]
The Cauchy representation theorem therefore gives
\[
\det m-1
=
C\big((\det m)_+-(\det m)_-\big)
=
0.
\]
Thus $\det m\equiv1$, and
\[
m^{-1}
=
\begin{pmatrix}
m_{22}&-m_{12}\\
-m_{21}&m_{11}
\end{pmatrix}
\in I+\dot E^2(\mathbb C\setminus\Gamma).
\]

If $\widetilde m$ is a second solution, then
$\widetilde m m^{-1}$ has no jump across $\Gamma$, belongs to
$I+\dot E^1(\mathbb C\setminus\Gamma)$, and tends to $I$ at infinity.
The generalized Liouville theorem yields
\[
\widetilde m m^{-1}\equiv I,
\]
and hence $\widetilde m\equiv m$.
\end{proof}

\subsection{Local behavior of $L^2$ Cauchy transforms}

The following elementary localization estimate will be used to control
the behavior of RH solutions at the transition points.

\begin{lemma}[Local growth of an $L^2$ Cauchy transform]
\label{lem:local-Cauchy-growth}
Let $\Gamma\subset\mathbb R$ be measurable, let
$h\in L^2(\Gamma)$, and let $\lambda_0\in\mathbb R$. Suppose that
$\lambda\to\lambda_0$ inside a fixed non-tangential approach region:
there exists $c>0$ such that
\begin{equation}
\label{eq:nontangential-condition}
|\Im\lambda|
\geq
c|\lambda-\lambda_0|.
\end{equation}
Then
\begin{equation}
\label{eq:local-Cauchy-growth}
(Ch)(\lambda)
=
o\!\left(
|\lambda-\lambda_0|^{-1/2}
\right),
\qquad
\lambda\to\lambda_0.
\end{equation}
\end{lemma}

\begin{proof}
Extend $h$ by zero to a function
$\widetilde h\in L^2(\mathbb R)$. Then $Ch=C\widetilde h$ on
$\mathbb C\setminus\mathbb R$, so it is enough to consider the case
$\Gamma=\mathbb R$.

Write
\[
\lambda=x+\ii y,
\qquad y\neq0.
\]
The Cauchy--Schwarz inequality gives
\begin{align}
|(Ch)(\lambda)|
&\leq
\frac{1}{2\pi}
\int_{\mathbb R}
\frac{|h(s)|}{|s-\lambda|}\,ds
\nonumber\\
&\leq
\frac{1}{2\pi}
\|h\|_{L^2(\mathbb R)}
\left(
\int_{\mathbb R}
\frac{ds}{|s-\lambda|^2}
\right)^{1/2}.
\label{eq:Cauchy-CS}
\end{align}
Since
\[
|s-\lambda|^2=(s-x)^2+y^2,
\]
we have
\begin{equation}
\label{eq:Cauchy-kernel-integral}
\int_{\mathbb R}
\frac{ds}{|s-\lambda|^2}
=
\frac{\pi}{|y|}.
\end{equation}
Therefore
\begin{equation}
\label{eq:Cauchy-distance-estimate}
|(Ch)(\lambda)|
\leq
\frac{1}{2\sqrt\pi}
\|h\|_{L^2(\mathbb R)}
|\Im\lambda|^{-1/2}.
\end{equation}

We now prove the little-$o$ estimate. Fix $\varepsilon>0$. By the
absolute continuity of the $L^2$ norm, there exists $\delta>0$ such
that
\begin{equation}
\label{eq:local-L2-smallness}
\left\|
h\mathds{1}_{(\lambda_0-\delta,\lambda_0+\delta)}
\right\|_{L^2(\mathbb R)}
<
\varepsilon.
\end{equation}
Write
\[
h=h_1+h_2,
\]
where
\[
h_1
=
h\mathds{1}_{(\lambda_0-\delta,\lambda_0+\delta)},
\qquad
h_2
=
h\mathds{1}_{\mathbb R
\setminus(\lambda_0-\delta,\lambda_0+\delta)}.
\]

Applying \eqref{eq:Cauchy-distance-estimate} to $h_1$ and using
\eqref{eq:nontangential-condition}, we obtain
\begin{equation}
\label{eq:near-part-estimate}
|(Ch_1)(\lambda)|
\leq
\frac{\varepsilon}{2\sqrt{\pi c}}
|\lambda-\lambda_0|^{-1/2}.
\end{equation}

For the second term, suppose that
\[
|\lambda-\lambda_0|<\frac{\delta}{2}.
\]
For every $s\in\operatorname{supp}h_2$,
\[
|s-\lambda|
\geq
|s-\lambda_0|-|\lambda-\lambda_0|
\geq
\frac12|s-\lambda_0|.
\]
Consequently,
\begin{align}
|(Ch_2)(\lambda)|
&\leq
\frac{1}{\pi}
\int_{|s-\lambda_0|\geq\delta}
\frac{|h(s)|}{|s-\lambda_0|}\,ds
\nonumber\\
&\leq
\frac{1}{\pi}
\|h\|_{L^2(\mathbb R)}
\left(
\int_{|s-\lambda_0|\geq\delta}
\frac{ds}{|s-\lambda_0|^2}
\right)^{1/2}
\nonumber\\
&=
\frac{\sqrt2}{\pi\sqrt\delta}
\|h\|_{L^2(\mathbb R)}.
\label{eq:far-part-bounded}
\end{align}
This bound is independent of $\lambda$, and therefore
\begin{equation}
\label{eq:far-part-vanishing}
|\lambda-\lambda_0|^{1/2}
|(Ch_2)(\lambda)|
\longrightarrow0
\end{equation}
as $\lambda\to\lambda_0$.

Combining \eqref{eq:near-part-estimate} and
\eqref{eq:far-part-vanishing}, we obtain
\[
\limsup_{\lambda\to\lambda_0}
|\lambda-\lambda_0|^{1/2}|(Ch)(\lambda)|
\leq
\frac{\varepsilon}{2\sqrt{\pi c}}.
\]
Since $\varepsilon>0$ is arbitrary, the desired result follows.
\end{proof}

\begin{remark}
\label{rem:transition-point}
The point $\lambda_0$ need not be an endpoint of $\Gamma$. The proof
uses only the assumptions
\[
\lambda_0\in\mathbb R,
\qquad
h\in L^2(\Gamma).
\]
In particular, Lemma~\ref{lem:local-Cauchy-growth} applies at the
branch points
\[
\lambda_0=\pm A_1^{-1},\qquad
\lambda_0=\pm A_2^{-1},
\]
and, more generally, at every point of the real axis at which the
formula defining the jump matrix changes.
\end{remark}

\begin{remark}
\label{rem:Cauchy-o-application}
By construction,
\[
\hat N(y,t,\lambda)-I
=
C\big(\mu(w^++w^-)\big)(\lambda),
\]
where
\[
\mu(w^++w^-)\in L^2(\Gamma).
\]
By Proposition~\ref{Prop 5.6}, the same type of representation holds
for the derivatives:
\[
\hat N_y
=
C\big(
\mu_y(w^++w^-)
+
\mu(w^++w^-)_y
\big)
\]
and
\[
\hat N_t
=
C\big(
\mu_t(w^++w^-)
+
\mu(w^++w^-)_t
\big),
\]
with densities belonging to $L^2(\Gamma)$.

Since the limits in Propositions~\ref{prop-y} and \ref{prop-t} are
taken non-tangentially from $\mathbb C^+$ or $\mathbb C^-$,
Lemma~\ref{lem:local-Cauchy-growth} gives
\begin{equation}
\label{eq:N-o}
\hat N(y,t,\lambda),\quad
\hat N_y(y,t,\lambda),\quad
\hat N_t(y,t,\lambda)
=
o\!\left(
|\lambda\mp A_j^{-1}|^{-1/2}
\right),
\qquad
j=1,2.
\end{equation}
\end{remark}



\bibliography{references}

\end{document}